%% file: rgg.tex
\documentclass[12pt]{article}
\usepackage{amssymb}
\usepackage{amsfonts}
\usepackage{epsf}
\usepackage{latexsym}
\usepackage{amscd}
\usepackage{eucal}
\usepackage{a4wide}
\usepackage{makeidx}
\usepackage[dvips]{graphicx}
\usepackage[english]{babel}
\usepackage{amsmath}
\usepackage{color}
\newtheorem{theorem}{Theorem}[section]
\newtheorem{lemma}[theorem]{Lemma} 
\newtheorem{corollary}[theorem]{Corollary} 
\newtheorem{proposition}[theorem]{Proposition} 
\newenvironment{proof}
{ \noindent {\bf Proof:} }{ \hfill $\blacksquare$ \bigskip }
\newenvironment{pf}[1]%
{ \noindent {\bf Proof #1:} }{ \hfill $\blacksquare$ \bigskip }
\newcommand{\noproof}[0]{\hfill $\blacksquare$}
\newcounter{ExNr}
\newcommand{\norm}[1]{\ensuremath{\|#1\|}}

\newcommand{\tel}[1]{\ensuremath{ |#1| }}

\newcommand{\eR}[0]{\ensuremath{ \mathbb R}}

\newcommand{\eN}[0]{\ensuremath{ \mathbb N}}
\newcommand{\Zed}[0]{\ensuremath{ \mathbb Z}}

\newcommand{\Pee}[0]{\ensuremath{{\mathbb P}}}
\newcommand{\Ee}[0]{\ensuremath{{\mathbb E}}}

\newcommand{\downto}{\ensuremath{\downarrow}}
\newcommand{\Fcal}[0]{\ensuremath{{\cal F}}}
\DeclareMathOperator{\clo}{cl}

\DeclareMathOperator{\dd}{d}
\DeclareMathOperator{\Bi}{Bi}
\DeclareMathOperator{\Po}{Po}
\DeclareMathOperator{\mult}{mult}
\DeclareMathOperator{\diam}{diam}
\DeclareMathOperator{\vol}{vol}
\DeclareMathOperator{\supp}{supp}

\DeclareMathOperator{\conv}{conv}
\newcommand{\volhalfBdelta}[0]{\ensuremath{\frac{\vol(B)}{2^d\delta}}}
\newcommand{\volhalfB}[0]{\ensuremath{\frac{\vol(B)}{2^d}}}
\newcommand{\volB}[0]{\ensuremath{\vol(B)}}
\newcommand{\fcol}[0]{\ensuremath{f_{\chi}}}
\newcommand{\fcli}[0]{\ensuremath{f_{\omega}}}
\newcommand{\frat}[0]{\ensuremath{f_{\chi/\omega}}}
\newcommand{\kvsparse}[0]{\ensuremath{{\textstyle\left\lfloor|\frac{\ln n}{\ln(nr^d)}| + \frac12\right\rfloor}}}
\newcommand{\m}[1]{}
  \newcommand{\eps}{\varepsilon}

\begin{document}

%
%
%

 \title{On the chromatic number of random geometric graphs}
  \author{Colin McDiarmid\footnotemark[1] and Tobias M\"uller\footnotemark[2]\\
  University of Oxford and Centrum Wiskunde \& Informatica}

\footnotetext[1]{Department of Statistics, 1 South Parks Road, Oxford, OX1 3TG, United Kingdom.  Email address: \tt{cmcd@stats.ox.ac.uk}}
\footnotetext[2]{Centrum Wiskunde \& Informatica, P.O.~Box 94079, 1090 GB Amsterdam, The Netherlands. Email address: {\tt tobias@cwi.nl}. The research in this paper was conducted while this author was a research student at the University of Oxford. He was partially supported by Bekker-la-Bastide fonds, Hendrik Muller's Vaderlandsch fonds, EPSRC, Oxford University Department of Statistics and Prins Bernhard Cultuurfonds.}

\maketitle

\begin{abstract}
Given independent random points $X_1,\dots,X_n\in\eR^d$ with common probability distribution $\nu$, and a positive distance  $r=r(n)>0$, we construct a random geometric graph $G_n$ with vertex set $\{1,\dots,n\}$ where distinct $i$ and $j$ are adjacent when $\norm{X_i-X_j}\leq r$.   Here $\norm{.}$ may be any norm on $\eR^d$, and $\nu$ may be any probability distribution on $\eR^d$  with a bounded density function.   We consider the chromatic number $\chi(G_n)$ of $G_n$ and its relation to the clique number $\omega(G_n)$ as $n \to \infty$.  Both McDiarmid~\cite{cmcdplane} and Penrose~\cite{penroseboek}  considered the range of $r$ when $r \ll (\frac{\ln n}{n})^{1/d}$ and the range when 
$r \gg (\frac{\ln n}{n})^{1/d}$, and their results showed a dramatic difference between these two cases.
Here we sharpen and extend the earlier results, and in particular we consider the `phase change' range 
when $r \sim (\frac{t\ln n}{n})^{1/d}$ with $t>0$ a fixed constant.  
Both~\cite{cmcdplane} and~\cite{penroseboek} asked for the behaviour of the chromatic number in this range.  We determine constants $c(t)$ such that $\frac{\chi(G_n)}{nr^d}\to c(t)$ almost surely.
Further, we find a ``sharp threshold'' (except for less interesting choices of the norm when the unit ball tiles $d$-space): 
there is a constant $t_0>0$ such that  if $t \leq t_0$ then $\frac{\chi(G_n)}{\omega(G_n)}$ tends to 1 almost surely,  but if $t > t_0$ then $\frac{\chi(G_n)}{\omega(G_n)}$ tends to a limit $>1$ almost surely.
\end{abstract}


\section{Introduction and statement of results}
\label{sec.intro}

  In this section, after giving some initial definitions including that of the random geometric graphs $G_n$,
  we present our main results on colouring these graphs. We then introduce more notation and
  definitions so that we can specify explicit limits; we discuss fractional chromatic numbers and
  `generalised scan statistics', which will be key tools in our proofs;
  and we sketch the overall plan of the proofs.


\subsection{Some definitions and notation: $G_n$, $\delta$, $\sigma$, $\chi$ and $\omega$} \label{subsec.rgg}

  To set the stage, we fix a positive integer $d$ and a norm $\norm{.}$ on $\eR^d$.
  Given points $x_1,\ldots,x_n$ in $\eR^d$ and a `threshold distance' $r>0$,
  the corresponding {\em geometric graph} $G(x_1,\ldots,x_n;r)$ has vertices $1,\ldots,n$ and distinct vertices
  $i$ and $j$ are adjacent when $\norm{x_i-x_j} \leq r$.
  It will be convenient sometimes (when the $x_i$ are distinct) to consider $G(x_1,\ldots,x_n;r)$
  as having as vertices the points $x_i$ rather than their indices $i$.
  
  Now introduce a probability distribution $\nu$ with bounded density function,
  and consider a sequence $X_1, X_2,\dots$ of independent random variables each with this distribution.
  Also we need a sequence $r=(r(1),r(2),..)$ of positive real numbers such
  that $r(n) \to 0$ as $n \to \infty$.  The {\em random geometric graph} $G_n$
  is the geometric graph $G(X_1,\ldots,X_n;r(n))$ corresponding to $X_1, \dots, X_n$ and $r(n)$. 
  Observe that almost surely the $X_i$ are all distinct and we never have $\norm{X_i - X_j} = r(n)$;
  thus it would not really matter if we took the vertex set of $G_n$ as $\{X_1, \dots, X_n\}$ and said that
  vertices $X_i$ and $X_j$ (for $i\neq j$) are adjacent when $\norm{X_i - X_j} < r(n)$.

  The distance $r=r(n)$ plays a role similar to that of the edge-probability $p(n)$ for Erd\H{o}s-Renyi
  random graphs $G(n,p)$.  Depending on the choice of $r(n)$, qualitatively different types of behaviour
  can be observed.
  The various cases are best described in terms of the quantity $n r^d$,
  which scales with the average degree of the graph
  (for a precise result see appendix A of the paper~\cite{twopoint} by the second author).
  We will often refer to the case when $nr^d / \ln n \to 0$ as the {\em sparse case}, and
  the case when $nr^d / \ln n \to \infty$ as the {\em dense case}.

Throughout this paper, we shall use the terminology {\em almost surely} (or a.s.) in the standard sense from probability
theory. That is, a.s.~means ``with probability one'', and if
$Z_1, Z_2, \dots$ are random variables and $c \in\eR$ a constant then 
$Z_n \to c$ a.s. means that $\Pee( Z_n \to c ) = 1$.
The notation $f(n) \ll g(n)$ means the same as $f(n)=o(g(n))$, and the notation
$f(n)\sim g(n)$ means that $f(n)/g(n) \to 1$ as $n\to\infty$.
So in particular, if $Z_1, Z_2, \dots$ are random variables and $(k_n)_n$ is a sequence of numbers
then $Z_n \sim k_n$ a.s.~means that $\Pee( Z_n/k_n \to 1 ) = 1$.

Before we can state our first result we will need some further notation and definitions.
\m{first $\delta$ then $\sigma$?}
We use $\sigma$ to denote the essential supremum of
the probability density function $f$ of $\nu$, that is
\[
\sigma := \sup\{ t : \vol( \{ x: f(x)> t\} ) > 0 \}.
\]
Here and in the rest of the paper $\vol(.)$ denotes the $d$-dimensional volume (Lebesgue measure).
We call $\sigma$ the {\em maximum density} of $\nu$.

  We also need to define the `packing density' for the given norm $\norm{.}$.
  Informally this is the greatest proportion of $\eR^d$ that can be filled
  with disjoint translates of the unit ball $B:=\{x \in\eR^d : \norm{x}<1\}$.
For $K>0$ let $N(K)$ be the maximum cardinality of a collection of pairwise disjoint translates of $B$ with centers in $(0,K)^d$.
The (translational) {\em packing density} $\delta$ (of the unit ball $B$ with respect to $\norm{.}$) may be defined as
\[
  \delta := \lim_{K\to \infty} \frac{N(K)\vol(B)}{K^d}.
\]
  This limit always exists, and $0<\delta \leq 1$. 
In the case of the Euclidean norm in $\eR^2$ we have $\delta = \frac{\pi}{2 \sqrt{3}} \approx 0.907$.
For an overview
  of results on packing see for example the books~\cite{pachboek} of Pach and
  Agarwal or~\cite{rogersboek} of Rogers.

  Recall that a $k$-{\em colouring} of a graph $G$ is a map
  $f:V(G)\to\{1,\dots,k\}$ such that $f(v)\neq f(w)$ whenever
  $vw\in E(G)$, and that the {\em chromatic number} $\chi(G)$ is the least
  $k$ for which $G$ admits a $k$-colouring. Also a {\em clique} in $G$ is
  a set of vertices which are pairwise adjacent,
  and the {\em clique number} $\omega(G)$ is the largest cardinality of a clique
  %
  (note that $\omega(G)$ is a trivial lower bound for $\chi(G)$).
  In this paper we are interested mainly in the behaviour of
  the chromatic number $\chi(G_n)$, and its relation to the clique
  number $\omega(G_n)$, of the random geometric graph $G_n$ as $n$
  grows large.
  \medskip
  
  \noindent
  {\bf Assumptions and notation for the random geometric graph $G_n$.} \hspace{.1in}
  For convenience of reference, we collect our standard assumptions.
  We assume throughout that we are given a fixed positive integer $d$ and a fixed norm $\norm{.}$ on $\eR^d$,
  with packing density $\delta$.
  Also $\nu$ is a probability distribution with finite maximum density $\sigma$;
  $X_1, X_2, \dots$ are independent random variables each with this distribution;
  $r=(r(1),r(2),\ldots)$ is a sequence of positive reals such that $r(n) \to 0$ as $n \to \infty$;
  and for $n=1,2,\ldots,$ the {\em random geometric graph} $G_n$ is the geometric graph
  $G(X_1,\ldots,X_n;r(n))$.
  %


\subsection{Main results}  
\label{subsec.main-results}

Our first theorem gives quite a full picture of the different behaviours of 
the chromatic number of the random geometric graph depending on the choice of the sequence $r$.

\begin{theorem}\label{thm.tobiaschi}
For the random geometric graph $G_n$ as in Section~\ref{subsec.rgg}, the following hold.
\begin{enumerate}
 \item\label{itm:tobiaschi.vsparse}
Suppose that $nr^d \leq n^{-\alpha}$ for some fixed $\alpha > 0$. 
Then

\[
\hspace{-1cm} 
\Pee\left( \chi( G_n ) \in \left\{ \kvsparse , \kvsparse+1 \right\} 
\text{ for all but finitely many } n \right) = 1.
\]

\item\label{itm:tobiaschi.sparse}
Suppose that $n^{-\eps} \ll nr^d \ll \ln n$ for all $\eps > 0$. Then

\[ 
\chi(G_n) \sim \ln n / \ln{\Big(}\frac{\ln n}{nr^d}{\Big)} 
\;\; \text{ a.s. } 
\]

\item\label{itm:tobiaschi.intermediate}
Suppose that $\frac{\sigma n r^d}{\ln n} \to t \in (0,\infty)$. Then 

\[ 
\chi(G_n) \sim \fcol(t) \cdot \sigma n r^d \;\; \text{ a.s.,}
\]

\noindent
where $\fcol$ is given by~\eqref{eqn.chi-defn} below. It depends only on $d$ and $\norm{.}$, is continuous 
and non-increasing; and satisfies $\fcol(t) \to \volhalfBdelta$  as $t \to \infty$, and
  $\fcol(t) \to \infty$ as $t \downto 0$.
\item\label{itm:tobiaschi.dense}
Suppose that $nr^d \gg \ln n$ (but still $r\to0$). Then

\[ \chi(G_n) \sim \volhalfBdelta \cdot \sigma n r^d \;\; \text{ a.s. } \]
\end{enumerate}

\end{theorem}

Part
~\ref{itm:tobiaschi.sparse} can basically already be found in
Penrose's book~\cite{penroseboek}, and 
also in~\cite{cmcdplane} for the case 
when $d=2$ and $\norm{.}$ is the euclidean norm.
We have however settled the minor technical issue of improving the type of convergence
in~\ref{itm:tobiaschi.sparse} from convergence in probability to
almost sure convergence, which was mentioned as an open 
problem in~\cite{penroseboek} and~\cite{cmcdplane}.

In part~\ref{itm:tobiaschi.dense}
  we obtain an improvement over a result in Penrose~\cite{penroseboek}, 
  where an almost sure upper bound for $\limsup \frac{\chi(G_n)}{\sigma nr^d}$ of $\frac{\vol(B)}{2^d\delta_L}$
  and an almost sure lower bound for $\liminf \frac{\chi(G_n)}{\sigma nr^d}$ of $\frac{\vol(B)}{2^d\delta}$ are given.
  Here $\delta_L$ is the {\em lattice packing density} of $B$ (that is, the proportion of $\eR^d$ that can be filled 
with disjoint translates of $B$ whose centres are the integer linear combinations of some basis for $\eR^d$).
  The paper~\cite{cmcdplane} by the first author considers only the Euclidean norm in the plane, where
  $\delta$ and $\delta_L$ coincide.
  However, let us note that in general dimension the question of whether $\delta = \delta_L$ is open,
  even for the Euclidean norm, and it may well be that $\delta > \delta_L$ for some dimensions $d$.

Part~\ref{itm:tobiaschi.intermediate} settles the open problem of a ``law of large numbers for $\chi$'' in this regime 
posed in~\cite{penroseboek}. 
Both Penrose~\cite{penroseboek} and the first author~\cite{cmcdplane} have studied
the chromatic number in the cases when $nr^d \ll \ln n$ and $nr^d \gg \ln n$, but 
not much was known previously about the behaviour of the chromatic number in the 
``intermediate'' regime when $nr^d = \Theta(\ln n)$.
The limiting constant $\fcol(t)$ is given explicitly by~\eqref{eqn.chi-defn} below: since
it requires an involved sequence of definitions we defer the precise definition until then.

For comparison to Theorem~\ref{thm.tobiaschi} above, we shall now give a result on the clique number in the same flavour as Theorem~\ref{thm.tobiaschi}.
The results listed in Theorem~\ref{thm.tobiasomega} below were already shown by Penrose~\cite{penroseboek}; except that there was an extra assumption that the probability density function of $\nu$ has compact support, and in the regime considered in part~\ref{itm:tobiasomega.sparse} only convergence in probability was shown and we have added some detail in the regime considered in part~\ref{itm:tobiasomega.vsparse}.

\begin{theorem}\label{thm.tobiasomega}
For the random geometric graph $G_n$ as in Section~\ref{subsec.rgg}, the following hold.
\begin{enumerate}
 \item\label{itm:tobiasomega.vsparse}
Suppose that $nr^d \leq n^{-\alpha}$ for some fixed $\alpha > 0$. 
Then

\[ 
\hspace{-1cm}
\Pee\left( \omega( G_n ) \in \left\{ \kvsparse, \kvsparse + 1 \right\} 
\text{ for all but finitely many } n \right) = 1.
\]
  
\item\label{itm:tobiasomega.sparse}
Suppose that $n^{-\eps} \ll nr^d \ll \ln n$ for all $\eps > 0$. Then

\[ \omega(G_n) \sim \ln n / \ln{\Big(}\frac{\ln n}{nr^d}{\Big)} \;\; \text{ a.s. } \]

\item\label{itm:tobiasomega.intermediate}
Suppose that $\frac{\sigma n r^d}{\ln n} \to t \in (0,\infty)$. Then 

\[ \omega(G_n) \sim \fcli(t) \cdot \sigma n r^d \;\; \text{ a.s.}
\]

\noindent
Here $\fcli(t)$ is the unique $f\geq \volhalfB$  that solves
$H(f 2^d / \vol(B) ) = \frac{2^d}{t \vol(B)}$, where $H(x) := x \ln x - x + 1$ for $x>0$.
The function $\fcli$ is continuous and strictly decreasing; and satisfies $\fcli(t) \to \volhalfB$  as $t \to \infty$, and $\fcli(t) \to \infty$ as $t \downto 0$.
\item\label{itm:tobiasomega.dense}
Suppose that $nr^d \gg \ln n$ (but still $r\to0$). Then

\[ \omega(G_n) \sim \volhalfB \cdot \sigma n r^d \;\; \text{ a.s. } \]
\end{enumerate}

\end{theorem}

Comparing Theorem~\ref{thm.tobiaschi} to Theorem~\ref{thm.tobiasomega}, it is rather striking that 
when $n r^d \ll \ln n$ the chromatic and clique number have the same behaviour, while when 
$n r^d \gg \ln n$ they differ by a multiplicative factor (provided $\delta < 1$).
Clearly there is some switch in behaviour when $nr^d = \Theta( \ln n )$.
The following result shows a threshold phenomenon occurs (provided $\delta<1$).

\begin{theorem}\label{thm.frat}
The following hold for $\fcol, \fcli$.
\begin{enumerate}
\item\label{itm:thm.frat.i} If $\delta = 1$ then $\fcol(t) = \fcli(t)$ for all $t \in (0,\infty)$.
\item\label{itm:thm.frat.ii} 
If $\delta < 1$ then there exists a constant $0 < t_0 < \infty$ such that
$\fcol(t) = \fcli(t)$ for all $t \leq t_0$ and the ratio $\fcol(t)/\fcli(t)$ is continuous and strictly increasing for 
$t \geq t_0$.
\end{enumerate}
\end{theorem}

\noindent
For $t \in (0,\infty)$ let us write

\begin{equation}\label{eq:fratdef}
 \frat(t) := \frac{\fcol(t)}{\fcli(t)}.
\end{equation}

\noindent
Observe that, by the properties of $\fcol$ and $\fcli$ listed in Theorem~\ref{thm.tobiaschi} 
and Theorem~\ref{thm.tobiasomega}, the function $\frat$ depends only on the choice of $d$ 
and $\norm{.}$, it is continuous on $(0,\infty)$ and 

\begin{equation}\label{eq:fratlim}
 \lim_{t\to\infty} \frat(t) = \frac{1}{\delta}.
\end{equation}

%
%
\noindent
By adding a small amount of work to the proofs of Theorem~\ref{thm.tobiaschi} and Theorem~\ref{thm.tobiasomega}
we will also show that

\begin{theorem}\label{thm.tobiasfrat}
For the random geometric graph $G_n$ as in Section~\ref{subsec.rgg}, 
with $r=r(n)$ an arbitrary sequence that tends to 0, the following holds.
Set $t(n) := \sigma n r^d / \ln n$ then

\[ \frac{\chi(G_n)}{\omega(G_n)} \sim \frat( t(n) ) \;\;\; \text{ a.s. } \]

\end{theorem}

\noindent
So in  particular, when $\delta<1$, we see a sharp threshold at $r_0 := (\frac{t_0\ln n}{\sigma n})^\frac{1}{d}$: 

\begin{corollary}\label{cor:fratthreshold}
Consider the random geometric graph $G_n$ as in Section~\ref{subsec.rgg}.
Suppose $\delta < 1$ and let $t_0$ be the constant in part~\ref{itm:thm.frat.ii} of Theorem~\ref{thm.frat}. 
If we set $r_0 := (\frac{t_0\ln n}{\sigma n})^\frac{1}{d}$ then the following hold
\begin{enumerate}
\item If $\limsup_{n\to\infty} \frac{r}{r_0}\leq 1$ then 

\[ \frac{\chi(G_n)}{\omega(G_n)}\to 1 \;\;\; \text{ a.s. } \]

\item If $\liminf_{n\to\infty} \frac{r}{r_0} > 1$ then 

\[ \liminf_{n\to\infty} \frac{\chi(G_n)}{\omega(G_n)} > 1 \;\;\; \text{ a.s. } \]

\end{enumerate}
\end{corollary}

In the course of proving Theorem~\ref{thm.tobiasfrat} we shall 
prove the following result, which may be of independent interest.
It shows that for very small $r$ the clique number and chromatic number are not only concentrated on the
same two consecutive integers (as shown by parts~\ref{itm:tobiaschi.vsparse} of Theorems~\ref{thm.tobiaschi} and~\ref{thm.tobiasomega}), but 
in fact the chromatic number and clique number are equal.

\begin{proposition} 
  \label{thm.verysparse}
  For the random geometric graph $G_n$ as in Section~\ref{subsec.rgg},
  if $n r^d \leq n^{-\alpha}$ for some fixed $\alpha > 0$ then
\[
  \Pee( \chi(G_n) = \omega(G_n) \text{ for all but finitely many } n) = 1.
\]%

\end{proposition}


\subsection{The weighted integral $\xi$ and explicit limits}
\label{subsec.explicit}
\m{or say `scaling value' and `scaled integral'?}
\m{or `weighting factor'?}

  For $x>0$ let $H(x)= \int_1^{x} \ln y \ dy =  x\ln x-x+1$ as in Theorem~\ref{thm.tobiasomega}. 
Observe that $H(1)=0$ and that
  the function $H(x)$ is strictly increasing for $x>1$.
  
  Now let $\varphi$ be a fixed non-negative, bounded, measurable function with
  $0<\int_{\eR^d}\varphi(x){\dd}x<\infty$.  For $s \geq 0$ let

\[ f(s) := \int_{\eR^d} H(e^{s\varphi(x)}){\dd}x .\]

\noindent
  It is routine to check that $f(0)=0$, that $f(s)$ is continuous
  and strictly increasing in $s$, that $f(s) < \infty$ for all $s \geq 0$ and that $f(s) \to \infty$ as
  $s \to \infty$.
  For $0<t<\infty$ the {\em weighting value} $s(\varphi,t)$ is defined
  to be the unique nonnegative solution $s$ to $f(s)=1/t$.
  Observe that the function $s(\varphi,t)$ is strictly decreasing in $t$, $s(\varphi,t) \to \infty$ as $t \to 0$ and
  $s(\varphi,t) \to 0$ as $t \to \infty$.
  
  We define $\xi(\varphi, t)$ for $0<t<\infty$ by
\begin{equation}\label{eq:seq}
  \xi(\varphi, t) := \int_{\eR^d}\varphi(x) e^{s\varphi(x)} dx
\end{equation}
  where $s$ is the weighting value $s(\varphi,t)$.
  So far, $\xi(\varphi, t)$ is strictly decreasing in $t$.
It is convenient also to set

\begin{equation}\label{eq:seqinf}
\xi(\varphi,\infty):=\int \varphi.
\end{equation}

\noindent
  Further, if $\int\varphi = 0$ (in which case $\varphi=0$ almost everywhere)
  then set $\xi(\varphi, t) = 0$ for each $t \in [0,\infty]$, and if $\int \varphi = \infty$
  then set $s(\varphi,t)=0$ and $\xi(\varphi,t) = \infty$ for each $t \in [0,\infty]$.
  We call $\xi(\varphi, t)$ the {\em weighted integral} :
  note that the function $\xi$ depends only on the dimension $d$. 

  We may identify $\xi(\varphi, t)$ when $\varphi = 1_W$ for a measurable set $W\subseteq\eR^d$ with
  $0<\vol(W)<\infty$.
  For each $w>0$ define $c(w,t)$ for $t \in (0,\infty]$ as follows:
  set $c(w,\infty)=w$, and for $0<t<\infty$ let $c(w,t)$ be the unique solution $x \geq w$ to
  $H(\frac{x}{w})= \frac{1}{wt}$.
  Observe that $c(w,t)$ is continuous and strictly decreasing in $t$ for $t \in (0,\infty)$; and 
that 

\begin{equation}\label{eq:cwtinf}
c(w,t) \to \infty \;\; \text{ as } t \to 0,
\end{equation}

\noindent
  and 

\begin{equation}\label{eq:cwtnul}
c(w,t) \to w \;\; \text{ as } t \to \infty.
\end{equation}  

\noindent
  For $0<t<\infty$ we have $\xi(1_W,t) = e^s \vol(W)$ where $s$ is such that $H(e^s) \vol(W) =\frac{1}{t}$; and so
\begin{equation}\label{eq:xiindicatorrewrite}
  \xi(1_W,t) = c(\vol(W),t)
\end{equation}

\noindent
  and this holds also for $t=\infty$ since then both sides equal $\vol(W)$.

  Let $B(x;\rho)$ denote the ball $\{ y : \norm{x-y}<\rho \}$, so that $B=B(0;1)$. 
  Let us set 

\begin{equation}\label{eq:phinuldef}
\varphi_0 := 1_{B(0;\frac12)}.
\end{equation}

\noindent
Observe that the function $\fcli$ in Theorem~\ref{thm.tobiasomega} satisfies
  
\begin{equation} \label{eqn.omega-def}
  \fcli(t) = c(\vol(B(0;1/2)),t) = \xi(\varphi_0, t).
\end{equation}

So in particular, by the properties of $c(w,t)$ listed above, $\fcli$ satisfies the properties 
listed in part~\ref{itm:tobiasomega.intermediate} of Theorem~\ref{thm.tobiasomega}.

  Call a set $S \subseteq \eR^d$ {\em well-spread} if $\norm{v-w} > 1$
  for all $v\neq w\in S$; and let $\cal S$ denote the collection of
  all such sets.
  Finally here we call a nonnegative, measurable function $\varphi:\eR^d\to\eR$
  {\em (dual) feasible} if it satisfies the condition that
  $\sum_{v\in S} \varphi(v) \leq 1$ for each set $S\in{\cal S}$.
  For example, $\varphi_0$ 
  is feasible.  Denote the set of all feasible functions by~$\cal F$.
  We may now define the real-valued function $\fcol$ on $(0,\infty)$ in Theorem~\ref{thm.tobiaschi},
  by setting

\begin{equation} \label{eqn.chi-defn}
  \fcol(t) := \sup_{\varphi \in {\cal F}} \xi(\varphi, t) \;\;\;\; \mbox{ for }  \; 0<t<\infty.
\end{equation}

\noindent
  Finally note that the real-valued function $\frat$ on $(0,\infty)$ defined by~\eqref{eq:fratdef} 
 satisfies
\begin{equation} \label{eqn.ratio-defn}
  \frat(t) = \frac{ \sup_{\varphi \in {\cal F}} \ \xi(\varphi, t) }{ \xi(\varphi_0, t) }
  \;\;\;\; \mbox{ for }  \; 0<t < \infty.
\end{equation}


\subsection{Fractional chromatic number}
\label{subsec.chif}

  We shall see that, in Theorem~\ref{thm.tobiaschi}, the same conclusion holds if we replace $\chi(G)$
  by the fractional chromatic number $\chi_f(G)$ and indeed this is the key to the proofs. 
  
  Recall that a {\em stable} or {\em independent} set in a graph $G$ is a set of vertices which are
  pairwise non-adjacent,  and the chromatic number $\chi(G)$ of $G$ corresponds to a natural integer 
linear program (ILP),
  expressing the fact that the chromatic number is the least number of stable sets needed to cover the vertices,
  as follows.  Let $A$ be the vertex-stable set incidence matrix of $G$, that is,
  the rows of $A$ are indexed by the vertices $v$, the columns are indexed by the stable sets $S$,
  and $(A)_{v,S} = 1$ if $v \in S$ and $(A)_{v,S} = 0$ otherwise.  Then $\chi(G)$ equals
 \begin{equation}\label{eq:ILP}
  \begin{array}{rl}
  \min & 1^T x \\
  \textrm{ subject to } &
   A x \geq 1, \\
  & x \geq 0, x \text{ integral}.
  \end{array}
  \end{equation}
  The {\em fractional chromatic number} $\chi_f(G)$ of a graph $G$ is the objective value of the LP-relaxation
  of~\eqref{eq:ILP} (that is, we drop the constraint that $x$ be integral).
  
  It is easy to see that always $\omega(G) \leq \chi_f(G) \leq \chi(G)$.
  In general both the ratios $\chi_f(G)/\omega(G)$ and $\chi(G)/ \chi_f(G)$
  \m{added: for example}
  can be arbitrarily large (see for example chapter 3 of Scheinerman and Ullman~\cite{fractionalgraphtheory}),
  but that is not the case for geometric graphs (for a given norm on $\eR^d$).
  For $\chi(G) \leq \Delta(G)+1$ for any graph $G$, as shown by a natural greedy colouring algorithm;
  and, since we may cover the unit ball $B$ with a finite number $k$
  of sets of diameter $<1$, for any geometric graph $G$ we have  $\Delta(G)+1 \leq k \omega(G)$,
  and so $\omega(G) \leq \chi_f(G) \leq \chi(G) \leq k \omega(G)$.
  (The {\em diameter} of a set $A \subseteq \eR^d$ is $\sup\{ \norm{x-y}: x,y\in A \}$.)
  
  For the random geometric graph $G_n$ the two quantities $\chi(G_n)$ and $\chi_f(G_n)$ are even closer.
  Our approach to proving Theorem~\ref{thm.tobiaschi} and the other results will naturally yield:
  
\begin{theorem} \label{thm.chif}
   For the random geometric graph $G_n$ as in Section~\ref{subsec.rgg}
  (with any distance function $r=r(n)=o(1)$),
  we have $\chi(G_n)/\chi_f(G_n) \to 1$ a.s. 
\end{theorem}


\subsection{Generalised scan statistics}
\label{subsec:genscan}

  For a set $V$ of points in $\eR^d$ and  a nonnegative function
  $\varphi:\eR^d\to\eR$ we define $M(V,\varphi)$ by:
\[
  M(V,\varphi) := \sup_{x\in\eR^d} \sum_{v\in V} \varphi(v-x).
\]
  This quantity plays a central role in our analysis, as do
  the random variables
\[
  M_\varphi = M_{\varphi}(n,r) 
  := \sup_{x\in\eR^d} \sum_{i=1}^n 
  \varphi\left(r^{-1} X_i -x \right) 
\]
  where we have scaled the $X_i$ by $r^{-1}$.  Thus if the points $X_1,\ldots,X_n$ are distinct,
  and $V=\{ r^{-1}X_1,\ldots,r^{-1}X_n\}$ then $M_{\varphi}(n,r)= M(V,\varphi)$.

  Feasible functions $\varphi$ correspond to feasible solutions
  to the dual of the LP for the fractional chromatic number.
  In the special case when $\varphi$ is the indicator function $1_W$
  of some set $W \subseteq \eR^d$, we denote $M_\varphi$ by $M_W$.
  We denote the number of indices $i \in \{1,\ldots,n\}$ such that $X_i \in W$
  by ${\cal N}(W) = {\cal N}_n(W)$; that is, ${\cal N}(W) = \sum_{i=1}^n 1_W(X_i)$.
  (We will often omit the argument or subscript $n$ for the sake of readability.)
  Notice that $M_W$ is the maximum number of points in any translate of $r W$; that is
  $M_W = \max_x {\cal N}( x+rW )$.
  The variable $M_W$ is a {\em scan statistic} (with respect to the
  scanning set $W$), see for example the book~\cite{glaznaus01} by Glaz, Naus and Wallenstein:
  we call $M_\varphi$ a {\em generalised scan statistic}.
  \smallskip

  We say that a set $W\subseteq\eR^d$ has a {\em small neighbourhood} if it has finite volume and
  $\lim_{\eps \to 0} \vol( W_\eps ) = \vol( W )$, 
  where $W_\eps = W + \eps B$.
  \m{insist that $\vol(W)<\infty$?}
  Then for sets $W$ with finite volume, $W$ has a small neighbourhood if and only if
  $W$ is bounded and $\vol(\clo(W))=\vol(W)$, where $\clo(.)$ denotes closure.
  In particular all compact sets and all bounded convex sets have small neighbourhoods
(and the choice of the norm $\norm{.}$ is not relevant).
  We say that a function $\varphi:\eR^d\to\eR$ is {\em tidy}
  if it is measurable, bounded, nonnegative, has bounded support and
  the sets $\{ x : \varphi(x) > a \}$ have small neighbourhoods for all $a>0$.

  The proofs of the above theorems rely heavily on the following limiting result
  concerning the generalised scan statistic $M_\varphi$ for a tidy function $\varphi$.

\begin{theorem}\label{thm.Mphi}
Let $\nu$ be a probability distribution on $\eR^d$ with finite maximum density $\sigma$;
  let $X_1,X_2,\ldots$ be independent random variables each with distribution $\nu$;
 and let $r=r(n)>0$ satisfy $r(n) \to 0$ as $n \to \infty$.
  If $\varphi$ is a tidy function and if $t(n):= \frac{\sigma n r^d}{\ln n}$ satisfies $\ \liminf_n t(n) >0$ then
\[
  \frac{M_{\varphi}}{\sigma nr^d}  \sim  \xi(\varphi, t(n)) \;\; \mbox{ a.s.}
\] 
\end{theorem}


\subsection{Plan of proofs}
\label{subsec:plan-proofs}

  In the next section, 
  Lemma~\ref{lem:xibasic} gives basic results on
  the weighted integral $\xi(\varphi,t)$ which we use throughout the paper.
  The following section, Section~\ref{sec:scan}, contains the proof of Theorem~\ref{thm.Mphi} on generalised scan statistics,
  and includes some more detailed lemmas on this topic which will be needed later.
  These two sections are quite technical: they could be skipped on a first reading, and referred back to as needed.
  
  In the short Section~\ref{sec.clique-proof} we give a quick proof of parts~\ref{itm:tobiasomega.intermediate}
and~\ref{itm:tobiasomega.dense} of Theorem~\ref{thm.tobiasomega} on $\omega(G_n)$,
  using Theorem~\ref{thm.Mphi} (together with Lemma~\ref{lem:xibasic}). 
  In Section~\ref{sec.col-proof}, we come to the heart of the proofs on $\chi(G_n)$:
  we first prove some deterministic results on $\chi$ and $\chi_f$ for geometric graphs,
  and on feasible functions $\varphi$ and the weighted integral $\xi$; and then we deduce
  parts~\ref{itm:tobiaschi.intermediate} and~\ref{itm:tobiaschi.dense} of Theorem~\ref{thm.tobiaschi} on $\chi(G_n)$
and Theorem~\ref{thm.frat} on $\chi(G_n)/\omega(G_n)$, using
  Theorem~\ref{thm.Mphi} (together with Lemma~\ref{lem:xibasic}).
  In the next section, Section~\ref{sec.x(t)}, we show that $\fcol, \fcli, \frat$ and $t_0$ have the properties claimed in
  Theorems~\ref{thm.tobiaschi},~\ref{thm.tobiasomega} and~\ref{thm.frat}.
  Here, as well as using Lemma~\ref{lem:xibasic} and Theorem~\ref{thm.Mphi},
  we bring in some detailed lemmas on generalised scan statistics from Section~\ref{sec:scan}.
  In Section~\ref{sec.rest-col-proofs} we complete our proofs, and finally we make some concluding remarks.


\section{The weighted integral - basic results}
\label{sec.scaling-function}

Here we collect some useful observations about the weighted integral $\xi(\varphi,t)$ defined
in Section~\ref{subsec.explicit}.
Throughout the paper,  we will usually omit the domain we are integrating over and 
simply write $\int\varphi$
  instead of $\int_{\eR^d}\varphi(x){\dd}x$. 
  All integrals in this paper are over $\eR^d$ (and wrt~the
  $d$-dimensional Lebesgue measure) unless explicitly stated
  otherwise. The following lemma lists a number of basic properties
  of $\xi(\varphi,t)$. We will make frequent use of these properties
  in the rest of the paper.

\begin{lemma}\label{lem:xibasic}
  Let $\varphi$ and $\psi$ be non-negative, bounded, integrable functions on $\eR^d$, 
  and let $t \in (0,\infty]$.
\begin{enumerate}

\item\label{itm:phipsi}
  If $\varphi \leq \psi$ then $\xi(\varphi, t) \leq \xi(\psi, t)$.

\item\label{itm:xiscalar}
  $\xi(\lambda\varphi, t) = \lambda \ \xi(\varphi,t)$ for any $\lambda> 0$.

\item\label{itm:xisum}
  $\xi( \varphi+\psi, t) \leq \xi(\varphi,t)+\xi(\psi,t)$.

\item\label{itm:philambda}
  For $0<\lambda<1$ let $\varphi_\lambda$ be given by
  $\varphi_\lambda(x) = \varphi(\lambda x)$.  Then
  $\xi(\varphi,t)\leq \xi(\varphi_\lambda, t) \leq \lambda^{-d} \xi(\varphi, t )$.

\item\label{itm:xitplush}
  $ \frac{t}{t+h} \ \xi(\varphi,t)\leq\xi(\varphi,t+h)\leq\xi(\varphi,t)$ for $0<t<\infty$ and $h>0$.

\item\label{itm:phiindicatorbiggest}
  If $\int \varphi 1_{\{\varphi\geq a\}} \leq \int\psi 1_{\{\psi\geq a\}}$
  for all $a$ then $\xi(\varphi,t) \leq \xi(\psi,t)$.

\item \label{itm:toinf}
  $\xi(\varphi, t) \to \int\varphi = \xi(\varphi, \infty)$ as $t \to \infty$.

\item \label{itm:phiseq}
  Let $\varphi_1, \varphi_2, \dots$ be non-negative, bounded, integrable functions on $\eR^d$,
  and suppose that $\varphi_n \to \varphi$ pointwise as $n \to \infty$,
  and $\varphi_n \leq \psi$ for all $n$.
  \m{stronger conditions than before}
  Then 
  $\xi(\varphi_n,t) \to \xi(\varphi,t)$ as $n \to \infty$.

\end{enumerate}
\end{lemma}

  The case $t=\infty$ is always trivial, so in the proofs we will only consider the case when $t<\infty$.
  On several occasions,
  \m{rephrase?} in the proof below and later, we will differentiate an integral over $x\in\eR^d$ with respect
  to a parameter $u$ and swap the order of integration. In all cases this can be justified by means of the
  fundamental theorem of calculus and Fubini's theorem\footnote{
Here we mean the following. If $g(x,u)$ denotes one of $\varphi(x)e^{u\varphi(x)}$ or $H(e^{u\varphi(x)})$ 
then
$\int_{\eR^d} g(x,u)-g(x,0){\dd}x =
\int_{\eR^d}\int_0^u g_2(x,w){\dd}w{\dd}x =
\int_0^u\int_{\eR^d}g_2(x,w){\dd}x{\dd}w$, where $g_2$ denotes the derivative of $g$
wrt.~the second argument, and we have used Fubini's theorem to switch the order of
integration. Now the fundamental theorem of calculus
shows that $\frac{\dd}{{\dd}u}\int_{\eR^d} g(x,u){\dd}x
= \frac{\dd}{{\dd}u}\int_{\eR^d} g(x,u)-g(x,0){\dd}x
= \frac{\dd}{{\dd}u}\int_0^u\int_{\eR^d}g_2(x,w){\dd}x{\dd}w =
\int_{\eR^d}g_2(x,u){\dd}x$.
}.
  A function is {\em simple} if it takes only finitely many values.
  We prove the parts of the lemma in a convenient order.
\medskip

\begin{pf}{of \ref{itm:toinf}}
  If $0< t \leq t'< \infty$ then $s(\varphi,t) \geq s(\varphi,t')$ so
  $\xi(\varphi,t) \geq \xi(\varphi,t')$.  Thus
  \ref{itm:toinf} follows from the monotone convergence theorem.
\end{pf}

\medskip

\begin{pf}{of \ref{itm:phipsi}}
If we differentiate the equation $t \int H( e^{s\varphi}) = 1$ wrt $t$ we find:
\[
  0 = \int H( e^{s\varphi}) + t\int s's\varphi^2 e^{s\varphi} =
  \frac{1}{t} + s's t\int \varphi^2e^{s\varphi},
\]
  which gives
\[
  s' = - \frac{1}{t^2 s \int \varphi^2 e^{s\varphi}}.
\]

\noindent
(That $s$ is differentiable can for instance be seen from the implicit function theorem\footnote{%
Set $F(t,s) := \int_{\eR^d}H(e^{s\varphi(x)}){\dd}x-\frac{1}{t}$. 
Observe that $\frac{\partial}{{\partial}s}F \neq 0$ for all $s>0$.
The implicit function theorem now gives that for every $t>0$ there 
are neighbourhoods $U$ of $t$ and $V$ of $s(\varphi,t)$ and a unique function
$g:U\to V$ such that $F(t',g(t'))= 0$ for all $t'\in U$; and moreover this 
$g$ is continuously differentiable.
Thus $s(\varphi,t')=g(t')$ for all $t'\in U$ and in particular $s(\varphi,t')$ is differentiable
at $t'=t$.
}.)
  Thus
\[
  \frac{\dd}{{\dd}t}\xi(\varphi, t)
  = \int s' \varphi^2 e^{s\varphi} = -\frac{1}{t^2 s}.
\]
  Now notice that $\varphi \leq \psi$ implies $s(\varphi, t) \geq s(\psi, t)$, so that
  for all $0<t<\infty$:
\[
\frac{\dd}{{\dd}t} \xi(\varphi, t) \geq \frac{\dd}{{\dd}t} \xi(\psi, t),
\]
  which implies that $\xi(\psi, t) - \xi(\varphi, t)$ is non-increasing.
  Finally, by \ref{itm:toinf} we have
\[
  \lim_{t\to\infty} \left( \xi(\psi, t)-\xi(\varphi,t) \right) = \int \psi - \int \varphi \geq 0,
\]
  so that $\xi(\psi,t) \geq \xi(\varphi, t)$ for all $t>0$.
\end{pf}

\medskip

\begin{pf}{of \ref{itm:xiscalar}}
  We must have $s(\lambda\varphi, t) = s(\varphi, t) / \lambda$
  as $\int H(e^{s(\varphi,t)\varphi}) = \int H(e^{s(\lambda\varphi,t)\lambda\varphi}) = \frac{1}{t}$.
  So indeed
  $\xi(\lambda\varphi,t) = \int \lambda \varphi e^{s(\lambda\varphi,t)\lambda\varphi}
  = \lambda \int \varphi e^{s(\varphi,t)\varphi} = \lambda\xi(\varphi,t)$.
\end{pf}

\medskip

\begin{pf}{of \ref{itm:phiseq}}
First, let $s\geq 0$ be fixed but otherwise arbitrary.
Observe that $H(e^{s\varphi_n})\leq H(e^{s\psi})$.
Since we also have $\int H(e^{s\psi}) < \infty$ (as observed in section~\ref{subsec.explicit}), the 
dominated convergence theorem gives

\[
  \lim_{n\to\infty} \int H(e^{s\varphi_n})
  = \int H(e^{s\varphi}).
\]

\noindent
  This shows that $\lim_{n\to\infty} s(\varphi_n,t) = s(\varphi, t)$.
  Thus, for all $\eps > 0$ and $n$ sufficiently large:
\[
  \int \varphi_n e^{(s(\varphi,t)-\eps)\varphi_n}
  \leq \xi(\varphi_n, t) 
  \leq \int \varphi_n e^{(s(\varphi,t)+\eps)\varphi_n}
  \leq \int \psi e^{(s(\varphi,t)+\eps)\psi}.
\]
  As $\int \psi e^{(s(\varphi,t)+\eps)\psi} < \infty$ the
  dominated convergence theorem also gives that
\[
\int \varphi e^{(s(\varphi,t)-\eps)\varphi}
\leq \liminf \xi(\varphi_n,t) \leq \limsup \xi( \varphi_n,t)
\leq \int \varphi e^{(s(\varphi,t)+\eps)\varphi}.
\]
  Two more applications of the the dominated convergence theorem now yield
\[
  \lim_{\eps\to 0}\int \varphi e^{(s(\varphi,t)-\eps)\varphi}
  = \lim_{\eps\to 0}\int \varphi e^{(s(\varphi,t)+\eps)\varphi}
  = \xi(\varphi,t),
\]
  giving the result.
\end{pf}

\medskip

\begin{pf}{of~\ref{itm:xisum}}
  By~\ref{itm:phiseq} it suffices to take $\varphi$ and  $\psi$ as simple functions.
  What is more, we can assume without loss of generality that $\supp(\varphi) = \supp(\psi)$.
  Hence we can assume that there are disjoint sets $A_i$ $i=1,\ldots,n$ such that
  $\varphi= \sum_i a_i 1_{A_i}$ and $\psi= \sum_i b_i 1_{A_i}$ where each $a_i>0$ and $b_i>0$.
  Let $s_0=s(\varphi+\psi,t)$ and define $\alpha_i$, $\beta_i$ and $\gamma_i$ by setting
  $1/\alpha_i = \int_{A_i} H(e^{s_0 a_i})$,
  $1/\beta_i = \int_{A_i} H(e^{s_0 b_i})$ and
  $1/\gamma_i = \int_{A_i} H(e^{s_0 (a_i+b_i})$.
  Clearly $0<\gamma_i < \alpha_i, \beta_i < \infty$.  Denote
  $(a_i+b_i) 1_{A_i}$ by $g_i$, and note that
  $s(g_i,\gamma_i) = s_0$.  Hence
\[
  \xi(\varphi+\psi,t) =
  \int (\varphi+\psi)e^{s_0 (\varphi+\psi)} =
  \sum_i \int g_i e^{s_0 g_i} =  \sum_i \xi(g_i,\gamma_i).
\]
  But now by~\ref{itm:xiscalar}
\begin{eqnarray*}
  \xi(\varphi+\psi,t) &=&
   \sum_i a_i \xi(1_{A_i}, \gamma_i) + \sum_i b_i \xi(1_{A_i}, \gamma_i)\\
  &\leq &
   \sum_i a_i \xi(1_{A_i}, \alpha_i) + \sum_i b_i \xi(1_{A_i}, \beta_i)\\
  &=& \xi(\varphi,t) + \xi(\psi,t),
\end{eqnarray*}
  completing the proof.
\end{pf}

\medskip

\begin{pf}{of~\ref{itm:phiindicatorbiggest}}
  It suffices to show that $s(\varphi, t) \geq s(\psi,t)$
  for all $t$, because then the argument given in the proof
  of part~\ref{itm:phipsi} will give the result.
  Therefore it also suffices to show that
\begin{equation}\label{eq:xiHeq0}
  \int_{\eR^d} H(e^{s\varphi(x)}) {\dd}x \leq \int_{\eR^d} H( e^{s\psi(x)})  {\dd}x
\end{equation}
  for all $s > 0$. 
  It is straightforward to check that $F(y):= \frac{\dd}{\dd y} \left( H(e^{sy})/y\right)\geq 0$
  for all $y \geq 0$.  But
\[
  H(e^{sy}) = y \int_0^y F(z) {\dd}z = \int_0^{\infty} F(z) y 1_{y \geq z} {\dd}z,
\]
  and so
\[
  \int_0^{\infty} H(e^{s \varphi(x)}){\dd}x  = \int_0^{\infty} \int_0^{\infty} F(z) \varphi(x) 1_{\varphi(x) \geq z} {\dd}z {\dd}x.
\]
  We may swap the order of integration since all the quantities involved are non-negative.
  Hence, using the fact that $F(z) \geq 0$ 
\begin{eqnarray*}
  \int_0^{\infty} H(e^{s \varphi(x)}){\dd}x  &=&
  \int_0^{\infty} F(z) \left(\int_0^{\infty}  \varphi(x) 1_{\varphi(x) \geq z}  {\dd}x \right) {\dd}z\\
  & \leq &
   \int_0^{\infty} F(z) \left(\int_0^{\infty}  \psi(x) 1_{\psi(x) \geq z}  {\dd}x \right) {\dd}z\\
   &=&
    \int_0^{\infty} H(e^{s \psi(x)}){\dd}x,
\end{eqnarray*}
  so that~(\ref{eq:xiHeq0}) holds, as desired.
\end{pf}

\medskip

\begin{pf}{of \ref{itm:philambda}}
  Note that the substitution $y = \lambda x$ gives that:
\[
  \frac{1}{t} = \int_{\eR^d} H(e^{s\varphi(\lambda x)}) dx =
  \lambda^{-d} \int_{\eR^d} H(e^{s\varphi(y)}) dy,
\]
  so that $s(\varphi_\lambda, t ) = s(\varphi, \lambda^{-d}t)$.
  Using the same substitution we get
\[
  \xi(\varphi_\lambda, t) = \lambda^{-d} \int_{\eR^d} \varphi(y) e^{s(\varphi, \lambda^{-d}t)\varphi(y)} dy
  = \lambda^{-d} \xi( \varphi, \lambda^{-d} t ).
\]
  The upper bound now follows from the fact that $\lambda^{-d} > 1$ and that
  $s(\varphi, t)$ is decreasing in~$t$.
  The lower bound follows from part~\ref{itm:phiindicatorbiggest},
  because $\int \varphi_\lambda 1_{\{\varphi_\lambda\geq a\}}
  = \lambda^{-d} \int \varphi 1_{\{\varphi\geq a\}}$ for all~$a$
  (again by the substitution $y=\lambda x$).
\end{pf}

\medskip

\begin{pf}{of~\ref{itm:xitplush}}
  Let $\lambda = \left(\frac{t}{t+h}\right)^\frac1d$.
  Then $0<\lambda<1$ and so by~\ref{itm:philambda} and its proof
 \[
  \xi(\varphi,t) \leq \xi(\varphi_{\lambda},t) = \lambda^{-d} \xi(\varphi, \lambda^{-d}t) = \frac{t+h}{t} \xi(\varphi,t+h).
 \]
   Also, we have already seen that $\xi(\varphi,t+h) \leq  \xi(\varphi,t)$.
\end{pf}


  We have now completed the proof of Lemma~\ref{lem:xibasic}.


\section{Proofs for generalised scan statistics}
\label{sec:scanstatistic} \label{sec:scan}

  In this section, after some preliminary results, we consider generalised scan statistics first in the 
sparse case, then the dense case,
  and finally prove Theorem~\ref{thm.Mphi} on the limiting behaviour of $M_{\varphi}$.

  Here is a rough sketch of the main idea of the proof of Theorem~\ref{thm.Mphi},
  when $\sigma n r^d \sim t \ln n$.  Consider first the special case $\varphi=1_W$.
  Since there are about $r^{-d}$ disjoint scaled translates $rW$ of $W$ where the probability density is close to $\sigma$,
  we see that $M_W$ behaves like the maximum of about $r^{-d}= n^{1+o(1)}$ independent copies of the number $Z$
  of points $X_1,\ldots,X_n$ in a fixed scaled translate $rW$ where the probability density is close to $\sigma$;
  and $Z$ is roughly $\Po(\lambda)$ where $\lambda= \vol(W)\sigma nr^d \sim \vol(W) t \ln n$.
  Large deviation estimates show that $M_W$ will be about $c \lambda$ where $c>1$ satisfies
  $\Pee(\Po(\lambda) \geq c \lambda) \sim 1/n$; and this happens when 
  $H(c) \lambda \sim \ln n$, that is $H(c)\vol(W) t \sim 1$.
  
  For the general case concerning $M_{\varphi}$ it suffices to consider a step function
  $\varphi= \sum_i a_i 1_{A_i}$ where the sets $A_i$ are disjoint.
  If $Z_i$ corresponds to $rA_i$ just as $Z$ corresponded to $rW$ above,
  then $M_{\varphi}$ behaves like the maximum of about $r^{-d}$ independent copies of $\sum_i a_i Z_i$,
  and we may proceed as above.


\subsection{Preliminaries}
\label{subsec.prelims}

  We need results on the maximum density $\sigma$ and disjoint `dense' sets.
  The first lemma is from M\"uller~\cite{twopoint}, and is a straightforward consequence of the fact
  that the set of points where $\nu$ has density at least $(1-\eps/2) \sigma$ has positive measure.

\begin{lemma}\label{prop:numax}
  Let $W \subseteq \eR^d$ be bounded with positive Lebesgue measure and fix $\eps >0$.
  Then there exist $\Omega(r^{-d})$ disjoint translates $x_1+rW, \dots, x_N+rW$
  of $rW$ with $\nu(x_i+rW)/\vol(rW) \geq (1-\eps)\sigma$ for all $i=1,\dots,N$.
\end{lemma}
  %
  This last result extends to:

\begin{lemma}
  Fix $\eps > 0$ and let $W \subseteq \eR^d$ be bounded and let
  $W_1, \dots, W_k$ be a partition of $W$ with $\vol(W_i) > 0$ for
  all $i$. Then there exist $\Omega(r^{-d})$ points
  $x_1,\dots, x_N$ such that the sets $x_i + rW_j$ are pairwise disjoint
  and $\nu(x_i+rW_j)/\vol(rW_j) > (1-\eps)\sigma$ for all
  $i=1,\dots,N$ and $j=1,\dots,k$. \label{cor:numax}
\end{lemma}

\begin{proof}
  Set $p_i := \frac{\vol(W_i)}{\vol(W)}$ and  $p := \min_i p_i$.
  By Lemma~\ref{prop:numax} there exist points $x_1, \dots, x_N$ with $N = \Omega( r^{-d} )$
  such that the sets $x_i + rW$ are disjoint and
  satisfy $\nu(x_i+rW) \geq (1-p \eps) \sigma\vol(rW)$.
  By construction the sets $x_i + rW_j$ are disjoint. We now observe that 
  $\nu(x_i+rW_j)$ must be $\geq (1-\eps) \sigma\vol(r W_j)$, because otherwise
\begin{eqnarray*}
  \nu(x_i+rW) &<&
  (1-p_j) \sigma\vol(r W) + (1-\eps) p_j  \sigma \vol(r W)\\
  &=&
  (1-p_j\eps) \sigma \vol(x_i+rW) \leq \nu(x_i+rW),
\end{eqnarray*}
  a contradiction.
\end{proof}

  For the proofs in this section we will also need some bounds on
  the binomial, Poisson and multinomial distributions. The following
  lemma is one of the so-called Chernoff-Hoeffding bounds.
  A proof can be found for example in Penrose~\cite{penroseboek}.

\begin{lemma}\label{lem:chernoffbinomial}\label{lem:chernoffpoisson}\label{lem:chernoff}\label{lem:chernoffH}
  Let $Z$ be either binomial or Poisson with $\mu := \Ee Z >0$. 
\begin{enumerate}
  \item If $k \geq \mu$ then $\Pee( Z \geq k ) \leq e^{ -\mu H(\frac{k}{\mu})
  }$.
  \item If $k \leq \mu$ then $\Pee( Z \leq k ) \leq e^{ -\mu H(\frac{k}{\mu}) }$.
\end{enumerate}
\end{lemma}

\noindent
  Often the upper bound given by Lemma~\ref{lem:chernoffpoisson} is quite close to the truth.
  The following lemma gives a lower bound on $\Pee( \Po(\mu) \geq k )$ which is sufficiently sharp
  for our purposes (see Penrose~\cite{penroseboek} for a proof).

\begin{lemma}\label{lem:PoAss}
  For $k, \mu > 0$ it holds that
  $\Pee( \Po(\mu) = k ) \geq \frac{e^{-\frac{1}{12k}}}{\sqrt{2\pi k}} e^{-\mu
  H(\frac{k}{\mu})}$.
\end{lemma}

  \noindent
  A direct corollary of lemmas~\ref{lem:chernoffH} and~\ref{lem:PoAss} is the following result:

\begin{lemma}\label{lem:chernoffH2}\label{cor:chernoffH2}
  For $\alpha > 1$ it holds that
  $\Pee( \Po(\mu) > \alpha\mu ) = e^{-\mu H(\alpha) + o(\mu)}$ as $\mu \to \infty$.
\end{lemma}

  \noindent
  Another bound on the binomial and Poisson that will be useful in the sequel is the following
  standard elementary result (see for example McDiarmid~\cite{cmcdplane}).

\begin{lemma} \label{lem:binineqe}\label{lem:poineqe}
  Let $Z$ be either binomial or Poisson and $k \geq \mu := \Ee Z$.  Then
\[ (\frac{\mu}{e k})^k \leq \Pee( Z \geq k ) \leq (\frac{e \mu}{k})^k . \]
\end{lemma}
  We will also need the following result from Mallows~\cite{mallows68} on the multinomial distribution:
\begin{lemma}  \label{lem:negassmult}
  Let $(Z_1, \dots, Z_m) \sim \mult( n; p_1,\dots, p_m )$. Then
\[  \Pee( Z_1 \leq k_1, \dots, Z_m \leq k_m ) \leq \Pi_{i=1}^m \Pee( Z_i \leq k_i ). \]
\end{lemma}

%
%
%
%
%
%
%
%
%
%
%
%
%
%
%


\subsection{Sparse case}
\label{subsec.sparse}

  Here we consider the behaviour of $M_W$ in the `very sparse' case, and then the `quite sparse' case.
  For suitable sets $W$ we find results on $M_W$ that do not depend on $W$.
  
  First we introduce a convenient piece of notation.
  If $A$ is an event then we say that $A$ holds almost surely (a.s.) if $\Pee( A ) = 1$,
  and if $A_1, A_2, \dots$ is a sequence of events then $\{ A_n \textrm{ almost always } \}$
  denotes the event that ``all but finitely many $A_n$ hold''.
  We will frequently deal with the situation in which
  $\Pee( A_n \textrm{ almost always } ) = 1$, which we shall denote by
  $A_n$ a.a.a.s. ($A_n$ almost always almost surely). We hope this is a convenient shorthand,
  which avoids clashes with the many different existing notations for
  $\Pee(A_n) = 1+o(1)$ (a.a., a.a.s., whp.) that are in use in the random graphs literature.
  The reader should observe that $A_n$ a.a.a.s. is a much stronger statement than $A_n$ a.a.s.
  Observe that the conclusion of Proposition~\ref{thm.verysparse} is that in the case considered there
  we have $\chi(G_n)=\omega(G_n)$ a.a.a.s.
  We will need the following lemma in order to prove Proposition~\ref{thm.verysparse}.

\begin{lemma} \label{lem:scanstatisticverysparse}
  Let $W \subseteq \eR^d$ be a measurable, bounded set with nonempty interior, and fix $k\in\eN$. 
 \begin{enumerate}
  \item If $n r^d \leq n^{-\alpha}$ with $\alpha > \frac{1}{k}$ then $M_W \leq k$ a.a.a.s.;
    \label{itm:scanverysparseub}
  \item If $n r^d \geq n^{-\beta}$ with $\beta < \frac{1}{k-1}$ then $M_W \geq k$ a.a.a.s.
    \label{itm:scanverysparselb}
 \end{enumerate}
\end{lemma}

\begin{pf}{ of part~\ref{itm:scanverysparseub}}
  We first remark that it suffices to prove the result for $W$ a ball.
This is because $W$ is bounded and hence we must have $W \subseteq B(0;R)$
for some $R > 0$ (and hence also $M_W \leq M_{B(0;R)}$, so that
$M_{B(0;R)} \leq k$ implies $M_W \leq k$).
Furthermore, since $W$ is a ball it is clear that $M_W$ is non-decreasing in $r$
and we may assume without loss of generality that $r$ is chosen such that
$n r^d = n^{-\alpha}$.
If some translate of $rW$ contains $k+1$ points, then some $X_i$ has at least
$k$ other points at distance $\leq 2Rr$. Hence
\[
  \Pee( M_W \geq k+1 ) \leq
  \Pee( \exists i: {\cal N}(B(X_i;2Rr) \geq k+1 ) \leq
  n \Pee( {\cal N}( B(X_1;2Rr ) ) \geq k+1 ).
\]
  Note that
\[ \begin{array}{rcl}
\Pee( {\cal N}( B(X_1;2Rr) ) \geq k+1 ) & \leq &
\Pee( \Bi( n, \sigma \volB 2^d R^d r^d ) \geq k )
\leq (\frac{e \sigma \volB 2^d R^d n r^d}{k})^{k} \\
& = & O( n^{-k\alpha}),
\end{array} \]
  where we have used Lemma~\ref{lem:binineqe}.
  As $\alpha > \frac{1}{k}$, we have $\alpha':=k\alpha-1>0$. We find
\[ \Pee( M_W \geq k+1 ) = O( n^{-\alpha'} ). \]
  Unfortunately this expression is not necessarily summable in $n$ so we cannot apply the
  Borel-Cantelli lemma directly. However, setting $K := \lceil \frac{1}{\alpha'}\rceil + 1$,
  we may conclude that
\[ \Pee( M_W(m^K,r(m^K)) \leq k \textrm{ for all but finitely many } m ) = 1, \]
  because $\sum_m (m^K)^{-\alpha'} < \infty$. We now claim that from this it can be deduced
  that $M_W \leq k$ a.a.a.s.  Note that
\[ \lim_{m\to\infty}\frac{r((m-1)^K)}{r(m^K)} = 1, \]
  because $n r^d = n^{-\alpha}$. Consequently
  $\gamma := \sup_m \frac{r((m-1)^K)}{r(m^K)} < \infty$. By the previous we may also conclude that
\[ \Pee( M_{\gamma W}(m^K,r(m^K)) \leq k \textrm{ for all but finitely many } m) = 1. \]
  Let $n,m$ be such that $(m-1)^K < n \leq m^K$.
  Note that for any $x \in \eR^d$ it holds that
  $x+r(n)W \subseteq x + \gamma r(m^K)W$ as $\gamma r(m^K) \geq r((m-1)^K) > r(n)$.
  In other words if $M_W(n) \geq k+1$ then also $M_{\gamma W}(m^K)\geq k+1$.
  Thus it follows that
\[ \Pee( M_W(n,r(n)) \leq k \textrm{ for all but finitely many } n ) = 1, \]
  as required.
\end{pf}

\begin{pf}{ of part~\ref{itm:scanverysparselb}}
  We may assume that $\beta>0$. Also, we may again assume that $W$ is a ball.
  This is because $W$ has non-empty interior and it must therefore
  contain some ball $B$, so that it suffices to show $M_B \geq k$ a.a.a.s.
  Again, by the fact that $M_W$ is non-decreasing in $r$ (when $W$ is a ball)
  we may assume wlog that $r$ is chosen such that $n r^d = n^{-\beta}$.
  By Lemma~\ref{prop:numax} we can find
  disjoint translates $W_1, \dots, W_N$ of $rW$ satisfying
  $\nu(W_i) \geq (1-\eps) \sigma \vol(W) r^d$ where $N = \Omega( r^{-d} )$.
  Now notice that the joint distribution of
  $({\cal N}(W_1), \dots, {\cal N}(W_N), {\cal N}(\eR^d\setminus\cup_iW_i) )$ is
  multinomial, so that we can apply Lemma~\ref{lem:negassmult} to see that
\[
  \Pee( M_W \leq k-1 ) \leq \Pee( {\cal N}(W_1) \leq k-1, \dots, {\cal N}(W_N) \leq k-1 ) \leq
  \Pi_{i=1}^N \Pee( {\cal N}(W_i) \leq k-1 ).
\]
  The (marginal) distribution of ${\cal N}(W_i)$ is $\Bi( n, \nu(W_i) )$, so that
  Lemma~\ref{lem:binineqe} tells us that
\[
  \Pee( {\cal N}(W_i) \geq k ) \geq (\frac{n (1-\eps) \sigma \vol(W) r^d}{e k})^k =c n^{-k\beta}.
\]
  Thus
\[ 
  \Pee( M_W \leq k-1 ) \leq (1-cn^{-k\beta})^N \leq \exp[ - c n^{-k\beta} N ].
\]
  As $n r^d = n^{-\beta}$ we have $r^{-d} = n^{1+\beta}$.
  As $\beta < \frac{1}{k-1}$ we also have $\beta' := 1+\beta - k\beta > 0$,
  so that $n^{-k\beta} N = \Omega( n^{\beta'} )$.  Thus
\[ \Pee( M_W \leq k-1 ) \leq \exp[-\Omega( n^{\beta'} ) ], \]
  which is summable in $n$. It follows from the
  Borel-Cantelli lemma that $M_W \geq k$ a.a.a.s. as required.
\end{pf}

  We now move on to consider the `quite sparse' case.
\begin{lemma}  \label{lem:scanstatisticsparse}
  Let $W \subseteq \eR^d$ be a measurable, bounded set with non-empty interior and fix $0< \eps < 1$.
  Then there exists a $\beta = \beta(W,\sigma,\eps) > 0$ such that if
  $n^{-\beta} \leq n r^d \leq \beta\ln n$ then
  
\[ (1-\eps) k(n) \leq M_W \leq (1+\eps)k(n) \quad \text{ a.a.a.s., } 
 \]

\noindent
with  $k(n) = \ln n / \ln(\frac{\ln n}{n r^d})$.
\end{lemma}

\begin{proof}
  As in the proof of the previous lemma we may again assume that $W$ is a ball.
  Set $k(n) := \ln n / \ln(\frac{\ln n}{nr^d})$.  Let us first consider the lower bound.
  Completely analogously to the proof of Lemma~\ref{lem:scanstatisticverysparse},
  part~\ref{itm:scanverysparselb}, we have
  \m{was $\Pee( M_W \leq $}
\begin{eqnarray}
  \Pee( M_W < (1-\eps)k )
 & \leq &
  (1-\Pee( \Bi(n, Cr^d )\geq (1-\eps)k))^{\Omega(r^{-d})} \nonumber \\
 & \leq &
  \exp[ - \Omega( r^{-d} (\frac{C n r^d}{e(1-\eps)k})^{(1-\eps)k} )] \nonumber \\
 & = &
  \exp[ - \Omega( r^{-d} \exp[ -(1-\eps)k(\ln(\frac{k}{nr^d})+D) ] ) ], \label{eq:Mkineq}
\end{eqnarray}
  with $C := (1-\eps)\sigma\vol(W), D := \ln(\frac{e(1-\eps)}{C})$.  By choice of $k$:
\begin{eqnarray}
  k ( \ln(\frac{k}{nr^d}) + D )
 & = &
  \frac{\ln n}{\ln(\frac{\ln n}{nr^d})}
  \left[ \ln(\frac{\ln n}{nr^d}) -
  \ln(\ln(\frac{\ln n}{nr^d}))+D\right] \nonumber \\
& = &
  \ln n \left[ 1 -\ln( \ln(\frac{\ln n}{nr^d}) )/\ln(\frac{\ln n}{nr^d})
  + D / \ln(\frac{\ln n}{nr^d}) \right].  \label{eq:klnn0}
\end{eqnarray}

\noindent
If $n^{-\beta}\leq nr^d\leq \beta\ln n$ then $\frac{\ln n}{nr^d} \geq \frac1\beta$.
Thus, if $\beta>0$ is chosen small enough then:

\begin{equation}\label{eq:klnn}
k ( \ln(\frac{k}{nr^d}) + D ) =
\left( 1 + \frac{D -\ln( \ln(\frac{\ln n}{nr^d}))}{
\ln(\frac{\ln n}{nr^d})} \right) \ln n
\leq \ln n.
\end{equation}

\noindent
Also note that
$r^{-d} \geq \frac{n}{\beta\ln n} = n^{1+o(1)}$.
Combining this with~\eqref{eq:Mkineq} and~\eqref{eq:klnn}, we get

\[ \begin{array}{rcl}
\Pee( M_W \leq (1-\eps)k )
& \leq &
\exp[ - \Omega( r^{-d} e^{-(1-\eps)\ln n}  ) ]
=
\exp[ - \Omega( r^{-d} n^{-1+\eps+o(1)} ) ]  \\
& \leq &
\exp[ - n^{\eps+o(1)} ].
\end{array} \]

\noindent
This last expression sums in $n$, so we may conclude that
$M_W \geq (1-\eps)k$ a.a.a.s. if $n^{-\beta} \leq nr^d \leq \beta\ln n$
for $\beta>0$ sufficiently small.


  Let us now shift attention to the upper bound. As in the proof of
  part~\ref{itm:scanverysparseub} of Lemma~\ref{lem:scanstatisticverysparse} the obvious upper bound on
  $\Pee( M_W \geq (1+\eps)k )$ does not sum in $n$. Unfortunately
  the trick we applied there does not seem to work here and we are
  forced to use a more elaborate method. For $s > 0$ let us set

\[ M(n,s) := \max_{x\in\eR^d}{\cal N}(x+sW), \quad
k(n,s) := \ln n/ \ln(\frac{\ln n}{n s^d}). \]

\noindent
  Note that $k(n,s)$ is non-decreasing in $n$ and $s$ and so is $M(n,s)$ (because $W$ is a ball).
  \m{was increasing}
The rough idea for the rest of the proof is as follows.
We fix a (large) constant $K$.
Given $n$ we appoximate $n$ by $m^K$, chosen to satisfy $(m-1)^K<n\leq m^K$, and
we approximate $r$ by $\tilde{s} \geq r$, which is one of
$O(\ln m)$ candidate values $s_1,\dots,s_{N(m)}$, in such a way that

\begin{enumerate}
\item[(a)] $M(m^K,\tilde{s}) \leq (1+\frac{\eps}{2})k(m^K,\tilde{s})$ a.a.a.s.
\item[(b)] $(1+\frac{\eps}{2})k(m^K,\tilde{s}) \leq (1+\eps)k(n,r)$;
\end{enumerate}

\noindent
Note that $M(m^K, \tilde{s}) \geq M(n,r(n))$ because
$m^K \geq n, \tilde{s}\geq r$ and $W$ is a ball, and that combining this with
items (a) and (b) will indeed show that $M(n,r) \leq (1+\eps)k$ a.a.a.s.
The reason we have chosen this setup is that if the constant $K$ is chosen sufficiently large we will
be able to use the Borel-Cantelli lemma to establish (a), making use of the fact that we are only
considering a subsequence of $\eN$ and $\tilde{s}$ is one of $O(\ln m)$ candidate values.

Let us pick $s_1(n) < s_2(n) < \dots$ such that
$k(n, s_i(n) ) = i$.
Let us denote by $A(n)$ the event

\[
A(n) :=
\{
M(n,s_i(n)) > (1+\frac{\eps}{2})i \textrm{ for some } 1 \leq i < I(n)
\},
\]

\noindent
with $I(n) := \ln n / \ln(\frac{1}{2\beta})$, the value of $k(n,s)$ corresponding to $ns^d = 2\beta\ln n$, where $\beta=\beta(\eps)$ is to be determined later
(note that $k(n,s) = i$ implies that $\ln(\frac{\ln n}{ns^d}) = \frac{\ln n}{i}$).
By computations done in the proof of \ref{itm:scanverysparseub} of
Lemma~\ref{lem:scanstatisticverysparse} we know that

\begin{equation}\label{eq:Mns}
\Pee( M(n,s) > (1+\frac{\eps}{2})k(n,s) ) \leq
n \left( \frac{C ns^d}{k(n,s)} \right)^{(1+\frac{\eps}{2})k(n,s)}
= n 
e^{
- (1+\frac{\eps}{2})k(n,s) (\ln(\frac{k(n,s)}{ns^d})+D)
}, 
\end{equation}

\noindent
for appropriately chosen constants $C, D$.
We may assume wlog that $D\leq 0$.
  By~\eqref{eq:klnn0} we have that
\[
  k(n,s) (\ln\left(\frac{k(n,s)}{ns^d}\right)+D) 
= 
\ln n (1 - 
\frac{\ln\ln(\frac{\ln n}{ns^d})}{\ln(\frac{\ln n}{ns^d})}
+\frac{D}{\ln(\frac{\ln n}{ns^d})}).
\]
  If $s_1\leq s\leq s_{\lfloor I\rfloor}$ then $\ln n / (ns^d) \geq \frac{1}{2\beta}$.
  Hence, by taking $\beta$ sufficiently small we can guarantee
  that for $s_1\leq s\leq s_{\lfloor I\rfloor}$:
  
\[
  (1+\frac{\eps}{2})(1 - 
\frac{\ln\ln(\frac{\ln n}{ns^d})}{\ln(\frac{\ln n}{ns^d})}
+\frac{D}{\ln(\frac{\ln n}{ns^d})}) 
\geq
  (1+\frac{\eps}{2})(1-\frac{\ln\ln(\frac{1}{2\beta})}{\ln(\frac{1}{2\beta})}+\frac{D}{\ln(\frac{1}{2\beta})}) 
\geq 
1+ \eps/3,
\]

\noindent
since $D\leq 0$.
  By~\eqref{eq:Mns} we have that for $s_1(n) \leq s \leq s_{\lfloor I(n)\rfloor}(n)$:

\[
  \Pee( M(n,s) \geq (1+\frac{\eps}{2})k(n,s) ) \leq n 
  e^{- (1+\frac{\eps}{2})k(n,s) (\ln(\frac{k(n,s)}{ns^d})+D)} \leq n^{-\eps/3+o(1)}.
\]
  It also follows that
\[
  \Pee( A(n) ) \leq I(n) n^{-\eps/3 +o(1)} = n^{-\eps/3 +o(1)}.
\]
  This last expression does not necessarily sum in $n$, but if we take $K$ such that
  $K \eps/3 > 1$ then we can apply Borel-Cantelli to deduce that (a) holds; that is
\[
  \Pee( A(m^K) \textrm{ holds for at most finitely many } m ) = 1.
\]

\noindent
Now let $n \in \eN$ be arbitrary and let the integer $m=m(n)$ be such that
$(m-1)^K < n \leq m^K$.
Let $i=i(n)$ be such that $s_i(m^K) \leq r(n) < s_{i+1}(m^K)$.
We first remark that if $n^{-\beta} \leq n r^d \leq \beta \ln n$ then
$(m^K)^{-\beta} \leq m^K r^d \leq (\frac{m}{m-1})^K\beta \ln(m^K)$, giving:

\[
\frac{1+o(1)}{\beta}
\leq
k(m^K, r )
\leq
-(1+o(1))\frac{\ln(m^K)}{\ln(\beta)}.
\]

\noindent
So for $n$ sufficiently large, we must have $\frac{1}{2\beta} < i < I(m^K)$.

To complete the proof we now aim to show that (for $n$ large enough),
If $ M(n,r(n)) \geq (1+\eps)k(n,r(n))$  then

\[
M(m^K,s_{i+1}(m^K))\geq (1+\frac{\eps}{2})k(m^K,s_{i+1}(m^K)) = (1+\frac{\eps}{2})(i+1) . 
\]

\noindent
It suffices to show that
$(1+\eps)k(n,r(n)) \geq (1+\frac{\eps}{2})(i+1) $, because
$M(n,r(n)) \leq M(m^K,s_{i+1}(m^K))$  (since $W$ is a ball).
Routine calculations show that
\[ 
k(n,r(n)) \geq  k((m-1)^K,s_i(m^K)) \sim k(m^K,s_i(m^K)) = i ;
\]
\noindent
and, assuming that $\beta \leq \eps/6$, for $n$ sufficiently large
\[  
(1+\eps)k(n,r(n)) \geq (1+\frac{\eps}{2})(1+2\beta)i \geq (1+\frac{\eps}{2})(i+1) 
\]
\noindent
 as required.
\end{proof}

   Lemma~\ref{lem:scanstatisticsparse} also allows us to deduce immediately the
  following corollary, which extends lemma 5.3 of the first author~\cite{cmcdplane}
and may be of independent interest.

\begin{lemma}  \label{lem:scanstatisticsparse-cor}
  Let $W \subseteq \eR^d$ be be a measurable, bounded set with non-empty interior.
  If for each fixed $\eps>0$ we have 
  $n^{-\eps} < n r^d < \eps \ln n$ for all sufficiently large $n$, then
\[
M_W \sim  \ln n / \ln(\frac{\ln n}{n r^d}) \;\; \text{ a.s.,}
\]

\end{lemma}

%
%




\subsection{Dense case}
\label{subsec.dense}

  The `dense' case is when $\sigma nr^d/\ln n$ is large.
  Let us sketch our approach to proving this result.  
We first show that it suffices to consider a simple function $\varphi$.  
Also we ``Poissonise''; that is, we consider random points $X_1,\ldots,X_N$ where $N$ has an appropriate Poisson distribution.
To prove the lower bound, we use Lemma~\ref{cor:numax} to see that there are points $x_i$ so that we can translate to 
regions with density near $\sigma$; then we lower bound $M_{\varphi}$ by a linear combination of independent 
Poisson random variables (indeed by independent copies of such linear combinations), and use 
Lemma~\ref{lem:chernoffH} to complete the proof.
The proof of the upper bound is more involved.  
We first find a suitable simple function $\varphi_{\eta} \geq \varphi$ and close to $\varphi$.  Consider a cell $C = [-R/2,R/2)^d$  containing the support of $\varphi_{\eta} $, and partition $\eR^d$ into the sets $\Gamma(x) = x+ rR \Zed^d$ of translates of the points $x \in rC$.  Let the random variable $U$ be uniformly distributed over $rC$ and let $M(U)$ be

\[ 
\max_{y \in \Gamma(U)}  \sum_{j=1}^N \varphi_{\eta} (\frac{X_j-y}{r}). 
\]

\noindent
Because of the way $\varphi_{\eta}$ was chosen, it suffices to upper bound $\Pee(M(U) \geq (1+\eps)k)$.  
To do this we show that for each possible $u$

\[
\Pee(M(u) \geq (1+\eps)k) = O(r^{-d}) \cdot  \Pee (Z \geq   (1+\eps)k)
\]

\noindent
where $Z$ is a linear combination of independent Poisson random variables, and finally we use Lemma~\ref{lem:chernoffH} 
to complete the proof.

\begin{lemma}\label{lem:MphilargeT}
  Let $\varphi$ be a tidy function.
  For every $\eps > 0$ there exists a $T = T(\varphi, \eps)$ such that if $\sigma nr^d \geq T\ln n$ then
\[
  (1-\eps) k \leq M_\varphi \leq (1+\eps) k \textrm{ a.a.a.s., }
\]
  where $k = \sigma nr^d\int \varphi$.
\end{lemma}

\begin{proof} 
  Let us first observe that it suffices to prove the result
  for $\varphi$ a simple function, because the functions $\varphi$ we are considering can be well approximated
  by the functions $\varphi_m^{\text{lower}}, \varphi_m^{\text{upper}}$ defined by:
\[
  \varphi_m^{\text{lower}} := \sum_{k=1}^{\lceil m\cdot\max\varphi\rceil}
  (\frac{k-1}{m}) 1_{\{\frac{k-1}{m} < \varphi \leq \frac{k}{m}\}},
  \quad
  \varphi_m^{\text{upper}} := \sum_{k=1}^{\lceil m\cdot\max\varphi\rceil}
  (\frac{k}{m}) 1_{\{\frac{k-1}{m} < \varphi \leq \frac{k}{m}\}},
\]
  Here we mean by ``well approximated'' that
  $\varphi_m^{\text{lower}} \leq \varphi \leq
  \varphi_m^{\text{upper}}$ for all $m$ and
\begin{equation}\label{eq:limmslargeT}
  \lim_{m\to\infty} \int\varphi_m^{\text{lower}} = \lim_{m\to\infty} \int\varphi_m^{\text{upper}}
  = \int \varphi.
\end{equation}
  \m{was $= \xi(\varphi,t)$}
  Observe that~\eqref{eq:limmslargeT} follows from the dominated
  convergence theorem (since $\varphi$ is bounded and has bounded
  support). Also observe that the sets
  $\{ \varphi_m^{\text{upper}} > a \} = \{\varphi>\frac{\lfloor am\rfloor}{m}\}$
  and
  $\{\varphi_m^{\text{lower}}>a\} = \{\varphi>\frac{\lceil am\rceil}{m}\}$
  have a small neighbourhood for all $a$.

  Clearly $M_{\varphi_m^{\text{lower}}} \leq M_{\varphi} \leq M_{\varphi_m^{\text{upper}}}$.
  Thus the result for non-simple functions will follow from the result for simple functions
  by taking $m$ such that
  $\int\varphi_m^{\text{upper}}-\int\varphi_m^{\text{lower}}<\frac{\eps}{3}\int\varphi$
  and setting $T := \max( T_1,T_2 )$, where $T_1 := T(\varphi_m^{\text{upper}}, \frac{\eps}{3})$
  is the value we get from the result for simple functions applied to $\varphi_m^{\text{upper}}$ with $\frac{\eps}{3}$ 
and
  $T_2 := T(\varphi_m^{\text{lower}},\frac{\eps}{3})$ is the value we get from the result for simple functions
  applied to $\varphi_m^{\text{lower}}$ with $\frac{\eps}{3}$.

  In the remainder of the proof we will always assume that
  $\varphi = \sum_{i=1}^m a_i 1_{A_i}$ is a simple function with the sets $A_i$ disjoint and bounded
  and that $\{\varphi>a\}$ has a small neighbourhood for all $a$.
  Let us set
\begin{equation}\label{eq:MphilargeTkdef}
  k =k(n) := \sigma nr^d\int\varphi.
\end{equation}
  It remains to show that
  $(1-\eps)k \leq M_\varphi \leq (1+\eps)k$ a.a.a.s., whenever
  $\sigma nr^d \geq T\ln n$ for some sufficiently large $T$.

\smallskip

\noindent
{\bf Proof of lower bound:}
  Let $N \sim \Po( (1-\frac{\eps}{100})n )$ be independent from $X_1, X_2, \dots$.
  It will be useful to consider $X_1, \dots, X_N$ (rather than $X_1, \dots, X_n$), because
  they constitute the points of a Poisson process with intensity function
  $(1-\frac{\eps}{100})n f$ (where $f$ is the probability density function of
  $\nu$), see for example Kingman~\cite{kingmanboek}.

  By Lemma~\ref{cor:numax} there are $\Omega( r^{-d} )$ points $x_1, \dots, x_K$ such that
  $\nu( x_i +rA_j ) \geq (1-\frac{\eps}{100}) \sigma \vol(A_j) r^d$
  for all $1\leq i\leq K, 1\leq j\leq m$ and the sets $x_i+rA_j$ are disjoint.
  For the current proof we will only need that $K\geq 1$, but the fact that
  $K = \Omega( r^{-d} )$ will be needed for the proof of Theorem~\ref{thm.Mphi},
  which proceeds along similar lines to the current proof.
  Let us set $M_{i} := \sum_{j=1}^{N} \varphi(\frac{X_j-x_i}{r})$, so that
\[
  M_i = a_1 {\cal N}_N(x_i+rA_1) + \dots + a_m {\cal N}_N(x_i+rA_m),
\]
  where ${\cal N}_N(B) =  \sum_{i=1}^{N} 1_{X_i \in B}$
  denotes the number of points of the Poisson process in $B$.
  Note that ${\cal N}_N(x_i+rA_j)$ is a Poisson random variable with mean
  at least
\[
  \mu_j := (1-\frac{\eps}{100})^2 \sigma \vol(A_j) nr^d.
\]
 Setting
\[
  M_\varphi' := \sup_{x\in \eR^d}\sum_{i=1}^N \varphi(\frac{X_i-x}{r}),
\]
 we have
\[
  \Pee( M_\varphi' \leq (1-\eps)k )  \leq
  \Pee( M_{1} \leq (1-\eps) k, \dots, M_{K} \leq (1-\eps) k )
  = \Pi_{i=1}^K \Pee( M_i \leq (1-\eps) k ),
\]
  where in the last equality we have used that distinct $M_i$
  depend on the points of a Poisson process in disjoint areas of
  $\eR^d$ and hence the $M_i$ are independent.
  If $Z = a_1 Z_1 + \dots + a_m Z_m$ with the $Z_j$ independent Poisson
  random variables satisfying $\Ee Z_j = \mu_j$ then $M_i$ stochastically
  dominates $Z$, so that:
\[
\Pee( M_\varphi \leq (1-\eps)k ) \leq
\Pee( M_\varphi' \leq (1-\eps)k ) + \Pee( N > n ) \leq
\Pee( Z \leq (1-\eps)k )^K + \Pee( N > n ),
\]
 and consequently, by Lemma~\ref{lem:chernoffpoisson}
\begin{equation} \label{eqn.probineq1}
  \Pee( M_\varphi \leq (1-\eps)k ) \leq
  \Pee( Z \leq (1-\eps)k )^K + e^{-\alpha n},
\end{equation}
  where
  $\alpha := (1-\frac{\eps}{100}) H(\frac{1}{1-\frac{\eps}{100}})$.
  But  $K = \Omega(r^{-d})$ is $\geq 1$ for $n$ sufficiently large, and then
  $ \Pee( M_\varphi \leq (1-\eps)k ) \leq \Pee( Z \leq (1-\eps)k ) + e^{-\alpha n}$.
  Further, using Lemma~\ref{lem:chernoffH},
\[
\begin{array}{rcl}
  \Pee( Z \leq (1-\eps)k ) 
 & \leq &
  \sum_{i=1}^m \Pee( Z_i \leq \frac{1-\eps}{(1-\frac{\eps}{100})^2} \mu_i ) \\
 & \leq &
  m \cdot \max_i \Pee( \Po(\mu_i) \leq \frac{1-\eps}{(1-\frac{\eps}{100})^2} \mu_i ) \\
 & \leq &
  m \cdot \exp[ - \min_i \mu_i H\left(\frac{1-\eps}{(1-\frac{\eps}{100})^2}\right)].
\end{array}
\]
  Now suppose that $T$ has been chosen in such a way that (and we
  may suppose this)
\[
  T \cdot (1-\frac{\eps}{100})^2 \cdot 
  \min_i \vol(A_i) \cdot H\left( \frac{1-\eps}{(1-\frac{\eps}{100})^2} \right)  \geq 2.
\]
  It follows that
  $\sum_n \Pee( M_\varphi < (1-\eps) k ) \leq m \sum_n n^{-2} + \sum_n e^{-\alpha n} < \infty$,
  which concludes the proof of the lower bound.

\smallskip

\noindent
{\bf Proof of upper bound:}
  We may assume wlog~that $a_1>a_2>\dots>a_m>0$.
  Recall that $A_\eta$ denotes $A+B(0;\eta) = \cup_{a\in A} B(a;\eta)$.
  For $\eta>0$ let $\varphi_\eta$ be defined by
\[
  \varphi_\eta(x) := \left\{ \begin{array}{rl}
  a_i & \textrm{ if } x \in (A_i)_\eta \setminus \bigcup_{j<i} (A_j)_\eta, \\
  0 & \textrm{ if } x \not\in (A_i)_\eta \textrm{ for all } 1 \leq i \leq m.
  \end{array} \right.,
\]
  and let $\eta$ be chosen such that
  $\int \varphi_\eta \leq (1+\frac{\eps}{100})\int\varphi$.
  This can be done, since $\varphi$ is tidy and so the sets $A_i$ have small neighbourhoods.
  Thus we can choose $\eta$ so that
\begin{equation} \label{eqn.smallvol}
  \vol( (A_i)_\eta \setminus \bigcup_{j<i} (A_j)_\eta )
  \leq (1+\frac{\eps}{100}) \vol(A_i)
\end{equation}
  for all $i$, and then we also have $\int \varphi_\eta = \sum_i
a_i\vol( (A_i)_\eta \setminus \bigcup_{j<i} (A_j)_\eta ) \leq
(1+\frac{\eps}{100}) \sum_i a_i\vol(A_i) =
(1+\frac{\eps}{100})\int\varphi$. Clearly $\varphi_\eta(x) \geq
\varphi(x)$ for all $x$ giving $M_{\varphi_\eta} \geq M_\varphi$.

Similarly to what we did for the lower bound, let
$N \sim \Po((1+\frac{\eps}{100})n)$
be independent of the $X_i$ and set
\[
  M_{\varphi}' := \max_{x\in\eR^d} \sum_{j=1}^{N} \varphi(\frac{X_j-x}{r}).
\]
We have
\begin{equation}\label{eq:MphiMphidashUB}
\Pee( M_\varphi > (1+\eps)k ) \leq
\Pee( M_\varphi' > (1+\eps)k) + \Pee( N < n ) \leq
\Pee( M_\varphi' > (1+\eps)k) + e^{-\alpha n},
\end{equation}
for some $\alpha > 0$ (where we have used the Lemma~\ref{lem:chernoffpoisson}
).
Again the points $X_1,\dots, X_N$ are the points of a Poisson process, this time with
intensity function $(1+\frac{\eps}{100})n f$.

Let $R > 0$ be a fixed constant such that the support of $\varphi_\eta$ is contained in
$[\frac{-R}{2}, \frac{R}{2})^d$ ($R$ exists because we assumed the $A_i$
are bounded).
Let $U$ be uniform on $[0, r R)^d$ and
let $\Gamma(U)$ be the random set of points
$U + rR\Zed^d$ ($= \{ U+rRz:z\in\Zed^d\}$).
For $x \in \eR^d$ let $M_x$ be the random variable given by
$\sum_{j=1}^{N} \varphi_\eta( \frac{X_j-x}{r} )$.
Let us define
\[
  M(U) := \max_{z \in \Gamma(U)} M_z.
\]
If $\norm{p-q}\leq \eta r$ then $\varphi_\eta(\frac{x-p}{r})\geq\varphi(\frac{x-q}{r})$ for all
$x$ by definition of $\varphi_\eta$.
For any $q \in \eR^d$, the probability that some point of $\Gamma(U)$ lies in
$B(q;\eta r)$ equals
\[
  \Pee( \Gamma(U) \cap B(q;\eta r) \neq \emptyset ) =
  \frac{\volB \eta^d}{R^d}.
\]
(We may assume wlog~that $R$ is much larger than $\eta$.)
Because $\sum_{j=1}^N \varphi(\frac{X_j-x}{r}) \leq \sum_{j=1}^N \varphi_\eta(\frac{X_j-y}{r})$
whenever $\norm{x-y}<\eta r$, this gives the following inequality:
\[
\Pee( M(U) \geq (1+\eps)k | M_{\varphi}' \geq (1+\eps)k ) \geq
\frac{\volB \eta^d}{R^d}.
\]
We find:
\begin{equation}\label{eq:MphiMU}
\Pee( M_{\varphi}' \geq (1+\eps)k) \leq
\frac{R^d}{\volB \eta^d} \Pee( M(U) \geq (1+\eps)k ).
\end{equation}
Let us now bound $\Pee( M(U) \geq (1+\eps)k )$.
To do this we will condition on $U=u$ and give a uniform bound on $\Pee( M(u) \geq (1+\eps)k)$.
The random variables $M_z$, $z \in \Gamma(u)$ can be written as
$a_1 M_{z,1} + \dots + a_m M_{z,m}$ with the
$M_{z,i}$ independent Poisson variables
with means
\[
  \Ee M_{z,i} \leq (1+\frac{\eps}{100})^2 \vol(A_i) \sigma n r^d =: \mu_i.
\]
Let us partition $\Gamma(u)$ into subsets $\Gamma_1, \dots, \Gamma_K$ with
$K = O( r^{-d} )$ such that
\begin{equation}\label{eq:muiconstraints}
\sum_{z\in \Gamma_j} \Ee M_{z,i} \leq \mu_i \text{ for all } i\in\{1,\dots,m\}.
\end{equation}
To see that this can be done, notice we can inductively choose
maximal subsets $\Gamma_j \subseteq \Gamma(u) \setminus \bigcup_{j'<j} \Gamma_{j'}$
with the property $\sum_{z\in \Gamma_j} \Ee M_{z,i} \leq \mu_i$ for all $i\in\{1,\dots,m\}$
(where by maximal we mean that the addition to $\Gamma_j$ of any $z \not\in \bigcup_{j'\leq j}\Gamma_{j'}$ would violate this last property).
With the $\Gamma_j$ chosen in this way, we must have
that $\Gamma_j\cup\{z\}$ violates one of the constraints~\eqref{eq:muiconstraints}
for any $z\in\Gamma_{j+1}$.
Thus, in particular
$\sum_{i=1}^m\sum_{z \in \Gamma_j\cup\Gamma_{j+1}} \Ee M_{z,i} >
\min_i \mu_i$ if $\Gamma_{j+1}\neq\emptyset$.
Consequently, if we were able to select $K$ subsets $\Gamma_j$ we must have
\[
  \lfloor\frac{K-1}{2}\rfloor \min_i \mu_i \leq
  \sum_{j=1}^K \sum_{i=1}^m \sum_{z\in\Gamma_j} \Ee M_{z,i} \leq (1+\frac{\eps}{100})n,
\]
where the second inequality follows because the $M_{z,i}$ correspond to the number of points of a Poisson process of
total intensity $(1+\frac{\eps}{100})n$ in disjoint regions of $\eR^d$.
So we must indeed have $K = O(r^{-d})$, and that the process of selecting $\Gamma_j$ must have
stopped after $O(r^{-d})$ many $\Gamma_j$ were selected.

Set $M_{\Gamma_j} := \sum_{z \in \Gamma_j} M_z$.
As $\Gamma(u)=\bigcup_j \Gamma_j$ we have
\[
  M(u) = \max_{z\in\Gamma(u)} M_z \leq \max_j M_{\Gamma_j}.
\]
Note the $M_{\Gamma_j}$ are stochastically dominated by
$Z = a_1 Z_1 + \dots + a_m Z_m$, where the $Z_i$ are independent with $Z_i \sim \Po( \mu_i )$.
Thus
\[
  \Pee( M(u) \geq (1+\eps) k ) \leq K \Pee( Z \geq (1+\eps) k ).
\]
Because this bound does not depend on the choice of $u$ we can also conclude
%
\begin{equation} \label{eqn.probineq2}
  \Pee( M(U) \geq (1+\eps)k )
  \leq K \Pee( Z \geq (1+\eps) k ).
\end{equation}
We then have:
\[ \begin{array}{rcl}
\Pee( Z \geq (1+\eps)k )
& = &
\Pee( \sum a_i Z_i \geq \frac{1+\eps}{(1+\frac{\eps}{100})^2} \sum_i a_i \mu_i )
\leq
\sum_{i=1}^m \Pee( Z_i \geq \frac{1+\eps}{(1+\frac{\eps}{100})^2} \mu_i ) \\
& \leq &
m \cdot \exp[ - \min_i \mu_i H\left(\frac{1+\eps}{(1+\frac{\eps}{100})^2}\right) ],
\end{array} \]
using Lemma~\ref{lem:chernoffH}.
Now suppose that $T$ has been chosen in such a way that (and we may suppose this):
\[
  T \cdot (1+\frac{\eps}{100})^2 \cdot 
  \min_i \vol(A_i)\cdot H\left(\frac{1+\eps}{(1+\frac{\eps}{100})^2}\right) \geq 3,
\]
  so that
\[
  \exp[ - \min_i \mu_i H\left(\frac{1+\eps}{(1+\frac{\eps}{100})^2}\right) ] \leq n^{-3},
\]
  whenever $\sigma nr^d \geq T\ln n$.
  Because $K = O(r^{-d})$ and $\sigma nr^d \geq T\ln n$, we have that $K = O(n)$.
  By~(\ref{eqn.probineq2}) 
  we then also have $\Pee( M(U) \geq (1+\eps)k ) = O(n^{-2})$.
  Combining this with~\eqref{eq:MphiMphidashUB} and~\eqref{eq:MphiMU} we find
\[ \Pee( M_\varphi \geq (1+\eps) k ) = O( n^{-2} ). \]
  The Borel-Cantelli lemma now gives the result.
\end{proof}


\subsection{Proof of Theorem~\ref{thm.Mphi} on  $M_\varphi$}
\label{subsec.Mphi-proof}

  Our next target will be to prove Theorem~\ref{thm.Mphi}
  on the generalised scan statistic $M_\varphi$. We will
  do this along the lines of the proof of Lemma~\ref{lem:MphilargeT}.
  We will however need a straightforward generalisation of a Chernoff bound to weighted sums
  of Poisson variables, which is given by the following lemma.

\begin{lemma}\label{lem:chernoffsum}
  Let $X_1, \dots, X_m$ be independent Poisson variables with
  $X_i\sim \Po(\lambda_i\mu)$ where $\lambda_i > 0$ is fixed, and set
  $Z := a_1 X_1 + \dots + a_m X_m$ with $a_1,\dots,a_m>0$ fixed.
  Then for each fixed $s > 0$, as $\mu \to \infty$
\[
  \Pee( Z \geq \mu \sum_i \lambda_i a_i e^{a_i s} ) =
  \exp\left( -\mu \sum_i \lambda_i H( e^{a_i s} ) + o(\mu)
  \right).
\]
\end{lemma}

\begin{proof}
  The moment generating function of $Z$ (evaluated at $s$) is
\[
  \Ee e^{s Z} = \Pi_i \Ee e^{a_i s X_i} = \exp[ \sum_i \lambda_i \mu (e^{a_i s} - 1) ].
\]
  Hence Markov's inequality gives
\begin{eqnarray*}
  \Pee( Z \geq \mu \sum_i \lambda_i a_i e^{a_i s} )
 & = &
  \Pee( e^{s Z} \geq e^{\mu s \sum_i \lambda a_i e^{a_i s} }  ) \\
 & \leq &
  \exp[ \mu \sum_i \lambda_i (e^{a_i s} - 1) - \mu s \sum_i \lambda_i a_i e^{a_i s}  ] \\
 & = &
  \exp[ - \mu \sum_i \lambda_i (a_i s e^{a_i s} - e^{a_i s} + 1) ]\\
 & = &
  \exp[ - \mu \sum_i \lambda_i H(e^{a_i s}). ]
\end{eqnarray*}
  %
  On the other hand,
\begin{eqnarray*}
  \Pee( Z \geq \mu \sum_i \lambda_i a_i e^{a_i s} ) & \geq &
  \Pee( X_1 \geq \mu \lambda_1 e^{a_1 s}, \dots, X_m \geq \mu \lambda_m e^{a_m s} )\\
  &=& \exp[ -\mu \sum_i \lambda_i H(e^{a_i s}) + o(\mu) ],
\end{eqnarray*}
  using Lemma~\ref{lem:chernoffH2}.
\end{proof}


\begin{pf}{of Theorem~\ref{thm.Mphi}}
  Suppose first that
\begin{equation} \label{eqn.Mphi-ass}
  \frac{\sigma n r^d}{\ln n} \to t \in (0,\infty) \mbox{ as } n \to \infty.
\end{equation}
Let us observe that in this case the statement to be proven amounts to
\begin{equation}\label{eqn.Mphi-lim}
\frac{M_{\varphi}}{\sigma nr^d} \to \xi(\varphi,t) \;\; \text{ a.s. }
\end{equation}
  We proceed as in the proof of Lemma~\ref{lem:MphilargeT}.
  Again it suffices to prove Theorem~\ref{thm.Mphi} for $\varphi$ a simple function,
  because the functions $\varphi$ considered can be well approximated by the
  functions $\varphi_m^{\text{lower}}, \varphi_m^{\text{upper}}$
  defined in the proof of Lemma~\ref{lem:MphilargeT}, where this
  time we mean by ``well approximated'' that
\begin{equation}\label{eq:xilimms}
  \lim_{m\to\infty} \xi(\varphi_m^{\text{lower}}, t) = \lim_{m\to\infty} \xi(\varphi_m^{\text{upper}}, t) = \xi(\varphi,t).
\end{equation}
  Observe that~\eqref{eq:xilimms} follows from  
  part~\ref{itm:phiseq} of Lemma~\ref{lem:xibasic}
  ($\varphi$ is bounded and has bounded support). So the result for
  non-simple functions will follow from the result for simple
  functions by noticing that
  $M_{\varphi_m^{\text{lower}}} \leq M_{\varphi} \leq M_{\varphi_m^{\text{upper}}}$
  for all $m$ and taking $m \to \infty$. 
  
  Now let $\varphi = \sum_{i=1}^m a_i 1_{A_i}$ be a tidy simple function with the sets $A_i$ disjoint,
  and with $0<\int \varphi<\infty$.
  \m{need tidy, $0<\int \varphi<\infty$?}
  Then $s=s(\varphi,t)>0$ solves $\int_{\eR^d} H( e^{s\varphi(x)} ){\dd}x = 1/t$.
  Note that
\[ \begin{array}{l}
  \int_{\eR^d} \varphi(x) e^{s\varphi(x)}{\dd}x =
  \sum_{i=1}^m a_i e^{s a_i} \vol(A_i), \\
  \int_{\eR^d} H( e^{s\varphi(x)} ){\dd}x
  = \sum_{i=1}^m H(e^{s a_i}) \vol(A_i).
\end{array} \]
  Let us set
\begin{equation}\label{eq:Mphikdef}
  k := \xi(\varphi,t)\sigma nr^d.
\end{equation}
  Again it suffices to prove that
  $(1-\eps)k \leq M_\varphi \leq (1+\eps)k$ a.a.a.s., for any $\eps > 0$.

\smallskip

\noindent
{\bf Proof of lower bound in~(\ref{eqn.Mphi-lim}):}
  We proceed as in the proof of the lower bound in Lemma~\ref{lem:MphilargeT}.
  We restate~(\ref{eqn.probineq1}) from there as:
\begin{equation}\label{eq:MphiLBint}
\Pee( M_\varphi \leq (1-\eps)k ) \leq \Pee( Z \leq (1-\eps)k )^K + e^{-\alpha n},
\end{equation}
  where $\alpha > 0$ is a fixed constant, $K = \Omega(r^{-d})$, and $Z = a_1Z_1+\dots+a_mZ_m$
  with the $Z_i$ independent $\Po(\mu_i)$-random variables,
  where $\mu_i := (1-\frac{\eps}{100})^2\sigma n r^d \vol(A_i)$.  We can write
\[
  \Pee( Z \leq (1-\eps)k ) =
  \Pee( Z \leq \frac{(1-\eps)}{(1-\frac{\eps}{100})^2} \sum_{i=1}^m a_i e^{s a_i} \mu_i )
  = \Pee( Z \leq \sum_{i=1}^m a_i e^{s' a_i} \mu_i ),
\]
  where $s' = s'(t,\eps)$ solves
  $\sum_{i=1}^m a_i e^{s' a_i} \vol(A_i) = 
  \frac{(1-\eps)}{(1-\frac{\eps}{100})^2} \sum_{i=1}^m a_i e^{s a_i} \vol(A_i)$.
  Note that $s' < s$, and (provided $\eps$ is small enough) also $s' > 0$.
  Lemma~\ref{lem:chernoffsum} now gives:
\[
   1-\Pee( Z \leq(1-\eps)k ) = \Pee( Z > (1-\eps)k )
  = \exp[ - (1-\frac{\eps}{100})^2\sigma n r^d ( \sum_{i=1}^m H(e^{a_i s'}) \vol(A_i)+ o(1) ) ]
\]
  As $0 < s' < s$ we have that
  $\sum_{i=1}^m H(e^{a_i s'}) \vol(A_i)
  < \sum_{i=1}^m H(e^{a_i s}) \vol(A_i) = \frac{1}{t}$.
  Consequently there is a constant $c = c(t,\eps) >0$ such that
\[
  \Pee( Z > (1-\eps)k )= \exp[ -(1-c+o(1))\ln n ] = n^{-1+c+o(1)}.
\]
  It follows that
\[
  \Pee( Z \leq (1-\eps)k )^K \leq (1-n^{-1+c+o(1)})^{K} \leq
  \exp[ - K n^{-1+c+o(1)} ] \leq \exp[ - n^{c+o(1)}],
\]
  using that $K$ is at least $n^{1+o(1)}$ (as $K = \Omega(r^{-d})$ and $r^{-d} \sim \frac{n}{t\ln n}$),
  we see that the right hand side of~\eqref{eq:MphiLBint} sums in $n$, so that we may conclude
  that $M_\varphi \geq (1-\eps)k$ a.a.a.s. by Borel-Cantelli.

\smallskip

  \noindent
{\bf Proof of upper bound in~(\ref{eqn.Mphi-lim}):}
Let $N, M_\varphi', \eta, \varphi_\eta, M(U)$ be as in the proof of the upper bound in Lemma~\ref{lem:MphilargeT},
where now $\eta>0$ satisfies $\xi(\varphi_{\eta},t) < (1+\eps/100)\xi(\varphi,t)$;
and recall from~(\ref{eqn.probineq2}) that:
\[ \Pee( M(U) \geq (1+\eps)k ) \leq K\Pee( Z \geq (1+\eps)k ), \]
  where $K = O(r^{-d} )$ and $Z = a_1 Z_1 + \dots + a_m Z_m$
  with the $Z_i$ independent $\Po( \mu_i )$ random variables, where
  $\mu_i := (1+\frac{\eps}{100})^2 \vol(A_i) \sigma n r^d$. We now have
\[
  \Pee( Z \geq (1+\eps)k )
  = \Pee( Z \geq \frac{1+\eps}{(1+\frac{\eps}{100})^2} \sum_i a_i e^{s a_i} \mu_i )
  = \Pee( Z \geq \sum_i a_i e^{s' a_i} \mu_i ),
\]
  where $s' = s'(\eps,t)$ is such that
  $\sum_{i=1}^m a_i e^{s' a_i} \vol(A_i) =
  \frac{1+\eps}{(1+\frac{\eps}{100})^2} \sum_{i=1}^m a_i e^{s a_i} \vol(A_i)$.
  Note that $s' > s$, giving
  $\sum_i H(e^{s'a_i})\vol(A_i) > \sum_i H(e^{s a_i})\vol(A_i) =\frac{1}{t}$,
  and consequently
\[
\sum_i H(e^{s' a_i})\mu_i =
(1+\frac{\eps}{100})^2\sigma nr^d \sum_i H(e^{s' a_i})\vol(A_i)
= (1+c+o(1))\ln n,
\]
  for some $c = c(\eps, t) > 0$. Since $K = O( r^{-d} ) \leq n$ for $n$ large enough we find that:
\begin{equation}\label{eq:MU}
  \Pee( M(U) \geq (1+\eps)k ) \leq n \Pee( Z > (1+\eps) k )
  = n \exp[ - (1+c+o(1))\ln n ] = n^{-c+o(1)}.
\end{equation}
  Unfortunately this does not necessarily sum in $n$, so we will have to use a more elaborate method
  than the one used in Lemma~\ref{lem:MphilargeT}.
  Note that for any $0 < \eta' < \eta$ we have, completely analogously to \eqref{eq:MphiMU}:

\begin{equation}\label{eq:MpsiMU}
\Pee( M_{\varphi_{\eta'}}' \geq (1+\eps)k) \leq \frac{R^d}{\volB (\eta-\eta')^d} \Pee( M(U) \geq (1+\eps)k ),
\end{equation}

\noindent
By \eqref{eq:MpsiMU} and \eqref{eq:MphiMphidashUB} we also have that for
all $0\leq \eta' < \eta$:

\[
\Pee( M_{\varphi_{\eta'}} \geq (1+\eps)k ) \leq n^{-c+o(1)} + e^{-\alpha n} = n^{-c+o(1)}.
\]

Although the right hand side does not necessarily sum in $n$, it does hold that if $L > 0$ is such that
$c L > 1$ then we can apply the Borel-Cantelli lemma to show that

\begin{equation}\label{eq:Mpsisubsequence}
\Pee( M_{\varphi_{\eta'}}(m^L,r(m^L)) < (1+\eps) k(m^L) \textrm{ for all but finitely many } m ) = 1.
\end{equation}

\noindent
We now claim that from this we can conclude that in fact
$M_\varphi \leq (1+2\eps)k$ a.a.a.s.
To this end, let $n \in \eN$ be arbitrary and let $m=m(n)$ be such that $(m-1)^L < n \leq m^L$.
The claim follows if we can show that (for $n$ sufficiently large)

\begin{equation}\label{eq:eqny}
\{M_{\varphi_{\eta'}}(m^L,r(m^L)) \leq (1+\eps)k(m^L) \} \Rightarrow
\{ M_\varphi(n) \leq (1+2\eps)k(n)\}.
\end{equation}

\noindent
To this end we will first establish that (for n sufficiently large and) for
any $x,y$:

\begin{equation}\label{eq:eqnx}
  \varphi( \frac{y-x}{r(n)} ) \leq \varphi_{\eta'}(\frac{y-x}{r(m^L)}).
\end{equation}
  Since the support of $\varphi$ is contained in $[\frac{-R}{2},\frac{R}{2}]^d$ we are done if
$\norm{\frac{y-x}{r(n)}} > \diam([0,\frac{R}{2}]^d) =: \gamma$.
If on the other hand $\norm{\frac{y-x}{r(n)}} \leq \gamma$ then
$\norm{\frac{y-x}{r(n)} - \frac{y-x}{r(m^L)}}
= \tel{1-\frac{r(n)}{r(m^L)}}\norm{\frac{y-x}{r(n)}}
\leq \tel{1-\frac{r(n)}{r(m^L)}}\gamma = o(1)$ (because $n r^d \sim t\ln n$ giving $r(n)= (1+o(1))r(m^L)$), so that
for $n$ sufficiently large this is $< \eta'$ and thus~\eqref{eq:eqnx} holds uniformly for all $x,y$ (for such sufficiently large $n$), as required.
  Since we also have $k(n) = (1+o(1))k(m^L)$, equation~\eqref{eq:eqny} does indeed hold for $n$ sufficiently large,
  which concludes the proof of~(\ref{eqn.Mphi-lim}) under the assumption~(\ref{eqn.Mphi-ass}) where $t>0$.
  \medskip
  
  \noindent
{\bf Completing the proof:}  
  Now let us drop the assumption~(\ref{eqn.Mphi-ass}) and prove the remaining parts of the theorem.
  Let $t(n) := \sigma n r^d / \ln n$ be as in the statement of Theorem~\ref{thm.Mphi}. 
  Let $\tau=\liminf_n \frac{\sigma n r^d}{\ln n}= \liminf_n t(n)$ so $\tau>0$ by assumption.
  Let $0<\eps<\frac12$.  We want to show
\begin{equation} \label{target.Mphibdd}
  1-\eps < \liminf_n \frac{M_{\varphi}}{\sigma n r^d \xi(\varphi,t(n))}
  \leq \limsup_n \frac{M_{\varphi}}{\sigma n r^d \xi(\varphi,t(n))} < 1+\eps \;\; \mbox{ a.s.}
\end{equation}
  By Lemma~\ref{lem:MphilargeT} and \ref{itm:toinf} in Lemma~\ref{lem:xibasic},  
  there is a constant $T<\infty$ such that if $\liminf_n \frac{\sigma n r^d}{\ln n} \geq T$
  then~(\ref{target.Mphibdd}) holds.  To cover the range between $\tau$ and $T$ we will use the following claim.
  \medskip
  
  \noindent
{\bf Claim}
Let $0 < t_a < t_b < \infty$ be such that $\frac{t_b}{t_a} < 1+\eps$.
  If $r(n)$ is such that $t_a \leq t(n) \leq t_b$ for $n$ sufficiently large then
\begin{equation} \label{claim.Mphibdd}
  1-\eps < \liminf_n \frac{M_{\varphi}}{\sigma n r^d \xi(\varphi,t(n))}
  \leq \limsup_n \frac{M_{\varphi}}{\sigma n r^d \xi(\varphi,t(n))} < 1+\eps \;\; \mbox{ a.s.}
\end{equation}

\begin{pf}{of Claim}
Let us set 

\[ 
n_a = n_a(n) := \lfloor(t_a/t(n))\cdot n\rfloor, \quad n_b = n_b(n) := \lceil(t_b/t(n))\cdot n\rceil.
\]

\noindent
By obvious monotonicities

\[ 
M_{\varphi}(n_a(n),r(n)) \leq  M_{\varphi}(n,r(n)) \leq   M_{\varphi}(n_b(n),r(n))
\] 

\noindent  
and by part~\ref{itm:xitplush} of Lemma~\ref{lem:xibasic}

\[
  \xi(\varphi,t_b) \leq \xi(\varphi,t(n)) \leq \xi(\varphi,t_a) \leq (t_b/t_a) \xi(\varphi,t_b).
\]

\noindent
Hence

\begin{equation}\label{eq:A} 
\frac{M_{\varphi}(n,r(n))}{\sigma n (r(n))^d \xi(\varphi,t(n))} 
\geq 
\frac{M_{\varphi}(n_a(n),r(n))}{\sigma n (r(n))^d \xi(\varphi,t_a)}, 
\end{equation}

\noindent
and, since $n_a(n) \sim (t_a/t(n))\cdot n$, we also have

\begin{equation}\label{eq:B} 
\frac{M_{\varphi}(n_a(n),r(n))}{\sigma n (r(n))^d \xi(\varphi,t_a)}
\sim 
\frac{M_{\varphi}(n_a(n),r(n))}{\sigma n_a(n) (r(n))^d \xi(\varphi,t_a)} \cdot \frac{t_a}{t(n)}.
\end{equation}

\noindent
Observe that $\sigma n_a(n) (r(n))^d \sim \sigma (t_a/t(n)) n (r(n))^d = t_a \ln n \sim 
t_a \ln n_a(n)$.
By the already proved special case of the result under the assumption~\eqref{eqn.Mphi-ass} applied to
$M( n_a(n), r(n) )$, we therefore have

\begin{equation}\label{eq:C} 
\frac{M_{\varphi}(n_a(n),r(n))}{\sigma n_a(n) (r(n))^d \xi(\varphi,t_a)} \sim 1 \;\; \text{ a.s.}
\end{equation}

\noindent
To be more explicit, we use the special case twice.
Since $\eps < \frac12$ and so $\frac{t_a}{t(n)} > \frac12$, for each positive integer $m$
the set $\{n : n_a(n)=m\}$ has size 1 or 2.
We apply the special case once to $M_{\varphi}(m,r_L(m))$ and once to $M_{\varphi}(m,r_U(m))$ 
where $r_L(m) := \min\{ r(n) : n_a(n) = m \}$ and  $r_U(m) := \max\{ r(n) : n_a(n) = m \}$; 
and finally we note that $M_{\varphi}(n_a(n),r(n))$ must be one of
$M_{\varphi}(m, r_L(m))$ or $M_{\varphi}(m,r_U(m))$ for each $n$ where $m=n_a$.

Hence, combining~\eqref{eq:A},~\eqref{eq:B},~\eqref{eq:C}, we also get

\[ \liminf_n \frac{M_{\varphi}}{\sigma n r^d \xi(\varphi,t(n))} 
\geq \liminf_n \frac{t_a}{t(n)} \geq \frac{t_a}{t_b} > 1-\eps \;\; \text{ a.s. } \]

\noindent
Completely analogously, $\limsup_n \frac{M_{\varphi}}{\sigma n r^d \xi(\varphi,t(n))} \leq \frac{t_b}{t_a}$
a.s. This completes the proof of the claim~(\ref{claim.Mphibdd}).
\end{pf} 
 
  Now we put the pieces together.
  Let $m$ and $t_1<t_2< \cdots<t_m$ be such that $t_1 < \tau$, $t_m \geq T$ and
$\frac{t_{k+1}}{t_k} < 1+\eps$ for $k=1,\dots,m-1$.
For convenience let us also set $t_{m+1}:=\infty$.  For each $k=1,\ldots,m$ let   
\[ 
  r_k(n) :=
  \left\{ \begin{array}{cl} r(n) & \textrm{if } \left(\frac{t_k\ln n}{\sigma n}\right)^{\frac{1}{d}} 
\leq r(n) <  \left(\frac{t_{k+1}\ln n}{\sigma n}\right)^{\frac{1}{d}} \\
                       \left(\frac{t_k\ln n}{\sigma n}\right)^{\frac{1}{d}} & \textrm{otherwise} \end{array} \right.
\]

Observe that for each $n$, $M_{\varphi}(n,r(n))$
equals $M_{\varphi}(n,r_k(n))$ for some $k$ (that may vary with $n$).
Since each $M_{\varphi}(n,r_k(n))$ separately satisfies~\eqref{target.Mphibdd}
and the intersection of finitely many events of probability one has again probability one, 
$M_{\varphi}(n,r)$ itself also satisfies~\eqref{target.Mphibdd}.
The theorem follows.
\end{pf}


\section{Proof of parts~\ref{itm:tobiasomega.intermediate}
and~\ref{itm:tobiasomega.dense} of Theorem~\ref{thm.tobiasomega} on $\omega(G_n)$}
\label{sec.clique-proof}

  In this section we use Theorem~\ref{thm.Mphi} on generalised scan statistics
  (together with Lemma~\ref{lem:xibasic}) to give a quick proof of
  the following theorem, which is in a general form that is convenient for the proof of 
Theorem~\ref{thm.tobiasfrat} later on.
  
\begin{theorem}\label{thm.tobiasomegareal}
If $t(n) := \frac{\sigma n r^d}{\ln n}$ satisfies $\liminf_n t(n) > 0$ then
  
\[ \frac{\omega(G_n)}{\sigma n r^d} \sim \xi(\varphi_0, t(n)) \;\; \text{ a.s. } \]
  
\end{theorem}

\noindent 
Since we have already established in Section~\ref{sec.scaling-function} that 
$\fcli$ satisfies the properties claimed by part~\ref{itm:tobiasomega.intermediate}, 
this implies parts~\ref{itm:tobiasomega.intermediate}
and~\ref{itm:tobiasomega.dense} of Theorem~\ref{thm.tobiasomega}.

\vspace{1ex}

\begin{proof}
  Assume that $\liminf_n t(n)>0$. First set $W := B(0; \frac{1}{2})$.
  Any set of points contained in a translate of $r W$ is a clique of $G_n$, so that by Theorem~\ref{thm.Mphi}:
\[
  \liminf_{n\to\infty} \frac{\omega(G_n)}{\sigma n r^d \xi(\varphi_0,t(n))} \geq
  \lim_{n\to\infty} \frac{M_W}{\sigma n r^d \xi(\varphi_0,t(n))} = 1 \textrm{ a.s. }
\]
  Let us now fix $\eps > 0$ and let $A_1, \dots, A_m \subseteq \eps\Zed^d$ be all the subsets of
  $\eps\Zed^d$ that satisfy $0 \in A_i$ and
  $\diam(A_i) \leq 1+2\eps\rho$, where $\rho := \diam([0,1]^d)$.
  Let $W_i := \conv(A_i)$. We now claim that $\omega(G_n) \leq \max_i M_{W_i}$.
  To see that this holds, suppose $X_{i_1}, \dots, X_{i_k}$ form a
  clique in $G_n$. Let us set $y_j := (X_{i_j}-X_{i_1}) / r$ and
  $A:= \{ p \in \eps\Zed^d : \norm{p-y_i} \leq \eps\rho \textrm{ for some } 1\leq i\leq k\}$.
  Observe that $0 = y_1 \in A$ and $\diam(A)\leq 1+2\eps\rho$,
  so that $A = A_i$ for some $1\leq i\leq m$.  What is more
  $\{y_1,\dots,y_k\} \subseteq W := \conv(A)$. But this gives
  $\{X_{i_1},\dots, X_{i_k}\}\subseteq X_{i_1}+rW$ by choice of $y_i$, and the claim follows. 

We will need the Bieberbach inequality, which is sometimes also called the 
isodiametric inequality. (For a proof of this classical result, see for instance Gruber and Wills~\cite{gruberwills93}.)

\begin{lemma}[Bieberbach inequality]\label{thm:bieberbach}
Let $A \subseteq \eR^d$ be measurable and bounded. If $A'$ is a ball 
with $\diam(A)=\diam(A')$ then $\vol(A) \leq \vol(A')$.\noproof
\end{lemma}

  By the Bieberbach inequality 
  $\vol(W_i) \leq \vol( B(0;\frac{1+2\eps\rho}{2}) )$, so that also
  $\int \varphi_i 1_{\{\varphi_i\geq a\}} \leq \int \psi 1_{\{\psi\geq a\}}$
  for all $a$ where  $\varphi_i = 1_{W_i}$ and $\psi$ denotes $1_{B(0;\frac{1+2\eps\rho}{2})}$.
  Also observe that $\psi(x) = \varphi_0( \frac{x}{1+2\eps\rho})$.
  By parts~\ref{itm:phiindicatorbiggest} and~\ref{itm:philambda} of
  Lemma~\ref{lem:xibasic} we therefore have
  $\max_i \xi(\varphi_i,t) \leq \xi(\psi,t) \leq (1+2\eps\rho)^d \xi(\varphi_0,t)$ for any $t \in(0,\infty)$.
  Hence for each $i$ we have
\[
  \frac{M_{W_i}}{\xi(\varphi_0,t(n))} \leq \frac{M_{\varphi_i}}{\xi(\varphi_i,t(n))} \cdot (1+2\eps\rho)^d,
\]
  and so by Theorem~\ref{thm.Mphi}
\[
  \limsup_n \frac{M_{W_i}}{\xi(\varphi_0,t(n))} \leq (1+2\eps\rho)^d \;\; \mbox{ a.s.}
\]
  Now by the claim established above we have $\limsup_n \frac{\omega(G_n)}{\xi(\varphi_0,t(n))} \leq (1+2\eps\rho)^d$ a.s.
  It follows that
\[
  \frac{\omega(G_n)}{\sigma n r^d \xi(\varphi_0,t(n))} \to 1 \mbox{ a.s.}
\]
 which completes the proof.
\end{proof}


\section{Proof of parts~\ref{itm:tobiaschi.intermediate} and~\ref{itm:tobiaschi.dense} 
of Theorem~\ref{thm.tobiaschi}}
\label{sec.col-proof}

  As we mentioned earlier, in Theorem~\ref{thm.tobiaschi} the same conclusions will hold if we replace $\chi(G)$
  by the fractional chromatic number $\chi_f(G)$ ,and this is the key to the proof. 
  We first give two deterministic results on $\chi_f$ and $\chi$ for a geometric graph.
  Given a finite set $V \subset R^d$, with say $|V|=n$, let us list $V$ arbitrarily as $v_1,\ldots,v_n$;
  and for $r>0$ set $G(V,r)$ as the geometric graph $G(v_1,\ldots,v_n;r)$.
  We are not interested here in the vertex labelling.
  We show that for each such set $V$ 
  we have $\chi_f(G(V,1)) = \sup_{\varphi \in {\cal F}} M(V, \varphi)$;
  and then we give an upper bound on $\chi(G(V,1))$ (from rounding up a solution for $\chi_f$)
  of the form $(1+\eps) \max_{i=1,\ldots,m} M(V, \varphi_i) +c$, where the functions $\varphi_i$ are
  nearly feasible tidy functions.
  After that, we give three technical lemmas, and use Theorem~\ref{thm.Mphi} to complete the proof.


\subsection{Deterministic results on fractional chromatic number}
\label{subsec.chif-ud} 

  In this subsection
  we give two deterministic results on $\chi_f(G)$ for geometric graphs $G$.
  %
  Recall that if $\varphi:\eR^d\to\eR$ is a function and $V\subseteq \eR^d$ is a set of points
  then $M(V,\varphi) := \sup_{x\in\eR^d}\sum_{v\in V}\varphi(v-x)$.

\begin{lemma} \label{lem.chifg}
  Let $V\subseteq\eR^d$ be a finite set of points and consider the graph $G=G(V,1)$. Then
\[
  \chi_f(G)= \sup_{\varphi \in {\cal F}} M(V,\varphi).
\]
\end{lemma}

\begin{proof}
  Recall that $\chi_f(G)$ is the objective value of the LP relaxation of the integer LP~(\ref{eq:ILP}).
  By LP-duality $\chi_f(G)$ also equals the objective value of the dual LP:
\[
  \begin{array}{rl}
  \max & 1^T y \\
  \textrm{ subject to } &
  A^T y \leq 1, \\
  & y \geq 0.
  \end{array}
\]
  For convenience let us write $V = \{v_1,\dots,v_n\}$ where $v_i$ is the vertex corresponding to the $i$-th row of $A$
  (and thus the $i$-th column of $A^T$). Notice that a vector $y=(y_1,\dots,y_n)^T$ is
  feasible for the dual LP if and only if it attaches nonnegative weights to the vertices of $G$
  in such a way that each stable set has total weight at most one.
  There is a natural correspondence between such vectors $y$ and certain feasible functions
  (hence our choice of the name `feasible function').

  Let $\varphi$ be any feasible function and $x\in\eR^d$ an arbitrary point.  We claim that the vector
\[
y = (\varphi(v_1-x), \dots, \varphi(v_n-x))^T \
\]
  is a feasible point of the dual LP given above. To see this note that each row of $A^T$
  is the incidence vector of some stable set $S$ of $G$;
  and $\norm{(z-x)-(z'-x)}=\norm{z-z'} > 1$ for each $z \neq z' \in S$ since $S$ is stable in $G$. Hence
 \[
  (A^T y)_S  = \sum_{z \in S} \varphi(z-x) \leq 1,
  \]
  by feasibility of $\varphi$.
  This holds for all rows of $A^T$, so that $y$ is indeed feasible for the dual LP as claimed.
  Also, notice that the objective function value $1^T y$ equals $\sum_{j=1}^n \varphi(v_j-x)$.
  This shows that
\[
  \chi_f(G) \geq
  \sup_{\varphi\in{\cal F}}\sup_{x\in\eR^d} \sum_{j=1}^n \varphi(v_j-x)
  = \sup_{\varphi\in{\cal F}} M(V,\varphi).
\]
  Conversely let the vector $y=(y_1,\dots,y_n)^T$ be feasible for the dual LP.
  Define $\varphi(z)= \sum_{i=1}^n y_i 1_{z=v_i}$.
  Then $\varphi$ is clearly a feasible function, and
\[
  1^T y = \sum_{j=1}^n\varphi(v_j) \leq \sup_{x\in\eR^d} \sum_{j=1}^n\varphi(v_j-x) = M(V,\varphi),
\]
  so that $\chi_f(G) \leq \sup_{\varphi\in{\cal F}}M(V,\varphi)$.
\end{proof}
  \smallskip


  We now turn our attention towards deriving an upper bound on the chromatic number, and give another 
deterministic lemma.
  Given $\alpha>0$ we say that the function $\varphi$ on $\eR^d$
  is $\alpha$-feasible if the function $\varphi_\alpha(x)=\varphi(\alpha x)$
  is feasible (that is, if $S\subseteq\eR^d$ satisfies $\norm{s-s'}>\alpha$ for all
  $s\neq s'\in S$ then $\sum_{s\in S}\varphi(s)\leq 1$).
  Thus $1$-feasible means feasible; and if $\alpha<\beta$ and $\varphi$ is $\alpha$-feasible
  then $\varphi$ is $\beta$-feasible.

\begin{lemma} \label{lem.detchi2}
  For each $\eps>0$ there exists a positive integer $m$, simple $(1+\eps)$-feasible, tidy functions
  $\varphi_1,\ldots,\varphi_m$, and a constant $c$ such that:
\[
\chi( G(V,1) ) \leq (1+\eps) \max_{i=1,\dots,m} M(V,\varphi_i) + c,
\]
  for each finite set $V \subseteq \eR^d$
\end{lemma}

\begin{proof}
  Let $\eps > 0$, and let $K \in \eN$ be a (large) integer.
Let us again set $\rho := \diam([0,1]^d)$ and set $L := \lceil (1+\eps\rho)/\eps\rceil$.
Observe that $\norm{y-z} \geq 1+\eps \rho$ whenever
  $\tel{y_i-z_i}\geq L\eps$ for some coordinate $1\leq i\leq d$.

  We shall show that there exist $(1+2\eps\rho)$-feasible tidy functions
  $\varphi_1,\ldots,\varphi_N$ 
  such that the following holds for any $V\subseteq\eR^d$:
\begin{equation} \label{eqn.detchi}
  \chi(G(V,1)) \leq \left(1+ \frac{L}{2K}\right)^d \max_i M(V,\varphi) +
  (2K)^{2d} \left(1+ \frac{L}{2K}\right)^{d}.
\end{equation}
  This of course yields the lemma, by adjusting $\eps$ and taking $K$ sufficiently large.

  We partition $\eR^d$ into hypercubes of side $\eps$.
  Let $\Gamma$ be the (infinite) graph with vertex set $\eps \Zed^d$ and an edge $pq$
  when $\norm{p-q} < 1+\eps \rho$.
  For each $q \in \eps \Zed^d$ let $C^q$ denote the hypercube $q + [0,\eps)^d$.
  Observe that the hypercubes $C^q$ for $q \in \eps \Zed^d$ partition $\eR^d$.
  Thus for each $z \in \eR^d$ we may define $p(z)$ to be the unique $q\in\eps\Zed^d$ such that $z \in C^q$.

  Now let $V_0 = [-K\eps,K\eps)^d \cap \eps \Zed^d$, and note that $|V_0| = (2K)^d$.
  For each $p \in \eps \Zed^d$ let $\Gamma^p$ be the
  subgraph of $\Gamma$ induced on the vertex set $p+V_0$,
  that is by the vertices of $\Gamma$ in $p + [-K\eps,K\eps)^d$.
  Observe that the graphs $\Gamma^p$ are simply translated copies of $\Gamma^0$.
  Let $B$ be the vertex-stable set incidence matrix of $\Gamma^0$.

  Now let $V$ be an arbitrary finite subset of $\eR^d$
  Given a subset $S$ of $\eR^d$, let us use the notation
  ${\cal N}(S)$ here to denote $\tel{S\cap V}$.
  Let $\Gamma_{V}$ be the graph we get by replacing each node $q$ of
  $\Gamma$ by a clique of size ${\cal N}( C^q )$ and adding all the edges between the cliques
  corresponding to $q,q'\in V_0$ if $qq'\in E(\Gamma^0)$.
  It is easy to see from the definition of the threshold distance in $\Gamma$
  that $G(V,1)$ is isomorphic to a subgraph of $\Gamma_{V}$.
  For each $p \in \eps\Zed^d$ let $\Gamma_{V}^p$ be the subgraph of
  $\Gamma_{V}$ corresponding to the vertices of $\Gamma^p$.

  Consider some $p \in \eps\Zed^d$.
  Then $\chi(\Gamma_{V}^p)$ is the objective value of the integer LP:
\begin{equation}\label{eq:primalILP}
\begin{array}{rl}
\min & 1^T x \\
\textrm{ subject to } &
  B x \geq b^p \\
  & x \geq 0, \ x \textrm{ integral }
  \end{array}
\end{equation}
  where $b^p = ({\cal N}(C^{p+q}))_{q\in V_0}$ and the vector
  $x$ is indexed by the stable sets in $\Gamma^0$.
  Here we are using the fact that $\Gamma^p$ is a copy of $\Gamma^0$
  and that the vertex corresponding to $q$ has been replaced
  by a clique of size ${\cal N}(C^{p+q})$.
  By again considering the LP-relaxation and switching to the dual we find that
  $\chi_f(\Gamma_{V}^p)$ equals the objective value of the LP:
\begin{equation}\label{eq:dualLP}
\begin{array}{rl}
  \max & (b^p)^T y \\
  \textrm{ subject to } &
  B^T y \leq 1 \\
  & y \geq 0.
  \end{array}
\end{equation}
  Notice that the vectors $y = (y_q)_{q\in V_0}$ 
  attach nonnegative weights to the points $q$ of $V_0$
  is such a way that if $S\subseteq V_0$ corresponds to a stable
  set in $\Gamma_0$ then the sum of the weights $\sum_{q\in S} y_q$
  is at most one.
  Note the important fact that the feasible region (that is, the set of all
  $y$ that satisfy $B^T y\leq 1, y\geq 0$) here does not depend on $p$ or $V$.

  The vectors $y$ that satisfy $B^T y \leq 1, y \geq 0$ correspond to `nearly feasible'
  functions $\varphi$ in a natural way, as follows.
  Observe that $x \in [-K\eps,K\eps)^d$ if and only if $p(x) \in V_0$.
  Let $\varphi:\eR^d \to \eR$ be defined by setting
\[
  \varphi(x) := \left\{\begin{array}{cl} y_{p(x)} & \text{ if } x \in [-K\eps,K\eps)^d, \\
  0 & \text{ otherwise. }
\end{array}\right.
\]
  Note that $\varphi(x) = \sum_{q \in V_0} 1_{C^q}(x) y_q$.
  Then for each $p \in \eps\Zed^d$
\[
\begin{array}{rcl}
  (b^p)^T y
 & = &
  \sum_{q\in V_0} {\cal N}(C^{p+q})y_{q}  =
  \sum_{q\in V_0} \sum_{v\in V} 1_{C^{p+q}}(v)y_q \\
 & = &
  \sum_{v\in V} \sum_{q\in V_0} 1_{C^{q}}(v-p)y_q
  = \sum_{v\in V} \varphi(v-p) \\
 & \leq &
  M(V,\varphi).
\end{array}
\]
  We claim next that the functions $\varphi$ thus defined are
  $(1+ 2\eps \rho)$-feasible; that is they satisfy $\sum_{j=1}^k \varphi(z_j) \leq 1$
  for any $z_1,\dots,z_k$ such that $\norm{z_j-z_l} > 1+2\eps\rho$ for all $j\neq l$.
  To see this, pick such $z_1,\dots, z_k$.
  Since $\varphi$ is 0 outside of $[-K\eps,K\eps)^d$ we may as well suppose that all the $z_j$ lie inside $[-K\eps,K\eps)^d$.
For $i=1,\dots,k$ let $p_i \in V_0$ be the unique point of $V_0$ such that $z_i \in p_i + [0,\eps)^d$.
  For all pairs $i \neq j$ we have
  $\norm{p_i-p_j} \geq \norm{z_i-z_j}-\eps\rho > 1 + \eps\rho$.
  Thus $p_1, \dots, p_k$ are distinct and form a stable set $S$ in $\Gamma^0$,
  and therefore correspond to one of the rows of $B^T$.
  The condition $B^T y \leq 1$ now yields
\[
  \varphi(z_1)+\dots+\varphi(z_k)= y_{p_1}+\dots+y_{p_k}= (B^Ty)_S \leq 1.
\]
  This shows that $\varphi$ is $(1+2\eps\rho)$-feasible as claimed,
  and it can be readily seen from the definition of $\varphi$ that it is simple and tidy.

  Recall that a basic feasible solution of an LP with $k$ constraints
  has at most $k$ nonzero elements and that, provided the optimum value is bounded,
  the optimum value of the LP is always attained at a basic feasible solution
  (see for example Chv\'atal~\cite{chvatal}).
  Thus, noting that by rounding up all the variables in an optimum basic feasible solution $x$
  to the LP-relaxation of~\eqref{eq:primalILP} we get a feasible
  solution of the ILP~\eqref{eq:primalILP} itself, we see that:
\[
  \chi(\Gamma^p_V) \leq \chi_f(\Gamma^p_V)+(2K)^d.
\]
  Now let $y^1, \dots, y^m$ be the vertices of the polytope $B^T y \leq 1, y \geq 0$
  and let $\varphi_1,\ldots,\varphi_m$ be the corresponding
  $(1+2\eps\rho)$-feasible, tidy functions.
  As the optimum of the LP~\eqref{eq:dualLP} corresponding to $\chi_f(\Gamma^p_V)$ is
 attained at one of these vertices we see that:
\begin{equation} \label{eqn.bounds}
\chi(\Gamma^p_V) \leq \max_{j=1,\dots,m} (b^p)^T y^j + (2K)^d
  \leq \max_{j=1,\dots,m} M(V,\varphi_j) + (2K)^d.
\end{equation}
  What is more, for each $p \in \eps\Zed^d$ we can colour any
  subgraph of $G(V,1)$ induced by the points in the set
  
 \[ W^p := p+[-K\eps,K\eps)^d+(2K+L)\eps\Zed^d, \]

\noindent
 with this many colours,
  since by the definition of $L = \lceil (1+\eps\rho)/\eps\rceil$, the set $W^p$ is the union of hypercubes of
  side $2K\eps$ which are far enough apart for any two points of
  $\Gamma$ in different hypercubes not to be joined by an edge.

\begin{figure}[h!]\label{fig:Wp}
 \begin{center}
  \input{Wp2.pstex_t}
 \end{center}
\caption{Depiction of a set $W^p$.}
\end{figure}
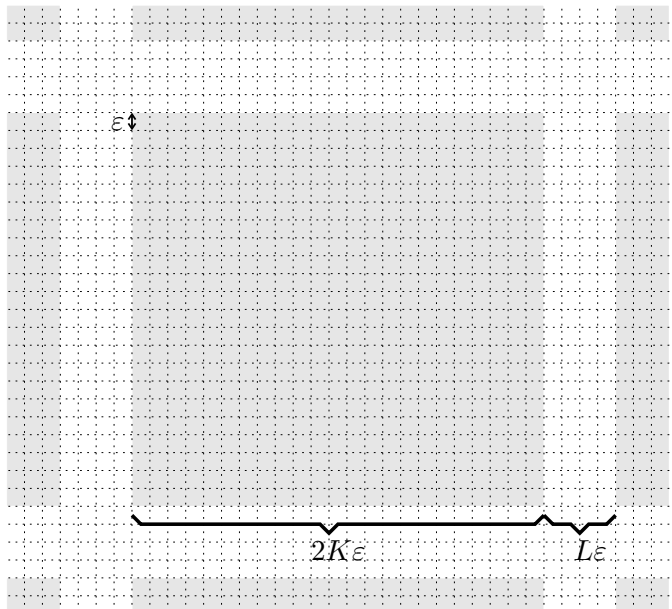

\noindent
Now let the set $P$ be defined by
\[
  P = \eps\Zed^d \cap [-K\eps,(K+L)\eps)^d =\{ (\eps i_1, \dots, \eps i_d) : -K \leq i_j <
  K+L \}.
\]
  Note that if $p$ runs through the set $P$ then each $q \in \eps\Zed^d$ is covered by exactly
 $(2K)^d$ of the sets $W^p$.  If $H^p$ is the graph we get by replacing every vertex $q$ of $\Gamma$ that
 lies in $W^p$ by a clique of size $\left\lceil\frac{{\cal N}(C^q)}{(2K)^d}\right\rceil$ rather than
 one of size ${\cal N}(C^q)$ and removing any vertex that does not lie in $W^p$, then
\[
  \chi(H^p) \leq \frac{1}{(2K)^d} \max_j M(V,\varphi_j) + (2K)^d.
\]
  This is because we can consider the hypercubes of side $2K\eps$ that
  make up $W^p$ separately (and each of these corresponds to
  some $\Gamma_V^q$) and for each such constituent hypercube
  $q+[-K\eps,K\eps)^d$ all we need to do is replace the `right hand side' vector $b^q$ by
  $\frac{1}{(2K)^d} b^q$ in the LP-relaxation of the ILP~\eqref{eq:primalILP}
  (the rounding up of the variables will then take care of the
  rounding up in the constraints, as the entries of $B$ are integers).
  Because each $q \in \eps\Zed^d$ is covered by exactly
  $N$ of the sets $W^p$ we can combine the $(2K+L)^d$ colourings of
  the graphs $H^p$ for $p \in P$ to get a proper colouring of $\Gamma_V$ with at most
\[
  (2K+L)^d \left(\frac{1}{(2K)^d} \max_j M(V,\varphi_j)+(2K)^d \right),
\]
  colours and the inequality~(\ref{eqn.detchi}) follows.
\end{proof}


\subsection{Some lemmas on $\varphi \in \Fcal$ and $\fcol(t)$}
\label{subsec.techlemphi}

We now give some lemmas on feasible functions $\varphi\in\Fcal$ 
and on the function $\fcol(t) = \sup_{\varphi\in{\cal F}} \xi(\varphi,t)$, 
needed for the proofs in the following subsections.

\begin{lemma}\label{prop:densesup}
  $\sup_{\varphi\in{\cal F}} \int_{\eR^d} \varphi(x) dx = \volhalfBdelta$. 
\end{lemma}

\begin{proof}
  First note that the function $\varphi_K$ which has the value $\frac{1}{N(2K)}$ on the
  hypercube $(0,K)^d$ and 0 elsewhere is feasible, giving that
\[
  \sup_\varphi \int_{\eR^d} \varphi(x) dx \geq \lim_{K\to\infty} \frac{K^d}{N(2K)}
  = \frac{\vol(B)}{2^d\delta},  
\]
  by definition of the packing constant $\delta$.
  On the other hand, let $\varphi$ be an arbitrary feasible function.
  Let $A \subseteq (0,K)^d$ with $\tel{A} = N(2K)$ be a set of points satisfying
  $\norm{a - b} > 1$ for all $a\neq b\in A$.
  If $\eta$ is a constant such that $\norm{a-b} > 1$ whenever $\tel{(a)_i-(b)_i} > \eta$
  for some $1\leq i\leq d$, then the set
\[
  B := A + (K+\eta)\Zed^d \;\; (=\{a+(K+\eta)z:a\in A, z\in\Zed^d\})
\]
  also satisfies the condition that $\norm{a - b} > 1$ for all $a\neq b\in B$.
  Set $\psi(x) := \sum_{b \in B} \varphi( b+x )$.
  Since $\varphi$ is feasible we must have $\psi(x) \leq 1$ for all $x$.
  For $a \in A$ let us denote by $B_a$ the ``coset''
  $a + (K+\eta)\Zed^d \subseteq B$, and let us set
  $\psi_a(x) := \sum_{b \in B_a} \varphi(b+x)$.  We have that
\[
  (K+\eta)^d \geq \int_{[0,K+\eta)^d} \psi(x) dx = \sum_{a\in A} \int_{[0,K+\eta)^d} \psi_a(x) dx =
  N(2K) \int_{\eR^d} \varphi(x) dx,
\]
  where the last equality follows because
\[
  \int_{[0,K+\eta)^d}\psi_a(x)dx = \sum_{b\in B_a} \int_{[0,K+\eta)^d} \varphi(b+x) dx
  = \sum_{b\in B_a} \int_{b+[0,K+\eta)^d} \varphi(x)dx
\] 
  and the sets $b+[0,K+\eta)^d$ with $b \in B_a$ form a dissection of $\eR^d$.
  Thus we see that indeed for any feasible $\varphi$
\[ 
  \int_{\eR^d} \varphi(x) dx \leq \lim_{K\to\infty}\frac{(K+\eta)^d}{N(2K)}
  = \frac{\vol(B)}{2^d\delta}, 
\]
  as required.
\end{proof}

\noindent
  From Lemma~\ref{prop:densesup} we may conclude:

\begin{lemma} \label{cor:chilb}
  Let $w= \volhalfBdelta$.  Then 
  $w \leq \sup_{\varphi\in\Fcal} \xi(\varphi,t) \leq c(w,t)$ for all $t \in (0,\infty]$. 
\end{lemma}
  \begin{proof}
  The lower bound follows from Lemma~\ref{prop:densesup} and the fact that
  $\xi(\varphi,t) \geq \int\varphi$ (as $s\geq 0$) for all $\varphi$.
  The upper bound follows from Lemma~\ref{prop:densesup} together
  with~\eqref{eq:xiindicatorrewrite} and part~\ref{itm:phiindicatorbiggest} of Lemma~\ref{lem:xibasic}
  (if $\varphi\in{\cal F}$ and $W \subseteq \eR^d$ has $\vol(W)= \volhalfBdelta$ 
  then $\int\varphi 1_{\{\varphi\geq a\}} \leq \int 1_W1_{\{1_W\geq a\}}$
  for all $a\in\eR$ so that $\xi(\varphi,t) \leq \xi(1_W,t)$).
  \end{proof}

\noindent
Together with observation~\eqref{eq:cwtinf} from section~\ref{subsec.explicit}, 
Lemma~\ref{cor:chilb} implies:

\begin{lemma}\label{lem:fcolliminf}
$\lim_{t\to\infty}\fcol(t) = \volhalfBdelta$. \noproof
\end{lemma}

\noindent
Moreover, since $\varphi_0 \in \Fcal$ so that $\fcol(t) = \sup_{\varphi\in\Fcal}\xi(\varphi,t)
\geq \xi(\varphi_0,t) = \fcli(t) = c(w,t)$ with $w=\volhalfB$, observation~\eqref{eq:cwtnul} gives:

\begin{lemma}\label{lem:fcollimnul}
$\lim_{t\downto 0} \fcol(t) = \infty$. \noproof
\end{lemma}

\noindent
Since each $\xi(\varphi,t)$ is non-increasing in $t$ for each $\varphi$ separately, we also have:

\begin{lemma}\label{lem:fcolnonincr}
$\fcol(t)$ is non-increasing.\noproof
\end{lemma}

\noindent
Observe that, by part~\ref{itm:xitplush} of Lemma~\ref{lem:xibasic}, for any $h>0$:
\[ 
  (\frac{t}{t+h}) \sup_{\varphi\in{\cal F}}\xi(\varphi,t) \leq 
  \sup_{\varphi\in{\cal F}}\xi(\varphi,t+h) \leq
  \sup_{\varphi\in{\cal F}}\xi(\varphi,t).
\]
\noindent
Thus: 

\begin{lemma}\label{lem:fcolcont}
$\fcol(t)$ is continuous in $t$. \noproof
\end{lemma}

\noindent
We shall also need the following two technical lemmas.

\begin{lemma} \label{lem.f*f}
   Let ${\cal F}^*$ be the collection of all $\varphi \in{\cal F}$ that are tidy.
   For each $0<t \leq \infty$
\[
  \sup_{\varphi\in{\cal F}^*} \xi(\varphi, t ) \; = \; \sup_{\varphi\in{\cal F}} \xi(\varphi, t )
\]
\end{lemma}

\begin{proof}
  If $t=\infty$ then by Lemma~\ref{prop:densesup} and the proof of the lower bound in it,
  both sides of the equation above equal $\volhalfBdelta$. 
  Thus we may suppose that $0<t<\infty$.
  Let $\varphi \in {\cal F}$.  It suffices to show that
\begin{equation} \label{eqn.fstar}
  \xi(\varphi,t) \leq \sup_{\psi \in {\cal F}^*} \xi(\psi,t).
\end{equation}
  We may assume that $\varphi$ has bounded support, because (by Lemma~\ref{lem:xibasic}, part~\ref{itm:phiseq})
  the sequence of functions $(\varphi_n)_n$ given by $\varphi_n = \varphi 1_{[-n,n]^d}$ satisfies
  $\lim_{n\to\infty} \xi(\varphi_n,t) = \xi(\varphi,t)$.
  Let $\eps>0$ and for each $q \in \eps \Zed^d$ let $C^q := q + [0,\eps)^d$ as before.
  Define the function $\hat{\varphi}$ on $\eR^d$ by setting
  $\hat{\varphi}(x) = \sup_{y \in C^{p(x)}} \varphi(y)$ (where again $p(x)$ is the unique
  $q\in\eps\Zed^d$ such that $x\in q+[0,\eps)^d$).
  Clearly $\hat{\varphi} \geq \varphi$.
  Although $\hat{\varphi}$ is not necessarily feasible, the function $\varphi'$ given by
  $\varphi'(x) = \hat{\varphi}((1+\eps\rho)x)$ is.
  Also, $\varphi'$ is tidy: clearly it is measurable, bounded, nonnnegative
  and has bounded support.
  That the set $\{\varphi'>a\}$ has a small neighbourhood
  for all $a>0$, follows from the fact that it is the union of finitely many hypercubes $(1+\eps\rho)^{-1} C^q$.
  So $\varphi' \in {\cal F}^*$. We find that
\[
  \xi(\varphi,t) \leq \xi( \hat{\varphi}, t) \leq (1+\eps\rho)^d \xi(\varphi', t)
  \leq (1+\eps\rho)^d \sup_{\psi \in {\cal F}^*} \xi(\psi, t),
\]
  using Lemma~\ref{lem:xibasic}, parts~\ref{itm:phipsi} and~\ref{itm:philambda},
  for the first and second inequalities respectively. Now we may send $\eps \to 0$ to conclude the proof
  of inequality~(\ref{eqn.fstar}), and thus of the lemma.
\end{proof}

\begin{lemma} \label{lem.psi-lb}
  Let $0<\tau<\infty$ and let $\eps>0$.
  Then there exist $m$ and tidy functions $\psi_1,\ldots,\psi_m$ in $\cal F$ such that
\begin{equation} \label{eqn.psi-lb}
  \max_i \xi(\psi_i,t) \geq (1-\eps) \sup_{\varphi \in {\cal F}} \xi(\varphi,t) \;\; \mbox{ for all } t \in [\tau,\infty].
\end{equation}
\end{lemma}

\begin{proof}
  Recall that each $\xi(\varphi,t)$ is non-increasing as a function of $t$,
  and thus so is $\fcol(t)= \sup_{\varphi \in {\cal F}} \xi(\varphi,t)$.
  Also $\fcol(\infty)>0$.
  Let $m$ and $\tau=\tau_0<\tau_1< \cdots < \tau_{m-1}< \tau_{m}=\infty$ be such that
  $\fcol(\tau_i)/\fcol(\tau_{i+1})< 1+\eps$ for each $i=0,1,\ldots,m-1$. 
  Let $\psi_1,\ldots,\psi_{m}$ be tidy functions in $\cal F$ such that $\xi(\psi_i,\tau_i) \geq (1-\eps) \fcol(\tau_i)$
  for each $i=1,\ldots,m$.
  If $\tau_i \leq t \leq \tau_{i+1}$ then
\[ 
  \xi(\psi_{i+1},t) \geq \xi(\psi_{i+1}, \tau_{i+1}) \geq (1-\eps) \fcol(\tau_{i+1}) \geq (1-\eps)^2 \fcol(\tau_i) \geq (1-\eps)^2 \fcol(t).
\]
  and the lemma follows.
\end{proof}


\subsection{Completing the proofs of parts~\ref{itm:tobiaschi.intermediate}
and~\ref{itm:tobiaschi.dense} of Theorem~\ref{thm.tobiaschi}} 
\label{subsec.thmcolinchi-end}

Lemmas~\ref{lem:fcolliminf},~\ref{lem:fcollimnul},~\ref{lem:fcolnonincr} and~\ref{lem:fcolcont}
show that $\fcol$ has the properties claimed in part~\ref{itm:tobiaschi.intermediate}
of Theorem~\ref{thm.tobiaschi}.

We shall now prove the following theorem, which implies
parts~\ref{itm:tobiaschi.intermediate}
and~\ref{itm:tobiaschi.dense} of Theorem~\ref{thm.tobiaschi}
and is also convenient for the proof of Theorems~\ref{thm.tobiasfrat} and~\ref{thm.chif} later on.

\begin{theorem}\label{thm.tobiaschireal}
If $t(n) := \frac{\sigma n r^d}{\ln n}$ satisfies $\liminf_n t(n) > 0$ then
  
\[ \frac{\chi(G_n)}{\sigma n r^d} \sim \fcol(t(n)) \;\; \text{ a.s. } \]
  
\end{theorem}

\begin{proof}
It suffices to prove the following.
Assume that $\liminf_n t(n) > 0$ where $t(n) := \frac{\sigma n r^d}{\ln n}$, and let $\eps>0$.  
We shall show that a.s.

\begin{equation} \label{target.chi-bdd}
  1-\eps < \liminf_n \frac{\chi_f(G_n)}{\sigma n r^d \fcol(t(n))}
  \leq \limsup_n \frac{\chi(G_n)}{\sigma n r^d \fcol(t(n))} < 1+\eps . 
\end{equation}

  For the lower bound in~(\ref{target.chi-bdd}), let $\psi_1,\ldots,\psi_m$ be as in Lemma~\ref{lem.psi-lb}.  
Then, using Lemma~\ref{lem.chifg}, Theorem~\ref{thm.Mphi} and Lemma~\ref{lem.psi-lb},
\begin{eqnarray*}
  \liminf_n \frac{\chi_f(G_n)}{\sigma n r^d \fcol(t(n))}
  & \geq & 
  \liminf_n \max_i \frac{M_{\psi_i}}{\sigma n r^d \xi(\psi_i,t(n))} \frac{\xi(\psi_i,t(n))}{\fcol(t(n))}\\
  & \geq &
  \liminf_n \max_i \frac{\xi(\psi_i,t(n))}{\fcol(t(n))} \; \geq \; 1-\eps \; \mbox{ a.s.}
\end{eqnarray*}
  For the upper bound, observe that for $\varphi_1,\ldots,\varphi_m$ and $c$ as in Lemma~\ref{lem.detchi2}, by rescaling we obtain
\[
  \chi(G_n) = \chi(G( r^{-1}X_1,\ldots,r^{-1}X_n ;1)) \leq (1+\eps) \max_i M_{\varphi_i} +c.
\]
  Hence
\[
  \limsup_n \frac{\chi(G_n)}{\sigma n r^d \fcol(t(n))}
  \leq (1+\eps)  \limsup_n \max_i \frac{M_{\varphi_i}}{\sigma n r^d \xi(\varphi_i,t(n))} 
  \leq 1+\eps \; \mbox{ a.s.}
\]
  This completes the proof of~(\ref{target.chi-bdd}), and hence of the Theorem.
 \end{proof}


\section{Proof of Theorem~\ref{thm.frat}}
\label{sec.x(t)}

\noindent
Recall that, since $\varphi_0 = 1_{B(0;\frac12)} \in \Fcal$ is feasible, and using~\eqref{eqn.omega-def} we have 

\[ \fcol(t) = \sup_{\varphi\in\Fcal}\xi(\varphi,t) \geq \xi(\varphi_0,t) = c(w, t) = \fcli(t), \] 

\noindent
where $w = \volhalfB$. Thus Lemma~\ref{cor:chilb} implies part~\ref{itm:thm.frat.i} of Theorem~\ref{thm.frat}.

In the remainder of the Section~\ref{sec.x(t)} we shall thus assume that $\delta < 1$.
Let us set 

\begin{equation}\label{eq:t0def}
  t_0 := \inf\{ t>0 : \frat(t)\neq 1 \}.
\end{equation}
  
To achieve our target for this section, 
we still need to show that $0<t_0<\infty$ and that $\frat(t)$ is strictly increasing for $t\geq t_0$.
We do this in the next subsections.
  


\subsection{Proof that $0 < t_0 < \infty$}  \label{sec:t0}

Note first that $\xi(\varphi_0,t) >0$ for all $t>0$, and by part~\ref{itm:xitplush}
  in Lemma~\ref{lem:xibasic} $\xi(\varphi_0,t)$ is continuous in~$t$.
  Together with Lemma~\ref{lem:fcolcont} this shows that $\frat$ is continuous too.
  
Since we have already established that $\lim_{t\to\infty} \fcol(t) = \volhalfBdelta$ 
and $\lim_{t\to\infty}\fcli(t) = \volhalfB$, we have $\lim_{t\to\infty}\frat(t) = \frac{1}{\delta}$.
This implies that $t_0 < \infty$ and it therefore only remains to show that $t_0 >0$.
  
  Let us first give a quick overview of the proof of this section.
  We first show that $\xi( 1_{B(0;3)}, t ) < 2\xi(\varphi_0,t)$
  holds for all sufficiently small $t>0$.  We then prove that if $t$ satisfies 
$\xi( 1_{B(0;3)}, t ) < 2\xi(\varphi_0,t)$
then $\frat(t)=1$.
  To do this we assume that $\frat(t) >1$, so that there is a feasible function $\psi$ with
  $\xi(\psi,t) > \xi(\varphi_0,t)$. 
  We show that there must then be such a function which is simple and has support contained in $B(0,2)$;
  then we use a convexity argument to replace $\psi$ by a function which takes values only in
  $\{0, \frac12, 1\}$; and then that there must be a function $\varphi_{\beta}$ satisfying
  $\xi(\varphi_{\beta},t) > \xi(\varphi_0,t)$ where $\varphi_{\beta}$ takes a particularly easy form,
  so that finally we can prove analytically that this cannot happen.
  
  Now we proceed to fill in the details.
  \smallskip

\begin{lemma}\label{lem:WW}
There is a $T>0$ such that 

\begin{equation}\label{eq:xi100xi0}
\xi( 1_{B(0;3)}, t ) < 2\xi(\varphi_0,t),
\end{equation}

\noindent
for all $0 < t \leq T$.
\end{lemma}

\begin{proof}
We shall prove the stronger statement that $\xi( 1_{B(0;3)}, t ) / \xi( \varphi_0, t ) \to 1$ as
$t \downto 0$. (Observe that this implies the lemma.)

Pick $W \in \{ B(0;3), B(0;1/2) \}$, and set 
$w = \vol(W)$.
Recall from section~\ref{subsec.explicit} that 
$\xi(1_W,t) = c(w,t)$ where $c(w,t) \geq w$
solves $w H( c/w ) = 1/t$.
Clearly

\begin{equation}\label{eq:WW1eq} 
c(w,t) \to \infty \;\; \mbox{ as } t\downto 0. 
\end{equation}

\noindent
Writing out the expression for $w H(c/w)$ we find

\begin{equation}\label{eq:WW2eq}
w H(c/w) = c(w,t)\ln c(w,t) - c(w,t)\ln w - c(w,t) + w. 
\end{equation}

\noindent
Thus, combining~\eqref{eq:WW1eq} and~\eqref{eq:WW2eq} we see that 

\[
\frac{1}{t} = w H( c(w,t) / w ) 
= (1+o(1)) c(w,t) \ln c(w,t) \;\;\;\; \mbox{ as } t \downto 0.
\]

\noindent
But this gives

\begin{equation}\label{eq:WWeq}
\xi( 1_W, t ) = (1+o(1)) \left(\frac{1}{t}\right) / \ln\left(\frac{1}{t}\right) \;\;\;\; \mbox{ as } t \downto 0.
\end{equation}

\noindent
Since~\eqref{eq:WWeq} is both true for $W = B(0;3)$ and $W=B(0;1/2)$, 
we see that 
$\xi( 1_{B(0;3)}, t ) / \xi( \varphi_0, t ) \to 1$ as
$t \downto 0$, which concludes the proof.
\end{proof}

\begin{lemma}\label{lem:island}
  If $t > 0$ satisfies~(\ref{eq:xi100xi0}) and $\frac{\sigma nr^d}{\ln n} \to t$ as $n \to \infty$,
  then a.a.a.s.~there exists a subgraph $H_n$ of $G_n$ induced by the points
  in some ball of radius $2r$ such that $\chi(G_n)=\chi(H_n)$. 
\end{lemma}

\begin{proof}
  Suppose that $t$ satisfies~(\ref{eq:xi100xi0}) and $\frac{\sigma nr^d}{\ln n} \to t$
  as $n \to \infty$.   Let $W = B(0;3)$. 
  By Theorem~\ref{thm.Mphi} and Theorem~\ref{thm.tobiasomegareal} we have
\[ 
  \frac{M_W}{\sigma nr^d} \to \xi(1_W,t) \;\; \mbox{ and } \;\; \frac{\omega(G_n)}{\sigma n r^d} \to \xi(\varphi_0,t) \;\; \textrm{ a.s.}
 \]
  and so $M_W < 2 \omega(G_n)$ a.a.a.s.
  It is convenient to identify vertex $i$ of $G_n$ with $X_i$.
  Note that if two vertices with $2r \leq \norm{X_i-X_j} \leq 3r$ have degrees $\geq \omega(G_n)$ then
  there must exist a translate of $rW$ (centred at $\frac12 (X_i +X_j)$)
  containing at least $2\omega(G_n)+2$ points, and so
  the condition $M_W < 2 \omega(G_n)$ fails.  
  Hence a.a.a.s. any two vertices $X_i$ and $X_j$ of $G_n$ with degrees $\geq \omega(G_n)$ satisfy
  either $\norm{X_i-X_j} < 2r$ or $\norm{X_i-X_j} > 3r$.
  But if we remove from $G_n$ all vertices which have degree 
  at most $\omega(G_n)-1$ then the chromatic number does not change, and by the above we will
  be left a.a.a.s. with a graph in which each component is contained in some ball
  of radius $2r$.  This completes the proof.
\end{proof}

  We will show that $t_0 \geq T$ (with $T$ as in Lemma~\ref{lem:WW}) by proving that
  $\sup_{\varphi \in {\cal F}} \xi(\varphi,t) =  \xi(\varphi_0,t)$  for all $t$ that
  satisfy~\eqref{eq:xi100xi0}.
  The following purely deterministic lemma
  perhaps surprisingly has a convenient probabilistic proof.

\begin{lemma}\label{lem:supsmallsupportiftsmall}
  Let $t>0$ satisfy~\eqref{eq:xi100xi0}. Then
\[
  \sup_{\varphi \in {\cal F}} \xi(\varphi,t) =
  \sup_{\varphi \in {\cal F}, \atop \supp(\varphi)\subseteq B(0;2)}\xi(\varphi,t).
\]
\end{lemma}

\begin{proof}
Let $r$ satisfy $\sigma nr^d\sim t\ln n$ and consider $\chi(G_n)$.
By Lemma~\ref{lem:island} a.a.a.s.~$\chi(G_n)$ equals the maximum over all
$x\in\eR^d$ of the chromatic number of the graph induced by the
vertices in $B(x;2r)$. Let us fix an $\eps>0$. Let us denote $V :=
r^{-1}\{X_1,\dots,X_n\}$ and let $\Gamma,\Gamma^p, \Gamma_V,
\Gamma_V^p$ be as in the proof of Lemma~\ref{lem.detchi2}.
For $p\in\eps\Zed^d$ let $\Lambda^p$
denote the subgraph of $\Gamma$ induced by the points of
$\eps\Zed^d$ inside $B(p,(2+\eps\rho))$, where again
$\rho:=\diam([0,1]^d)$, and let $\Lambda^p_V$ be the corresponding
subgraph of $\Gamma^p_V$.
  Since for every $x\in\eR^d$ the subgraph of $G_n$ 
  induced by the vertices inside $B(x,2r)$ is a subgraph of some $\Lambda^p_V$, we have:
\begin{equation}\label{eq:chismsuppeq}
  \chi(G_n) \leq \max_p \chi(\Lambda_V^p) \leq \max_{i=1,\dots,m}
  M_{\varphi_i} + c \text{ a.a.a.s., }
\end{equation}
  where $\varphi_1,\dots,\varphi_m$ are obtained from the
  ILP formulation of $\chi(\Lambda_V^p)$ via the same procedure we
  used in the proof of 
  Lemma~\ref{lem.detchi2} (that is, the upper
  bound in~\eqref{eq:chismsuppeq} is the analogue of the upper bound
in~\eqref{eqn.bounds}) and $c$ is a constant that depends
only on $\eps, d$ and $\norm{.}$. 
By construction we have that $\supp(\varphi_i)
\subseteq B(0, 2+2\eps\rho)$ and that $\varphi_i'$ given by
$\varphi_i'(x) = \varphi_i( (1+2\eps\rho)x )$ is feasible. Notice
that $\varphi_i'$ also satisfies $\supp(\varphi_i')\subseteq
B(0;2)$. Thus,~\eqref{eq:chismsuppeq} together with
Theorem~\ref{thm.tobiaschireal}, Theorem~\ref{thm.Mphi} and
part~\ref{itm:philambda} of Lemma~\ref{lem:xibasic}
shows that
\[
  \sup_{\varphi\in{\cal F}} \xi(\varphi,t)
  \leq \max_{i=1,\dots,m} \xi(\varphi_i,t)\leq
  \max_{i=1,\dots,m} (1+2\eps\rho)^d\xi(\varphi_i',t)
  \leq  (1+2\eps\rho)^d\sup_{\varphi\in{\cal F},\atop \supp(\varphi)\subseteq B(0;2)} \xi(\varphi,t).
\]
  The statement now follows by
letting $\eps\to 0$.%
\end{proof}

\noindent
Let us now fix a $t>0$ that satisfies~\eqref{eq:xi100xi0}.
If $\sup_{\varphi\in{\cal F}} \xi(\varphi, t) > \xi(\varphi_0,t)$ then there must
also exist a feasible simple function $\psi := \sum_{k=1}^m \frac{k}{m} 1_{A_k}$
with $\supp(\psi) \subseteq B(0;2)$ such that
$\xi(\psi,t) > \xi(\varphi_0,t)$, because (by Lemma~\ref{lem:xibasic},
item~\ref{itm:phiseq}) for any $\varphi$ the increasing sequence of functions $(\varphi_n)_n$
given by
$\varphi_n = \sum_{k=1}^{2^n} \frac{k}{2^n} 1_{\{\frac{k}{2^n}\leq \varphi < \frac{k+1}{2^n}\}}$
satisfies $\lim_{n\to\infty} \xi(\varphi_n,t) = \xi(\varphi,t)$.

So let $\psi = \sum_{i=1}^m \frac{i}{m} 1_{A_i}$ be a feasible simple function with $\xi(\psi,t) > \xi(\varphi_0,t)$ and $\supp(\psi)\subseteq B(0;2)$.
We may suppose wlog that the $A_k$ are disjoint and $m$ is even.
For $1\leq k\leq \frac{m}{2}$ let
$\psi_k$ be the function which is $\frac{1}{2}$ on
$\bigcup_{i=k}^{m-k} A_i$ and 1 on $\bigcup_{i>m-k} A_i$.
  We can write
\[
  \psi = \frac{2}{m}\sum_{k=1}^{m/2} \psi_k,
\]
  because for $x \in A_i$ with $i \leq m/2$ we have
  $\frac{2}{m} \sum_{k=1}^{m/2} \psi_k(x) =
  i\frac{2}{m} \frac{1}{2} = \frac{i}{m}$,
  and if $x \in A_{m-i}$ with $i \leq m/2$ then
  $\frac{2}{m} \sum_{k=1}^{m/2} \varphi_k(x) = 1 - i \frac{2}{m} \frac{1}{2} = \frac{m-i}{m}$.

  Let us now observe that $\xi$ is convex in its first argument; that is, for any two
  nonnegative, bounded, measurable functions $\sigma, \tau$ and any $t>0$ and $\lambda \in [0,1]$ we have
\[
\xi( \lambda\sigma+(1-\lambda)\tau, t ) \leq
\lambda \xi(\sigma, t) + (1-\lambda)\xi(\tau, t).
\]
  This follows from parts~\ref{itm:xiscalar} and \ref{itm:xisum} of Lemma~\ref{lem:xibasic}.
  Because we have written $\psi$ as a convex combination of the
  $\psi_k$, we must therefore have
  $\xi(\psi, t) \leq \xi(\psi_k, t)$ for some $k$.

Let us first assume that $\{\psi_k=1\} = \bigcup_{l>m-k} A_k = \emptyset$.
Since $\supp(\psi) \subseteq B(0;2)$ we must have
that $\psi_k \leq \varphi'$, where $\varphi'$ is the function which is $\frac{1}{2}$ on $B(0;3)$ and 0 elsewhere, and thus also
$\xi(\psi,t) \leq \xi(\psi_k, t) \leq \xi(\varphi',t) = \frac{1}{2} \xi( 1_{B(0;3)}, t )$
(by choice of $k$ and Lemma~\ref{lem:xibasic}, items~\ref{itm:phipsi} and~\ref{itm:xiscalar}).
But then~\eqref{eq:xi100xi0} gives:
\[
\xi( \psi,t) 
\leq \frac12 \xi( 1_{B(0;3)}, t) < \xi(\varphi_0,t),
\]
  a contradiction.

\noindent
  So we must have $\{\psi_k=1\}\neq \emptyset$.
  Let us denote by $C := \clo(B)$ the closed unit ball.
  Notice that 
\[ \begin{array}{c}
\displaystyle
\diam(\{\psi_k=1\}) \leq 1, \\
 \\
\displaystyle
\supp(\psi_k) \subseteq \bigcap_{x : \psi_k(x)=1 } (x+C), 
\end{array}
\]
  by feasibility of $\psi$
  (if $x\in \supp(\psi_k)$ and $\psi_k(y)=1$ then
$\psi(x)+\psi(y) > \frac{k}{m} + \frac{m-k}{m} = 1$).
  Bieberbach's inequality (Lemma~\ref{thm:bieberbach})
  tells us that $\vol(\{\psi_k=1\})$ cannot exceed $\volhalfB$, 
  the volume of a ball of diameter 1.
  Hence there is a $0\leq\beta\leq 1$ with
  $\vol(\{\psi_k=1\}) = \vol(B(0;\frac{1-\beta}{2})) = (1-\beta)^d \volhalfB$.
  We will need another inequality, given by the following lemma.

\begin{lemma}[K.~B\"or\"oczky~Jr., 2005] \label{prop:karcsi}
  Let $C\subseteq \eR^d$ be a compact, convex set.
  Let $A \subseteq \eR^d$ be measurable and let $A'$ be a homothet (that is, a scaled copy)
  of $-C$ with $\vol(A) = \vol(A')$. Then
\[ 
  \vol\left( \bigcap_{a \in A} (a+C) \right) \leq \vol\left( \bigcap_{a \in A'} (a+C) \right).
\]
%
\end{lemma}

\noindent
  With the kind permission of K.~B\"or\"oczky Jr.~we present a proof
  in appendix~\ref{sec:karcsi}, because such a proof is not readily available elsewhere.

  It follows from Lemma~\ref{prop:karcsi} that we must have
\[
  \vol(\supp(\psi_k)) \leq \vol\left( \bigcap_{x\in B(0;\frac{1-\beta}{2})}(x+C)\right)
  = \vol\left( B(0;\frac{1+\beta}{2})\right) = (1+\beta)^d \volhalfB.
\]
  For $0 \leq \beta \leq 1$ let $\varphi_\beta$ be the function which is 1 on $B(0;\frac{1-\beta}{2})$ and
  $\frac{1}{2}$ on  $B(0;\frac{1+\beta}{2})\setminus B(0;\frac{1-\beta}{2})$.
  (This agrees with our earlier definition of $\varphi_0$.)
  We see that $\vol(\{\psi_k=1\}) = \vol(\{\varphi_\beta = 1\})$ and
  $\vol( \{\psi_k = \frac{1}{2}\} ) \leq \vol(\{\varphi_\beta=\frac{1}{2}\})$.
  Thus we have
  $\int \psi_k 1_{\{\psi_k\geq a\}} \leq \int \varphi_\beta 1_{\{\varphi_\beta\geq a\}}$
  for all $a$, which gives $\xi(\psi_k, t) \leq \xi(\varphi_\beta, t)$
  by part~\ref{itm:phiindicatorbiggest} of Lemma~\ref{lem:xibasic}.
  We may conclude that if
  $\sup_{\varphi\in{\cal F}} \xi(\varphi, t) > \xi(\varphi_0,t)$
  and~\eqref{eq:xi100xi0} holds then also
  $\xi(\varphi_\beta, t) > \xi(\varphi_0, t)$ for some $0<\beta \leq 1$.
  We will show that this last statement is false.
  Set $\mu(\beta) := \xi(\varphi_\beta, t)$ for $0\leq \beta \leq 1$.

\begin{lemma}  \label{lem.max-mubeta}
\[
  \max_{0\leq\beta\leq 1} \mu(\beta) = \max \{ \mu(0), \mu(1) \}.
\]
\end{lemma}

\begin{proof} Notice that for $0\leq\beta\leq 1$
\[
  \mu(\beta) =
   \volhalfB ( \frac{1}{2}((1+\beta)^d-(1-\beta)^d)e^{s/2} + (1-\beta)^d e^s ),
\]
  where $s = s(\beta)$ solves
\begin{equation}\label{eq:Hs}
   \volhalfB \left( ((1+\beta)^d-(1-\beta)^d)H(e^{s/2}) + (1-\beta)^d H(e^s) \right) = \frac{1}{t}.
\end{equation}
  The function $\mu(\beta)$ is continuous on $[0,1]$.
  Differentiating equation~\eqref{eq:Hs} wrt~$\beta$ we see that for $0<\beta<1$
\[ \begin{array}{rcl}
0 & = &
d((1+\beta)^{d-1}+(1-\beta)^{d-1}) H(e^{s/2}) +
((1+\beta)^d-(1-\beta)^d)) \frac{s}{4} e^{s/2} s'
- d(1-\beta)^{d-1} H(e^s ) \\
 & & + (1-\beta)^d s e^s s'.
\end{array} \]
(That $s$ is differentiable wrt~$\beta$ can be justified using the 
implicit function theorem.)  
This gives 

\[
  s'(\beta)=\frac{d(1-\beta)^{d-1}H(e^s) - d((1+\beta)^{d-1}+(1-\beta)^{d-1})H(e^{s/2})}
  {((1+\beta)^d-(1-\beta)^d)\frac{s}{4}e^{s/2}+(1-\beta)^dse^s}.
\]
  Thus,
\[ \begin{array}{rcl}
  \volhalfB \mu'(\beta)
& = &
  d((1+\beta)^{d-1}+(1-\beta)^{d-1})\frac{1}{2}e^{s/2} - d(1-\beta)^{d-1} e^s \\
  & & \quad + \left.
  s' ( ((1+\beta)^d-(1-\beta)^d)\frac{1}{4}e^{s/2} + (1-\beta)^de^s ) \right. \\
& = &
  d((1+\beta)^{d-1}+(1-\beta)^{d-1})\frac{1}{2}e^{s/2} - d(1-\beta)^{d-1} e^s \\
  & & \left. \quad +
  \frac{1}{s} ( d(1-\beta)^{d-1} H(e^s) - d((1+\beta)^{d-1}+(1-\beta)^{d-1})H(e^{s/2})) \right.\\
& = &
  \frac{1}{s} \left[
  d((1+\beta)^{d-1}+(1-\beta)^{d-1})( e^{s/2} - 1 )
  - d(1-\beta)^{d-1}( e^s - 1 ) \right].
\end{array} \]
  Clearly $\mu'(\beta) > 0$ for $\beta$ sufficiently close to 1, so that it
  suffices to show that (for any $t$)
  $\mu'(\beta) = 0$ for no more than one $\beta \in (0,1)$.
  Note that $\mu'(\beta) = 0$ if and only if
\[
  e^s - 1 = ((\frac{1+\beta}{1-\beta})^{d-1} + 1)(e^{s/2}-1).
\]
  Writing $a := (\frac{1+\beta}{1-\beta})^{d-1} + 1$ and $x:=e^{s/2}$ this translates into
  the quadratic $x^2 - a x + (a-1) = 0$, which has roots $1, a-1$.
  Now notice that $e^{s/2} = 1$ would give $s(\beta) = 0$, but this is never a solution of \eqref{eq:Hs}.
  So if $\mu'(\beta) = 0$ for some $0<\beta<1$ then we must have
  $s(\beta) = 2(d-1)\ln(\frac{1+\beta}{1-\beta})$.
  Notice that, as $s$ cannot equal 0, this also shows that we must
  have $d\geq 2$ for $\mu'(\beta) = 0$ to hold.
  The curve $u(\beta) := 2(d-1)\ln(\frac{1+\beta}{1-\beta})$ has derivative
\[
  u'(\beta) = \frac{4(d-1)}{(1-\beta)(1+\beta)}.
\]
  On the other hand, for $0<\beta<1$ 
\[
  s'(\beta) < \frac{d(1-\beta)^{d-1}H(e^s)}{(1-\beta)^dse^s}<\frac{d}{1-\beta}.
\]
  We find $s'(\beta) < \frac{d}{1-\beta} < 4(d-1) / (1+\beta)(1-\beta) = u'(\beta)$
  for $0<\beta< 1$ (recall $d\geq 2$ by a previous remark).
  We may conclude that the curves $u(\beta)$ and $s(\beta)$ meet in at most one point,
  as $u(\beta)-s(\beta)$ is strictly increasing on $(0,1)$.
  In other words, there is at most one $\beta \in (0,1)$ with $\mu'(\beta) = 0$
  and the lemma follows.
\end{proof}

  Since we had chosen $t$ so that~\eqref{eq:xi100xi0} holds,
  parts~\ref{itm:philambda} and~\ref{itm:phipsi} of Lemma~\ref{lem:xibasic}
  tell us that:
\[
\xi(\varphi_1,t) = \frac12\xi( 1_{B}, t ) \leq 
 \frac{1}{2}\xi(1_{B(0;3)},t) < \xi(\varphi_0,t).
\]
  Thus 
  
\[ \max_{0\leq\beta\leq 1}\mu(\beta) = \mu(0)=\xi(\varphi_0,t). \]
  
  Observe that we have now achieved our aim in this subsection, namely to show that $t_0>0$.
  Thus we have proved the following lemma.

\begin{lemma}\label{lem:t0pos}
  When the packing constant $\delta<1$, we have $0 < t_0 < \infty$.
\end{lemma}


\subsection{When $\delta<1$ the function $\frat(t)$ is strictly increasing for $t\geq t_0$}

  In this subsection we shall prove the result just stated in the heading.
  The proof uses the following lemma. 

\begin{lemma}\label{lem:supismax}
  For each $t>0$, either $\sup_{\varphi \in {\cal F}} \xi(\varphi,t) = \volhalfBdelta$ 
  or the supremum is attained.
\end{lemma}

\begin{proof} 
  Let us assume $\sup_{\varphi \in {\cal F}} \xi(\varphi,t) > \volhalfBdelta$ 
  (as otherwise there is nothing to prove).
  Let us consider a sequence $\varphi_1, \varphi_2, \ldots \in {\cal F}$ such that
  
\begin{equation}\label{eq:limeq1}
\lim_{n\to\infty} \xi(\varphi_n, t ) =
\sup_{\varphi \in {\cal F}} \xi(\varphi,t),
\end{equation}

\noindent
  and let us suppose (wlog) that 

\begin{equation}\label{eq:subjto}
\lim_n \int_{B} \varphi_n \text{ exists and is as large as possible
  subject to~\eqref{eq:limeq1}}
\end{equation}

\noindent
(recall that $B = B(0;1)$ is the unit ball).
  We will first exhibit a subsequence $\varphi_{n_1},\varphi_{n_2},\dots$ of $(\varphi_n)_n$
  and a function $\psi \in {\cal F}$ such that
\begin{equation}\label{eq:limvarphinkleqpsi}
\limsup_{k\to\infty} \varphi_{n_k}(x) \leq \psi(x)
\text{ for all } x\in \eR^d.
\end{equation}
  It will need further work to show that the function $\psi$ achieves the supremum as required.

  In order to construct $\psi$ and the subsequence
  $(\varphi_{n_k})_k$, let ${\cal D}_k$ be the dissection
  $\{ i+[0,2^{-k})^d : i=(i_1,\dots,i_d)\in 2^{-k}\Zed^d \}$
  of $\eR^d$ into cubes of side $2^{-k}$ (observe that ${\cal D}_{k+1}$
  refines ${\cal D}_k$).
  For $\sigma \in \cal F$ let us define the functions $\sigma^k$ by setting:
\[ 
  \sigma^k(x) := \sup_{y \in C_{x,k}} \sigma(y),
\]
  where $C_{x,k}$ is the unique cube $C \in {\cal D}_k$ with $x\in C$.
Let us now construct a nested sequence
${\cal F}_1 \supseteq {\cal F}_2 \supseteq \dots$ of infinite subsets of
$\{\varphi_1,\varphi_2, \dots\}$ with the property that
\begin{equation}\label{eq:Fdefprop}
  \tel{\sigma^k(x)-\tau^k(x)}\leq\frac{1}{k} \text{ for all }
  x \in [-k,k)^d \text{ and all } \sigma, \tau \in {\cal F}_k.
\end{equation}
  To see that this can be done, notice that
  the behaviour of $\sigma^k$ on $[-k,k)^d$ is determined completely
  by $(\sigma^k(p_1), \dots, \sigma^k(p_{K}))$ where
  $p_1,\dots, p_{K}$ is some enumeration of $[-k,k)^d\cap 2^{-k}\Zed^d$.
  Given ${\cal F}_{k-1}$ there must be intervals $I_1, \dots, I_{K} \subseteq [0,1]$
  each of length $\frac{1}{k}$ such that the collection
  $\{ \sigma \in {\cal F}_{k-1} : \sigma^k(p_i) \in I_i \text{ for all } 1 \leq i \leq K \}$
  is infinite.  So we can take ${\cal F}_k$ to be such an infinite collection.
  Let us now pick a subsequence  $\varphi_{n_1}, \varphi_{n_2},\dots$ of
  $(\varphi_n)_n$ with $\varphi_{n_k} \in {\cal F}_k$ and
  let the function $\psi$ be defined by:
\begin{equation}\label{eq:psilimdef}
  \psi(x) := \lim_{k\to\infty}\varphi_{n_k}^k(x).
\end{equation}
  To see that this limit exists for all $x$, notice that
  $\varphi_{n_l}^{l}(x) \leq \varphi_{n_l}^{k}(x) \leq \varphi_{n_k}^{k}(x) + \frac1{k}$
  for all $l\geq k>\norm{x}_{\infty}$. Thus,
\[
\begin{array}{rcl}
\limsup_{k\to\infty} \varphi_{n_k}^k(x)
 & \leq & 
\inf_{k>\norm{x}_{\infty} } \varphi_{n_k}^k(x)+\frac{1}{k} \\
& = & 
\liminf_{k\to\infty} \varphi_{n_k}^k(x)+\frac{1}{k} \\
& = & 
\liminf_{k\to\infty} \varphi_{n_k}^k(x).
\end{array}
\]

%

  We now claim that $\psi$ and the sequence $(\varphi_{n_k})_k$ are as required
  (that is, $\psi\in{\cal F}$ and~\eqref{eq:limvarphinkleqpsi} holds).
  To see that~\eqref{eq:limvarphinkleqpsi} holds, notice that
  $\sup_{l\geq k} \varphi_{n_{l}}(x)\leq \varphi_{n_k}^k(x)+\frac{1}{k}$
  for any $x$ and any $k>\norm{x}_{\infty}$, so that
  
 \begin{equation}\label{eq:phikdomphi}
 \limsup_{k\to\infty} \varphi_{n_k}(x) \leq \lim_{k\to\infty} \varphi_{n_k}^{k}(x) = \psi(x).
 \end{equation}
  To see that $\psi \in {\cal F}$, let $S=\{s_1,\dots,s_p\}\in {\cal S}$ be finite
  (observe it suffices to show $\sum_{x\in S} \psi(x) \leq 1$ for all
  finite $S \in {\cal S}$).
  Since $\norm{s_i-s_j}>1$ for all $i\neq j$, there is a $k_0$
  such that $\norm{s_i-s_j}>1+2^{-k_0}\rho$
  for all $i\neq j$ where $\rho :=  \diam([0,1]^d)$.
  Thus if $k \geq k_0$ then
\[ 
  \varphi_{n_k}^k(s_1)+\dots+\varphi_{n_k}^k(s_p) \leq 1,
\]
  and hence the same must hold for $\psi$.

  Also notice that the dominated convergence theorem (using $\psi, \varphi_{n_k} \leq 1$) gives that
\begin{equation}\label{eq:psiintislimint}
  \int_{B} \psi(x){\dd}x =
  \lim_{k\to\infty} \int_{B} \varphi_{n_k}^k(x){\dd}x
  \geq \lim_{n\to\infty} \int_{B} \varphi_n(x) {\dd}x,
\end{equation}

\noindent
using~\eqref{eq:phikdomphi} for the second equation.
  Furthermore, for any fixed $R > 0$ we have that
\begin{equation}\label{eq:Rconveq}
  \lim_{n\to\infty} \xi( \varphi_{n_k}^k 1_{B(0;R)}, t )
  =  \xi( \psi 1_{B(0,R)}, t ) \leq \xi( \psi, t ).
\end{equation}
  Here we have used parts~\ref{itm:phipsi} 
  and~\ref{itm:toinf} of Lemma~\ref{lem:xibasic}.
  Hence, there also is a sequence $(R_k)_k$ with 
  $R_k$ tending to infinity and
  
  \begin{equation}\label{eq:xiphikdomxiphi}
  \limsup_{k\to\infty} \xi(\varphi_{n_k}^k 1_{B(0;R_k)},t) \leq \xi(\psi,t).
  \end{equation}
  
  \noindent
  To see this, notice that by~\eqref{eq:Rconveq} there exist $k_1 < k_2 < \dots$ such that
  $\xi( \varphi_{n_k}^k 1_{B(0;m)}, t ) \leq \xi(\psi,t) + \frac{1}{m}$
  for all $k\geq k_m$. Thus, we may put $R_k := \max\{ m : k_m\leq k\}$.  
  
  Let us put
\begin{equation} \label{eqn.inout-def}
  \psi_{k,i} := \varphi_{n_k} 1_{B(0;R_k)}, \quad
  \psi_{k,o} := \varphi_{n_k} 1_{\eR^d\setminus B(0;R_k+1)}, \quad
  \psi_{k} := \psi_{k,i} + \psi_{k,o}.
\end{equation}
  We may assume wlog that $R_k$ has been chosen in such a way that
  $\xi( \psi_k, t ) = (1+o(1)) \xi( \varphi_{n_k}, t )$.
  To see this note that for $s=s(\varphi_{n_k},t)$ there is an
  $\frac{R_k}{2} \leq R' \leq R_k$ s.t.
\[ 
  \int \varphi_{n_k} e^{s\varphi_{n_k}} 1_{B(0;R'+1)\setminus B(0;R')} \leq \frac{1}{\lfloor\frac{R_k}{2}\rfloor}\int \varphi_{n_k} e^{s\varphi_{n_k}}.
\]
  If we take such an $R'$ and set $\psi_k' := \varphi_{n_k} 1_{\eR^d \setminus B(0;R'+1) \cup B(0;R')}$
  then $s(\psi_k',t) \geq s(\varphi_{n_k},t)$ (by the definition of $s$, as $\psi_k'\leq \varphi_{n_k}$) so that
  $\xi(\psi_k',t) \geq (1-\frac{1}{\lfloor\frac{R_k}{2}\rfloor})\xi(\varphi_{n_k},t)$.

  Observe that by our choice of $R_k$
\begin{equation} \label{eqn.psik}
  \lim_{k\to\infty}\xi( \psi_k,t ) =  \lim_{k\to\infty}\xi( \varphi_{n_k},t ) = \sup_{\varphi\in{\cal F}}\xi(\varphi,t) ,
\end{equation}
  and (since $R_k \geq 1$ for $k$ sufficiently large)
\begin{equation} \label{eqn.psik2}
    \lim_{k\to\infty}\int_B \psi_k =  \lim_{k\to\infty}\int_B \varphi_{n_k} = \lim_{n \to \infty} \int_B \varphi_n.
\end{equation}

  Let us define $\lambda(\varphi) := \sup_{S\in{\cal S}} \sum_{x\in S} \varphi(x)$.
  Since $\psi_k \leq \varphi_{n_k} \in {\cal F}$ we have
\begin{equation}\label{eq:lambdaio}
\lambda(\psi_k) = \lambda(\psi_{k,i}) + \lambda(\psi_{k,o} ) \leq 1.
\end{equation}
  For convenience let us write $\lambda_k := \lambda(\psi_{k,o})$.
  First let us suppose that $\lambda_k\to 0$ as $k \to \infty$.
  Notice that $\frac{1}{\lambda_k}\psi_{k,o} \in {\cal F}$, which implies
\[
\xi( \psi_{k,o}, t ) \leq \lambda_k \sup_{\varphi\in{\cal F}} \xi(\varphi, t) = o(1),
\]
  using part~\ref{itm:xiscalar} of Lemma~\ref{lem:xibasic}. 
Observe that 
 
\[ 
\xi(\psi_{k,i},t) \leq \xi(\psi_k,t) \leq \xi(\psi_{k,i},t) + \xi(\psi_{k,o},t)
\]

\noindent
by parts~\ref{itm:phipsi} and~\ref{itm:xisum} of Lemma~\ref{lem:xibasic}.
Thus, $\xi(\psi_k,t) = \xi(\psi_{k,i},t) + o(1)$. 
Using~\eqref{eqn.psik} and~\eqref{eq:xiphikdomxiphi}
\[
  \sup_{\varphi\in{\cal F}}\xi(\varphi,t) = \lim_{k\to\infty}\xi( \psi_k, t)
\leq \xi( \psi, t ),
\]
  so that the lemma follows in the case when $\lambda_k\to 0$.
  \smallskip

  In the remaining part of the proof we shall show that in fact we must have $\lambda_k\to 0$.
  Let us assume not, and that $\limsup \lambda_k > 0$.
  We may assume for convenience that $\lim \lambda_k = \lambda > 0$
  (by considering a subsequence if necessary).
  We next establish the following claim.
  \medskip
  
  \noindent
  {\bf Claim}
\begin{equation} \label{eqn.claim-volto0} 
   \vol( \{\psi_{k,o}\geq\eps\})\to 0 \;\;\;\; \mbox{ for all } \; \eps > 0.
\end{equation}
  
\begin{pf}{of~\eqref{eqn.claim-volto0}}  
Let us construct a new sequence of functions $\psi_k'$ as follows.
  For each $k$ pick an $x_k \in \eR^d\setminus B(0;R_k+1)$
  that maximises $\int_{B(x_k;1)} \psi_{k,o}$.

  To see that such an $x_k$ exists, let us write $I(x) := \int_{B(x;1)}\psi_{k,o}$.
Notice that $I$ is continuous ($\psi_{k,o}\leq 1$ so that $\tel{I(x)-I(y)}\leq \vol( B(x;1)\setminus B(y;1) )$).
Let us suppose that
$c := \sup_{x\in\eR^d\setminus B(0;R_k+1)} I(x) > 0$, for otherwise there
is nothing to prove as any $x\in\eR^d\setminus B(0;R_k+1)$ will do.
We first claim that the set $\{ x \in \eR^d\setminus B(0;R_k+1) : I(x) > \frac{c}{2} \}$
can be covered by at most $\lfloor\frac{2\volB}{c}\rfloor$ balls of radius two.
This is because if $I(x_1), \dots, I(x_k) > \frac{c}{2}$ there must exist $y_i\in B(x_i;1)$ for
$1\leq i\leq k$ with $\psi_{k,o}(y_i) > \frac{c}{2\volB}$.
By feasibility of $\psi_{k,o}$ we must have either
$k < \frac{2\volB}{c}$ or $\norm{y_i-y_j}\leq 1$ for some
$i\neq j$.
Thus, the $y_i$ can be covered by at most $\frac{2\volB}{c}$ balls of radius one, and hence the
$x_i$ can be covered by at most $\frac{2\volB}{c}$ balls of radius two, as claimed.
As $I$ is continuous and we can restrict ourselves to a compact subset of $\eR^d\setminus B(0;R_k+1)$ we see that the supremum $c = \sup_{x\in\eR^d\setminus B(0;R_k+1)} I(x)$ is indeed attained by some point $x_k\in\eR^d\setminus B(0;R_k+1)$.

  Now let $\psi_k' := \psi_{k,i} + \psi_{k,o} \circ T_k$ where
  $T_k : y\mapsto y+{x_k}$ is the translation that sends $0$ to $x_k$.
  By~\eqref{eq:lambdaio} we have $\psi_k' \in \cal F$.
  Notice that
\[ 
  \int\psi_k'1_{\{\psi_k'\geq a\}} \geq \int\psi_k1_{\{\psi_k\geq a\}}
\]
  for all $a$, because $\{\psi_k'\geq a\} \supseteq \{\psi_{k,i}\geq a\}\cup T_k^{-1}[\{\psi_{k,o}\geq a\}]$ so that
\[ \begin{array}{rcl}
  \int \psi_k'1_{\{\psi_k'\geq a\}}
 & \geq &
  \int\psi_{k,i}1_{\{\psi_{k,i}\geq a\}} +
  \int (\psi_{k,o}\circ T_k)1_{T_k^{-1}[\{\psi_{k,o}\geq a\}]} \\
 & = &
  \int\psi_{k,i}1_{\{\psi_{k,i}\geq a\}} +
  \int\psi_{k,o}1_{\{\psi_{k,o}\geq a\}} \\
 & = &
  \int\psi_k1_{\{\psi_k\geq a\}}.
\end{array} \]
  Part~\ref{itm:phiindicatorbiggest} of Lemma~\ref{lem:xibasic} therefore gives that
\[
  \xi(\psi_k', t ) \geq \xi( \psi_k, t ) = (1+o(1))\xi(\varphi_{n_k}, t).
\]
  Thus $\xi(\psi_k', t ) \to \sup_{\varphi \in {\cal F}} \xi(\varphi,t)$ as $k \to \infty$,
  and we saw earlier that $\psi_k' \in {\cal F}$.
  We therefore must have $\int_{B(x_k;1)} \psi_{k,o} \to 0$; for otherwise,
  using~(\ref{eqn.psik2}), we have
  $\limsup_k \int_B \psi'_k > \lim_k \int_B \psi_k = \lim_n \int_B \varphi_n$,
  and (a subsequence of) the $\psi_k'$ would contradict ~\eqref{eq:subjto}.
  Now suppose that for some $\eps > 0$ we have
  $\limsup_k \vol( \{\psi_{k,o} \geq \eps \} ) = c > 0$.
  Because $\psi_{k,o}\in{\cal F}$ we can cover $\{\psi_{k,o} \geq \eps \}$ by
  at most $\lfloor\frac{1}{\eps}\rfloor$ balls of radius 1.
  But this gives
  $\limsup_k \int_{B(x_k;1)} \psi_{k,o}(x){\dd}x \geq c\eps$
  and we know this cannot happen. The claim~(\ref{eqn.claim-volto0}) follows.
  \end{pf}

  Recall that $\sigma_k := \frac{1}{\lambda_k}\psi_{k,o} \in {\cal F}$.
  Because $\lim_k \lambda_k=\lambda > 0$ the previous  also gives
  $\lim_k \vol(\{\sigma_k\geq\eps\}) = 0$ for all $\eps > 0$.
  Let us fix $\eps > 0$ for now and let
  $V_\eps, W_{k,\eps} \subseteq \eR^d$ be disjoint (measurable) sets with
  $\vol(V_\eps) = \frac{\volB}{\eps 2^d\delta}$ and
  $\vol(W_{k,\eps}) = \vol(\{\sigma_k\geq \eps\})$.
  Let us set $\tau_k := \eps 1_{V_\eps} + 1_{W_{k,\eps}}$.
  Then $\int\sigma_k1_{\{\sigma_k\geq a\}} \leq \int\tau_k1_{\{\tau_k\geq a\}}$
  for all $a$ (using $\sigma_k \leq 1$ and Lemma~\ref{prop:densesup}), so that
  parts~\ref{itm:phiindicatorbiggest},~\ref{itm:xiscalar} and~\ref{itm:xisum}
  of Lemma~\ref{lem:xibasic} give:
\[ \begin{array}{rcl}
  \xi( \sigma_k, t ) \;\; \leq \;\; \xi( \tau_k, t ) 
 & \leq & \eps \xi( 1_{V_\eps}, t ) + \xi( 1_{W_{k,\eps}}, t ) \\
 & = &
  \eps c(\vol(V_\eps),t) + c(\vol(W_{k,\eps}),t). 
\end{array} \]
  Now  
  $x=c(w,t) w$ solves $H(x) =1/(wt)$, so $x \to 1$ as $w \to 0$; and thus
  $c(w,t) \sim w$ as $w \to 0$.
  Hence $\eps c(\vol(V_\eps),t) \sim \eps \vol(V_\eps) = \volhalfBdelta$ as $\eps \to 0$;
  and for any fixed $\eps>0$
  $c(\vol(W_{k,\eps}),t) \to 0$ as $k \to \infty$.
  
  It follows that $\limsup_{k\to\infty}\xi( \sigma_k, t )\leq  \volhalfBdelta$. 
  Since $\sigma_k' := \frac{1}{1-\lambda_k}\psi_{k,i} \in {\cal F}$ by~\eqref{eq:lambdaio}, we thus have
\[
  \lim_{k\to\infty}\xi( \psi_k,t )
  = \lim_{k\to\infty} \xi(\lambda_k\sigma_k+(1-\lambda_k)\sigma_k',t) \leq \lambda  \volhalfBdelta 
  + (1-\lambda)\sup_{\varphi\in{\cal F}}\xi(\varphi,t) < \sup_{\varphi\in{\cal F}}\xi(\varphi,t),
\]
  using parts~\ref{itm:phipsi} and~\ref{itm:xiscalar} of Lemma~\ref{lem:xibasic}.
  But this contradicts equation~(\ref{eqn.psik}):
  so we must have $\lambda_k \to 0$, completing the proof of Lemma~\ref{lem:supismax}.
\end{proof}

\begin{lemma}\label{prop:xincr}
  Assume that $\delta<1$ (so $0<t_0<\infty$).
  Then the function $\frat(t)$ is continuous and strictly increasing for $t_0\leq t<\infty$.
\end{lemma}

\begin{proof} 
That $\frat(t)$ is continuous follows immediately from 
parts~\ref{itm:tobiaschi.intermediate} of Theorems~\ref{thm.tobiaschi} and~\ref{thm.tobiasomega}, which 
we have already established in sections~\ref{sec.clique-proof} and~\ref{sec.col-proof}.
Let us observe that, by definition~\eqref{eq:t0def} of $t_0$, 
it suffices to show that whenever $t>0$ is such that $\frat(t)>1$
then $\frat(t')>\frat(t)$ for all $t'>t$.

Consider a $t>t_0$ with $\frat(t)>1$. 
  First suppose that $\sup_{\varphi\in{\cal F}}\xi(\varphi,t) = \volhalfBdelta$. 
  Notice that the lower bound in Lemma~\ref{cor:chilb} then shows
  that also $\sup_{\varphi\in{\cal F}}\xi(\varphi,t') = \volhalfBdelta$
  for all $t'\geq t$, so that in this case $\frat(t') > \frat(t)$ for all $t'>t$ as
  $\xi(\varphi_0,t)$ is strictly decreasing in $t$.

So we may assume that $\sup_{\varphi\in{\cal F}}\xi(\varphi,t) > \volhalfBdelta$.
  By Lemma~\ref{lem:supismax} there is a $\varphi \in {\cal F}$ with $0<\int\varphi<\infty$ 
such that the supremum equals $\xi(\varphi,t)$.
Next, we claim that would suffice to prove that
for any $\lambda > 1$ there is at most one
$t' > 0$ that solves the equation $\xi(\varphi,t') = \lambda\xi(\varphi_0,t')$.
This is because, since $\xi(\varphi,t_0)\leq\xi(\varphi_0,t_0)$ and 
$t' \mapsto \xi(\varphi,t')/\xi(\varphi_0,t')$ is continuous in $t'$, we would
then also get that
 
\[ 
\frat(t') \geq \frac{\xi(\varphi,t')}{\xi(\varphi_0,t')} > \frac{\xi(\varphi,t)}{\xi(\varphi_0,t)} = \frat(t), 
\]

\noindent
for all $t'>t$.

  %
  %
  
  Set $\psi := \lambda \varphi_0$ with $\lambda > 1$.
By Lemma~\ref{lem:xibasic}, part~\ref{itm:xiscalar},
$\xi(\psi, t) = \lambda\xi(\varphi_0, t)$.
So to prove the lemma, it suffices to show that the system of equations

\begin{equation}\label{eq:top}
\int_{\eR^d} \varphi(x) e^{w\varphi(x)}{\dd}x =
\int_{\eR^d} \psi(x) e^{s\psi(x)}{\dd}x,
\end{equation}

\begin{equation}\label{eq:bottom}
\int_{\eR^d} H(e^{w\varphi(x)}){\dd}x =
\int_{\eR^d} H(e^{s\psi(x)}){\dd}x
\end{equation}
  has at most one solution $(w,s)$ with $w,s > 0$.
  For $s \geq 0$ let $v(s)$ be the unique solution of~\eqref{eq:top} and
  let $u(s)$ be the unique non-negative solution of~\eqref{eq:bottom}.
  Let us write
  
  \[ 
  F(s) := v(s) - u(s).
  \]
  
  \noindent
  Our goal for the remainder of the proof will be to show that there is 
  at most one solution $s > 0$ of $F(s)=0$, which implies that 
there is at most one solution $(w,s)$ with $w,s > 0$ of the
system given by~\eqref{eq:top} and~\eqref{eq:bottom} and thus also implies the lemma. 
 
  Differentiating both sides of equation \eqref{eq:top} wrt~$s$ we get
\[
  \int_{\eR^d} v'(s) \varphi^2(x) e^{v(s)\varphi(x)}{\dd}x =
  \int_{\eR^d} \psi^2(x) e^{s\psi(x)}{\dd}x,
\]
  where we have swapped integration wrt $x$ and differentiation wrt $s$.
  (This can be justified using the fundamental theorem of calculus and
  Fubini's theorem for nonnegative functions; and that $v$ is differentiable
can be justified using the implicit function theorem -- for more details see the footnotes 
in section~\ref{sec.scaling-function}.)
This gives

\begin{equation}\label{eq:vprimedef}
v'(s) =\frac{\int_{\eR^d} \psi^2(x) e^{s\psi(x)}{\dd}x}{\int_{\eR^d} \varphi^2(x) e^{v(s)\varphi(x)}{\dd}x}.
\end{equation}

\noindent
Similarly, differentiating~\eqref{eq:bottom} wrt~$s$ we get that for $s>0$: 

\begin{equation}\label{eq:uprimedef} 
u'(s)
=  
\frac{
s \int_{\eR^d} \psi^2(x) e^{s\psi(x)}{\dd}x
}{
u(s) \int_{\eR^d} \varphi^2(x) e^{u(s)\varphi(x)}{\dd}x 
}.
\end{equation}

\noindent
Observe that from~\eqref{eq:vprimedef}
we get
 
\begin{equation}\label{eq:vprimebiggerone}
v'(s) =
\lambda\frac{\int_{\eR^d} \psi(x) e^{s\psi(x)}{\dd}x}{\int_{\eR^d} \varphi^2(x) e^{v(s)\varphi(x)}{\dd}x}
  \geq 
  \lambda \frac{\int_{\eR^d} \psi(x) e^{s\psi(x)}{\dd}x}{\int_{\eR^d} \varphi(x) e^{v(s)\varphi(x)}{\dd}x}
  =\lambda, 
\end{equation}
  where we have used the specific form of $\psi$ as
  a constant times an indicator function, the fact that $\varphi^2(x) \leq \varphi(x)$
  (as $0\leq \varphi(x) \leq 1$ for all $x$) and the fact that $v(s)$ solves~\eqref{eq:top}.
  
\noindent
From \eqref{eq:vprimedef}
and~\eqref{eq:uprimedef} we see that: 

\begin{equation}\label{eq:obs} 
\text{ If  } u(s)=v(s) \text{  for some  } s > 0 \text{ then  }
u'(s) =  \frac{s}{v(s)} v'(s).
\end{equation}

\noindent
Let us first suppose that $v(0) \geq 0$.  Since $v'(s) \geq \lambda > 1$ for all $s>0$,
  it follows that $v(s) > s$ for all $s > 0$.
  Thus, from~\eqref{eq:obs} we see that whenever $v(s)=u(s)$ we have $u'(s)<v'(s)$.
  In other words, at every zero of $F$ we have $F'(s) =v'(s)-u'(s)>0$.
  This shows $F$ can have at most one zero, as required.

It remains to consider the case when $v(0) < 0$.
  We may assume that there is a solution $s>0$ to $v(s) = u(s)$, and since $F$ is
  continuous there is a least such solution $s_1$.
  Now $u(0)=0$ so $F(0)<0$, and thus $F(s)<0$ for each $0\leq s<s_1$.
  Hence $F'(s_1) = v'(s_1) - u'(s_1) \geq 0$.
From~\eqref{eq:obs} we now see that $v(s_1) \geq s_1$, so that using~\eqref{eq:vprimebiggerone}
we have $v(s) > s$ for all $s>s_1$.
Reasoning as before this gives that:

\begin{equation}\label{eq:obs2}
\text{ If } F(s_2) = 0 \text{ for some } s_2>s_1 \text{ then } F'(s_2) > 0.
\end{equation}

\noindent
Hence, if $F(s)>0$ for some $s>s_1$ then $F(s')>0$ for all $s' \geq s$:
for if not then there is a least $s'>s$ such that $F(s')=0$.
Such an $s'$ must then satisfy $F'(s') \leq 0$, contradicting~\eqref{eq:obs2}.

Now suppose there is an $s_2 > s_1$ such that $F(s_2)=0$.
By the above we must then have $F(s) \leq 0$ for all $s_1 < s < s_2$.
Hence, for $s_1< s < s_2$ we have
  
\[
  \frac{v'(s)}{u'(s)} = 
\frac{u(s)}{s} 
\frac{
\int_{\eR^d} \varphi^2(x)e^{u(s)\varphi(x)}  {\dd}x
}{
\int_{\eR^d} \varphi^2(x) e^{v(s)\varphi(x)}{\dd}x
}
\geq 
\frac{v(s)}{s} >1,
\]

\noindent
(using that $u(s) \geq v(s)$ and $v(s)>s$).
  Thus $F'(s)>0$ for all $s_1<s<s_2$.
But this implies $F(s_2)>F(s_1)=0$, contradicting the choice of $s_2$.

It follows that $s_1$ is the only zero of $F$, which concludes the proof of the lemma.
\end{proof}

\noindent
Lemmas~\ref{lem:t0pos} and~\ref{prop:xincr} give Theorem~\ref{thm.frat}.


\section{Remaining proofs}
\label{sec.rest-col-proofs}

  %
There are still some loose ends. 
Here we will finish the proof of Theorems~\ref{thm.tobiaschi},~\ref{thm.tobiasomega},~\ref{thm.tobiasfrat} and~\ref{thm.chif} and
Proposition~\ref{thm.verysparse}.


\subsection{Proof of parts~\ref{itm:tobiaschi.sparse} of Theorems~\ref{thm.tobiaschi}
and~\ref{thm.tobiasomega}}
\label{subsec.sparse-proof}

The following Lemma will immediately imply parts~\ref{itm:tobiaschi.sparse} of Theorems~\ref{thm.tobiaschi} and~\ref{thm.tobiasomega}; and it is also convenient for the proof of Theorem~\ref{thm.tobiasfrat} later on.
Recall that the notation $A_n$ a.a.a.s.~means that $\Pee( A_n \text{ holds for all but finitely many } n ) = 1$.

\begin{lemma}\label{lem:sparsereal}
For every $\eps>0$, there exists a $\beta=\beta(\sigma, \eps)$
such that if $n^{-\beta} \leq nr^d \leq \beta \ln n$ for
all sufficiently large $n$, then

\[ (1-\eps)k(n) \leq \omega(G_n), \chi(G_n) \leq (1+\eps)k(n) \text{ a.a.a.s., } \]

\noindent
where $k(n) := \ln n / \ln{\Big(}\frac{\ln n}{nr^d}{\Big)}$.
\end{lemma}

\begin{proof}
Set $W_1 := B(0;\frac12)$ and $W_2 := \clo(B)$, and observe that 

\begin{equation}\label{eq:MWomegachi} 
 M_{W_1} \leq \omega(G_n) \leq \chi(G_n) \leq \Delta(G_n)+1 \leq M_{W_2}, 
\end{equation}

\noindent
where $\Delta$ denotes the maximum degree.
Now set $\beta := \min( \beta(W_1,\sigma,\eps), \beta(W_2,\sigma,\eps) )$ with
$\beta(W,\sigma,\eps)$ as in Lemma~\ref{lem:scanstatisticsparse}. 
The statement immediately follows from Lemma~\ref{lem:scanstatisticsparse}
together with~\eqref{eq:MWomegachi}.
\end{proof}

\subsection{Proof of parts~\ref{itm:tobiaschi.vsparse} of Theorems~\ref{thm.tobiaschi}
and~\ref{thm.tobiasomega} and of Proposition~\ref{thm.verysparse}}
\label{subsec.vsparse-proof}

  Our plan is to use Lemma~\ref{lem:scanstatisticverysparse} on the generalised
  scan statistic.  To make use of this lemma, we `split $r$ into parallel sequences'.
  Let $K \in \eN$ be such that $\frac{1}{K} < \alpha$.
  For $k=0,1, \dots, K$ set $a_k := \frac{1}{k+\frac12}$
  and set:

\[ 
  r_k(n) := \left\{ \begin{array}{cl} r(n) & \textrm{if } n^{-a_{k-1}} \leq nr^d < n^{-a_k}, \\
  n^{-(a_{k-1}+1)/d} & \textrm{otherwise} \end{array} \right.,
\]

\noindent
for $1 \leq k \leq K$ and for $k=0$ set: 

\[ 
  r_0(n) := \left\{ \begin{array}{cl} r(n) & \textrm{if } nr^d < n^{-a_0}, \\
        n^{-4/d} & \textrm{otherwise} \end{array} \right..
\]

\noindent
Observe that $n^{-\frac{1}{k-\frac12}} \leq n r_k^d < n^{-\frac{1}{k+\frac12}}$ for $k=1,\dots,K$ and
$n r_0^d < n^{-2}$.
Hence

\begin{equation}\label{eq:kndef}
r(n) = r_k(n) \quad \text{ if and only if } \quad k = \kvsparse.
\end{equation}

\noindent
Let us now put $G_n^{(k)} := G(X_1,\dots,X_n ; r_k(n) )$ for $1\leq k\leq K$.
 Because $G_n$ coincides with some $G_n^{(k)}$ for each $n$, and the intersection of finitely 
many events of probability one has probability one itself, it suffices to show that
$\chi(G_n^{(k)}) = \omega(G_n^{(k)}) \in \{k,k+1\}$ a.a.a.s.~for  each $k$ separately.

Let us thus suppose that there is a fixed $0\leq k\leq K$ such that $r(n) = r_k(n)$ for all $n$.
Set $W_1 := B(0;\frac12)$ and $W_2 := B(0;100)$. Then we again have

\begin{equation}\label{eq:scanomegachi}
M_{W_1} \leq \omega(G_n) \leq \chi(G_n) \leq \Delta(G_n) + 1 \leq M_{W_2}. 
\end{equation}

\noindent
(The reason for choosing radius larger than 1 will become clear shortly.)
  Now notice that part~\ref{itm:scanverysparseub} of Lemma~\ref{lem:scanstatisticverysparse} 
shows that a.a.a.s.:

\begin{equation}\label{eq:vsparseM1M2conc1}
  M_{W_2} \leq k+1.
\end{equation}

\noindent
If $k \in \{0,1\}$ then $M_{W_1} \geq k$ holds trivially, and 
for $k \geq 2$ we can apply part~\ref{itm:scanverysparselb} of Lemma~\ref{lem:scanstatisticverysparse} to
show that a.a.a.s.:

\begin{equation}\label{eq:vsparseM1M2conc2}
  M_{W_1} \geq k.
\end{equation}

\noindent
Putting~\eqref{eq:scanomegachi},~\eqref{eq:vsparseM1M2conc1} and~\eqref{eq:vsparseM1M2conc2} together, 
we see that we have just proved parts~\ref{itm:tobiaschi.vsparse} of Theorems~\ref{thm.tobiaschi} 
and~\ref{thm.tobiasomega}.

To finish the proof we will derive (deterministically) that if~\eqref{eq:vsparseM1M2conc1} and~\eqref{eq:vsparseM1M2conc2} hold then $\chi(G_n)=\omega(G_n)$ must also hold. So let us assume that~\eqref{eq:vsparseM1M2conc1} and~\eqref{eq:vsparseM1M2conc2}  hold.  First note that  $\Delta(G_n) \in \{ \omega(G_n)-1, \omega(G_n)\}$.  If $\Delta(G_n)=\omega(G_n)-1$ then we are done (since always $\chi(G) \leq \Delta(G)+1$), so let us suppose that $\Delta(G_n)=\omega(G_n)$.  In this case Brooks' Theorem (see for example van Lint and Wilson~\cite{vanlintwilson}) tells us that $\chi(G_n) = \omega(G_n)$ unless $\omega(G_n)=2$ and $G_n$ contains an odd cycle of length at least 5.  
Let us therefore assume that $\omega(G_n)=2$. Then we must have $k\leq 2$ and hence $M_{W_2}\leq 3$.  
But now each component of $G_n$ must have at most 3 vertices: to see this note that if the subgraph induced by $X_{i_1}, X_{i_2}, X_{i_3}, X_{i_4}$ is connected then all four points are contained in a ball of radius $<100r$.  
Hence there are no odd cycles of length at least 5 and so $\chi(G_n) = 2$ as required.

\subsection{Proof of Theorem~\ref{thm.frat}}
\label{subsec.fratproof}

Let $r$ be any sequence of positive numbers that tends to 0, and let $\eps > 0$ be arbitrary, but fixed.
Our aim will be to prove

\begin{equation}\label{eq:fraclimaaas} 
1 - \eps \leq \frac{\chi(G_n)}{\frat(\frac{\sigma n r^d}{\ln n}) \omega(G_n)} < 1+\eps \text{ a.a.a.s.},
\end{equation}

\noindent
which clearly implies the theorem.

Set $\eps' := \frac{1+\eps}{1-\eps}-1$, and 
let $0<\beta \leq \beta(\sigma, \eps')$ where $\beta(\sigma, \eps')$ is as in Lemma~\ref{lem:sparsereal} above.
We can assume without loss of generality that $\beta < t_0 / \sigma$ with $t_0$ as in Theorem~\ref{thm.frat} if $\delta < 1$.
(If $\delta = 1$ then we do not need any additional assumption.)
We will apply the same trick we used in the previous section.
Set $r_1(n) := \min( r(n), n^{-(\beta+1)/d} )$,

\[ r_2(n) := \left\{ \begin{array}{cl}
                      r(n) & \text{ if } n^{-\beta } < nr^d < \beta \ln n, \\
		      \left(\frac{\beta\ln n}{n}\right)^{\frac{1}{d}} & \text{ otherwise. }
                     \end{array}
\right.,
\]

\noindent
and $r_3(n) := \max( \left(\frac{\beta \ln n}{n}\right)^{\frac{1}{d}}, r(n) )$.
Now set $G_n^{(k)} := G(X_1,\dots,X_n ; r_k(n) )$ for $k=1,2,3$.
It again suffices to prove~\eqref{eq:fraclimaaas} for each $G_n^{(k)}$ separately.

Notice that $\frac{\sigma n r_1^d}{\ln n} \to 0$ as $n\to\infty$, so that $\frat(\frac{\sigma n r_1^d}{\ln n})\to 1$.
The statement~\eqref{eq:fraclimaaas} for $G_n^{(1)}$ now follows immediately from Proposition~\ref{thm.verysparse}.

Since $\liminf_n \frac{\sigma n r_3^d}{\ln n} > 0$, the statement~\eqref{eq:fraclimaaas} for $G_n^{(3)}$ follows 
immediately from Theorems~\ref{thm.tobiaschireal} and~\ref{thm.tobiasomegareal}.

Finally, notice that $\frac{\sigma n r_2^d}{\ln n} < t_0$ for all $n$, so that in fact $\frat(\frac{\sigma n r_2^d}{\ln n})=1$
for all $n$.
Applying Lemma~\ref{lem:sparsereal} we see that 

\[ 
1 \leq \frac{\chi(G_n^{(2)})}{\omega(G_n^{(2)})} < \frac{1+\eps'}{1-\eps'} = 1+\eps \text{ a.a.a.s. }
\]

\noindent
This concludes the proof of Theorem~\ref{thm.tobiasfrat}

\subsection{Proof of Theorem~\ref{thm.chif}}

When $\delta=1$ then the theorem follows directly from $\omega(G_n) \leq \chi_f(G_n) \leq \chi(G_n)$ together 
with part~\ref{itm:thm.frat.i} of Theorem~\ref{thm.frat} and Theorem~\ref{thm.tobiasfrat}.

Let us thus suppose that $\delta < 1$.
We can again ``split into subsequences'' as in sections~\ref{subsec.vsparse-proof} and~\ref{subsec.fratproof}, 
to see that it suffices to prove the theorem for the case when $\sigma n r^d \leq t_0 \ln n$ 
and the case when $\sigma n r^d > t_0 \ln n$.

If $\sigma n r^d \leq t_0 \ln n$ then $\frat(\frac{\sigma n r^d}{\ln n}) = 1$ so that the argument for 
the $\delta=1$ case applies.

Finally, if $\sigma n r^d > t_0 \ln n$ then in particular $\liminf \frac{n r^d}{\ln n} > 0$; and now we 
may use~\eqref{target.chi-bdd} to complete the proof.

\section{Concluding remarks}
\label{sec.concl}

  In this paper we have proved a number of almost sure convergence
  results on the chromatic number of the random geometric graph and
  we have investigated its relation to the clique number.
  Amongst other things we have set out to describe the ``phase change'' regime when $nr^d=\Theta(\ln n)$.
  An important shift in the behaviour of the chromatic number occurs in this range of $r$
  (except in the less interesting case when the packing constant $\delta=1$).
  We have seen that (except when $\delta=1$) there exists a finite positive constant $t_0$ such that
  if $\sigma nr^d \leq t_0\ln n$ then the chromatic number and the clique number
  of the random geometric graph are essentially equal in the sense that
\[
  \frac{\chi(G_n)}{\omega(G_n)} \to 1 \text{ a.s. };
\]
  and if on the other hand $\sigma nr^d \geq (t_0+\eps) \ln n$ for
  some fixed (but arbitrarily small) $\eps > 0$ then
  the $\liminf$ of this ratio is bounded away from 1 almost surely.
  Moreover, if $nr^d\gg \ln n$ then
\[
  \frac{\chi(G_n)}{\omega(G_n)} \to \frac{1}{\delta} \text{ a.s. }
\]
  where $\delta$ is the packing constant.

  We have also given expressions for the almost sure limit $\fcol(t)$ of
  $\frac{\chi(G_n)}{\sigma nr^d}$ and the almost sure limit
  $\frat(t)$ of $\frac{\chi(G_n)}{\omega(G_n)}$ if $\sigma nr^d\sim t\ln n$ for some $t>0$.
  An interesting observation is that $t_0$ and the limiting functions $\fcli(t)$, $\fcol(t)$ and $\frat(t)$
  do not depend on the choice of probability measure $\nu$, and that the only feature of the
  probability measure that plays any role in the proofs and
  results in this paper is the maximum density $\sigma$.

  It should be mentioned that considering the ratio $\frac{\chi(G_n)}{\omega(G_n)}$,
  apart from the fact that it provides an easy to state summary of the results,
  can also be motivated by the fact that while colouring unit disk graphs
  (non-random geometric graphs when $d=2$ and $\norm{.}$ is the Euclidean norm)
  is NP-hard (Clark et al~\cite{clarkcolbourn}, Gr\"af et al~\cite{grafstumpfweissenfels}),
  their clique number may be found in polynomial time~\cite{clarkcolbourn}, unlike
  finding the clique number in general graphs.
  In fact the clique number of a unit disk graph may be found in polynomial time even if
  an embedding (that is, an explicit representation with points on the plane) is not given
  (Raghavan and Spinrad~\cite{RaghavanSpinrad03}).
  Thus, the results given here suggest that even though finding the chromatic number of a unit disk graph $G$ is NP-hard, for graphs that are not very sparse the polynomial approximation of finding the clique number and multiplying 
this by $\frac1\delta$
  might work quite well in practice. Note that for the Euclidean norm in the plane
  $\frac1\delta = \frac{2\sqrt{3}}{\pi} \approx 1.103$: also
  in this case always $\chi(G)/\omega(G) < 3$ (Peeters~\cite{Pee91}).
  \m{fuller version of \cite{Pee91}?}

  It is instructive to consider for comparison the ratio of chromatic number
  to clique number in the Erd\H{o}s-R\'enyi model, with expected degree similar to the values
  we have been investigating.  Let us consider $p=p(n)$ 
  such that $np \to \infty$ with $np = o(n^{\frac13})$ as $n \to \infty$.
  Then $\omega(G(n,p))=3$ whp (see for example Bollob\'as~\cite{BollobasBoek} Theorem 4.13),
  and $\chi(G(n,p)) \sim \frac{np}{2 \ln np}$ whp
  ({\L}uczak~\cite{luczak91a}, or see~\cite{BollobasBoek} Theorem 11.29);
  and thus whp
\[
  \frac{\chi(G(n,p))}{\omega(G(n,p))} \sim
  \frac{ np} {6 \ln np} \; \to \infty \text{ whp }.
\]
 Also, when does $\chi(G(n,p)) = \omega(G(n,p))$ whp?
 This property holds if $np \to 0$ as $n \to \infty$,
 since $G(n,p)$ is then a forest whp; and conversely. if the property holds
 (and $p$ is bounded below one) then $np \to 0$.
  Thus, the results on the chromatic number given here, and in the earlier work of the
  first author~\cite{cmcdplane} and Penrose~\cite{penroseboek},
  highlight a dramatic difference between the Erd\H{o}s-R\'enyi model on the one hand and
  the random geometric model on the other hand.

  Although we have presented substantial progress on the current
  state of knowledge on the chromatic number of random geometric
  graphs in this paper, several questions remain. Our proofs for
  instance do not yield an explicit expression for $t_0$ (when
  $\delta<1$) and it would certainly be of interest to find such an
  expression or to give some (numerical) procedure to determine it,
  in particular for the euclidean norm in $\eR^2$.
  More generally, it is far from trivial to extract information from
  the expressions for $\fcol(t)$ and $\frat(t)$.
In particular we would be interested to know whether $\fcol$ is 
differentiable. In particular: is $\fcol$ differentiable at $t_0$?
Also Lemma~\ref{lem:supismax} suggests the question:
Is $\fcol(t) > \volhalfBdelta$ for all $0<t<\infty$?

Another natural question concerns the extent to which the chromatic number is `local' or `global'.  Define the random variable $R_n$ to be the infimum of the values $R>0$ such that $G_n$ contains a subgraph $H_n$ which is induced by the points in some ball of radius $R$ and which satisfies $\chi(H_n)=\chi(G_n)$.  In the very sparse case when $nr^d \leq n^{-\alpha}$ for some fixed $\alpha>0$, the proof of Proposition 1.6 shows that $R_n \leq \eps r$ a.a.a.s.~for any fixed $\eps>0$, so 
the behaviour is `very local'. Now suppose that $\sigma nr^d\sim t\ln n$ for some $t>0$.  
Then by Lemmas~\ref{lem:WW} and~\ref{lem:island}, for $t$ sufficiently small we have 
$R \leq 2r$ a.a.a.s.~so still the behaviour is local.  
But what happens for large $t$?

  A question that has not been addressed in this paper (except
in the ``very sparse'' case when $nr^d$ is bounded by a negative power of $n$) is the
probability distribution of $\chi(G_n)$.
In a recent paper by the second author~\cite{twopoint} it was
shown that whenever $nr^d\ll\ln n$ then $\chi(G_n)$ is two-point
concentrated, in the sense that

\[ \Pee( \chi(G_n)\in\{k(n),k(n)+1\} ) \to 1, \]

\noindent
as $n\to\infty$ for some sequence $k(n)$.
Analogous results were also shown to hold for the clique number, the maximum degree and the degeneracy of $G_n$.  
For other choices of $r$ the distributions of these random variables are not known.
However, it is possible to extend an argument in Penrose~\cite{penroseboek} to show that
if $\nu$ is the uniform distribution on the hypercube and $\ln n \ll nr^2 \ll (\ln n)^d$ then there 
are $(a_n)_n$ and $(b_n)_n$ 
such that $(\Delta(G_n)-a_n)/b_n$ tends to a Gumbel distribution, and if $nr^d = \Theta(\ln n)$ then $\Delta(G_n)$ is not 
finitely concentrated, but the  distribution does not look like any of the standard probability 
distributions.


\section{Acknowledgments}
  The authors would like to thank Joel Spencer for helpful discussions and e-mail correspondence related to the paper.
  We would also like to thank Gregory McColm, Mathew Penrose, Alex Scott, Miklos Simonovits and Dominic Welsh for helpful
  discussions related to the paper, and Karolyi B\"or\"oczky Jr. both for helpful discussions and for allowing us to 
present Lemma~\ref{prop:karcsi}  here.
  Finally we would like to thank a most meticulous and helpful referee,
  whose extensive comments led to a much improved presentation of our results.

\bibliographystyle{plain}
\bibliography{References}


\appendix

\section{Proof of Lemma~\ref{prop:karcsi} of B\"or\"oczky} \label{sec:karcsi}

  Lemma~\ref{prop:karcsi} above is due to K.~B\"or\"oczky Jr.~and with his kind permission
  we give a proof here,
  as it is not readily available from other sources.
  \smallskip

\begin{pf}{of Lemma~\ref{prop:karcsi}}
  Let $I := \cap_{a\in A} (a+C)$. Then $I$ is compact and convex.
  We may suppose wlog~that $\vol(I)>0$ (as otherwise there is nothing to prove).
  Let us remark that
\[
  I + \clo(-A) \subseteq C.
\]
  This is because for any $x\in I$ and $a\in A$ there exists a
  $c\in C$ such that $x=c+a$, by definition of $I$.
  Thus, for any $x\in I, a\in A$ we have $x-a\in C$:
  in other words $I+(-A)\subseteq C$.
  This also gives that $I+\clo(-A)=\clo(I+(-A)) \subseteq C$
  as $C$ is closed.

  We now use the Brunn-Minkowski inequality (see chapter 12 of Matousek~\cite{matousekboek} for a very readable proof).
  This states that, if $A,B\subseteq\eR^d$ are nonempty and compact, then
  $\vol(A+B) \geq \left( \vol(A)^{\frac{1}{d}} + \vol(B)^{\frac{1}{d}} \right)^d$.
%
  By this inequality
\[
  \vol(I)^{\frac1d}+\vol(\clo(-A))^\frac1d \leq
  \vol(I+\clo(-A))^\frac1d \leq \vol(C)^\frac1d.
\]
  Thus
\begin{equation}\label{eq:karcsieq1}
  \vol(I) \leq \left(\vol(C)^\frac1d-\vol(A)^\frac1d\right)^d.
\end{equation}
  The lemma will now follow if we show that equality holds in~\eqref{eq:karcsieq1} when
  $A$ is of the form $A=\lambda(-C)$ for some $\lambda>0$.
  Let us then suppose that $A = \lambda(-C)$ for some $0\leq \lambda<1$
  (note that $\lambda\geq 1$ would contradict $\vol(I)>0$).
  
  We claim that $(1-\lambda)C \subseteq I$.
  Let $x\in(1-\lambda)C$, and let $a\in A = \lambda(-C)$ be arbitrary.
  We can write $x=(1-\lambda)c_1$ and  $a=-\lambda c_2$ for some $c_1,c_2\in C$.
  Because $C$ is convex, $c_3 := (1-\lambda)c_1+\lambda c_2\in C$ and
  thus $x = a+c_3 \in (a+C)$. As $a\in A$ was arbitrary this gives $x\in I$.
  Thus indeed $(1-\lambda)C \subseteq I$, as claimed.  But now
  $\vol(I) \geq (1-\lambda)^d\vol(C)$ and $\vol(A)=\lambda^d\vol(C)$);
  and so equality holds in~\eqref{eq:karcsieq1}, and we are done.
\end{pf}


\end{document}

%% file: Wp2.pstex_t
\begin{picture}(0,0)%
\includegraphics{Wp2.pstex}%
\end{picture}%
\setlength{\unitlength}{2960sp}%
\begingroup\makeatletter\ifx\SetFigFont\undefined%
\gdef\SetFigFont#1#2#3#4#5{%
  \reset@font\fontsize{#1}{#2pt}%
  \fontfamily{#3}\fontseries{#4}\fontshape{#5}%
  \selectfont}%
\fi\endgroup%
\begin{picture}(5574,5124)(4339,-8023)
\put(5226,-3961){\makebox(0,0)[lb]{\smash{{\SetFigFont{11}{13.2}{\familydefault}{\mddefault}{\updefault}{\color[rgb]{0,0,0}$\eps$}%
}}}}
\put(9101,-7561){\makebox(0,0)[lb]{\smash{{\SetFigFont{11}{13.2}{\rmdefault}{\mddefault}{\updefault}{\color[rgb]{0,0,0}$L\eps$}%
}}}}
\put(6901,-7561){\makebox(0,0)[lb]{\smash{{\SetFigFont{11}{13.2}{\rmdefault}{\mddefault}{\updefault}{\color[rgb]{0,0,0}$2K\eps$}%
}}}}
\end{picture}%

%% file: rgg.bbl
\begin{thebibliography}{10}

\bibitem{bollobas88}
B.~Bollob{\'a}s.
\newblock The chromatic number of random graphs.
\newblock {\em Combinatorica}, 8(1):49--55, 1988.

\bibitem{chvatal}
V.~Chv\'atal.
\newblock {\em Linear Programming}.
\newblock W. H. Freedman and Company, New York, 1983.

\bibitem{clarkcolbourn}
B.~N. Clark, C.~J. Colbourn, and D.~S. Johnson.
\newblock Unit disk graphs.
\newblock {\em Discrete Math.}, 86(1-3):165--177, 1990.

\bibitem{glaznaus01}
J.~Glaz, J.~Naus, and S.~Wallenstein.
\newblock {\em Scan Statistics}.
\newblock Springer, New York, 2001.

\bibitem{grafstumpfweissenfels}
A.~Gr{\"a}f, M.~Stumpf, and G.~Wei{\ss}enfels.
\newblock On coloring unit disk graphs.
\newblock {\em Algorithmica}, 20(3):277--293, 1998.

\bibitem{gruberwills93}
P.~M. Gruber and J.~M. Wills.
\newblock {\em Handbook of Convex Geometry}.
\newblock North-Holland, Amsterdam, 1993.

\bibitem{randomgraphs}
S.~Janson, T.~{\L}uczak, and A.~Rucinski.
\newblock {\em Random graphs}.
\newblock Wiley-Interscience Series in Discrete Mathematics and Optimization.
  Wiley-Interscience, New York, 2000.

\bibitem{kingmanboek}
J.~Kingman.
\newblock {\em Poisson Processes}.
\newblock Oxford University Press, Oxford, 1993.

\bibitem{luczak91a}
T.~{\L}uczak.
\newblock The chromatic number of random graphs.
\newblock {\em Combinatorica}, 11(1):45--54, 1991.

\bibitem{mallows68}
C.~L. Mallows.
\newblock An inequality involving multinomial probabilities.
\newblock {\em Biometrika}, 55(2):422--424, 1968.

\bibitem{matousekboek}
J.~Matou{\v{s}}ek.
\newblock {\em Lectures on Discrete Geometry}, volume 212 of {\em Graduate
  Texts in Mathematics}.
\newblock Springer-Verlag, New York, 2002.

\bibitem{cmcdplane}
C.~J.~H. McDiarmid.
\newblock Random channel assignment in the plane.
\newblock {\em Random Structures and Algorithms}, 22(2):187--212, 2003.

\bibitem{twopoint}
T.~M\"uller.
\newblock Two-point concentration in random geometric graphs.
\newblock {\em Combinatorica}, 28(5):529--545, 2008.

\bibitem{pachboek}
J.~Pach and P.~K. Agarwal.
\newblock {\em Combinatorial Geometry}.
\newblock Wiley-Interscience Series in Discrete Mathematics and Optimization.
  John Wiley \& Sons Inc., New York, 1995.
\newblock A Wiley-Interscience Publication.

\bibitem{Pee91}
R.~Peeters.
\newblock On coloring $j$-unit sphere graphs.
\newblock Technical Report {FEW} 512, Economics Department, Tilburg University,
  1991.

\bibitem{penroseboek}
M.~D. Penrose.
\newblock {\em Random Geometric Graphs}.
\newblock Oxford University Press, Oxford, 2003.

\bibitem{RaghavanSpinrad03}
V.~Raghavan and J.~Spinrad.
\newblock Robust algorithms for restricted domains.
\newblock {\em J. Algorithms}, 48(1):160--172, 2003.
\newblock Twelfth Annual ACM-SIAM Symposium on Discrete Algorithms (Washington,
  DC, 2001).

\bibitem{rogersboek}
C.~A. Rogers.
\newblock {\em Packing and Covering}.
\newblock Cambridge Tracts in Mathematics and Mathematical Physics, No. 54.
  Cambridge University Press, New York, 1964.

\bibitem{fractionalgraphtheory}
E.~R. Scheinerman and D.~H. Ullman.
\newblock {\em Fractional Graph Theory}.
\newblock Wiley-Interscience Series in Discrete Mathematics and Optimization.
  John Wiley \& Sons Inc., New York, 1997.

\bibitem{vanlintwilson}
J.~H. van Lint and R.~M. Wilson.
\newblock {\em A Course in Combinatorics}.
\newblock Cambridge University Press, Cambridge, second edition, 2001.

\end{thebibliography}
